\documentclass[10pt,reqno]{amsart}
\usepackage[letterpaper, hmargin=1in, top=1in, bottom=1.2in, footskip=0.6in]{geometry}
\RequirePackage[OT1]{fontenc}
\RequirePackage{amsthm,amsmath}
\RequirePackage[numbers]{natbib}
\RequirePackage[colorlinks,linkcolor=darkblue,citecolor=darkred,urlcolor=darkblue]{hyperref}
\usepackage{amssymb}
\usepackage{anysize}
\usepackage{hyperref}
\usepackage{extarrows}
\usepackage{multirow}
\usepackage{natbib}
\usepackage{stmaryrd}
\usepackage{color}
\usepackage{amsmath}
\usepackage{enumitem}

 \definecolor{darkred}{rgb}{0.9,0,0.3}
 \definecolor{darkblue}{rgb}{0,0.3,0.9}
 \definecolor{test}{rgb}{1,0,1}

\numberwithin{equation}{section}

\theoremstyle{plain} 
\newtheorem{thm}{Theorem}[section]
\newtheorem{lem}[thm]{Lemma}
\newtheorem{cor}[thm]{Corollary}
\newtheorem{pro}[thm]{Proposition}
\newtheorem{assumption}[thm]{Assumption}
\newtheorem{defn}[thm]{Definition}
\newtheorem{example}[thm]{Example}
\newtheorem*{example*}{Example}

\newtheorem{rem}[thm]{Remark}

\renewcommand{\Im}{\mathrm{Im}\,}
\newcommand{\re}{\mathrm{Re}\,}
\newcommand{\im}{\mathrm{Im}\,}

\newcommand{\E}{{\mathbb E }}
\newcommand{\R}{{\mathbb R }}
\newcommand{\N}{{\mathbb N}}

\newcommand{\C}{{\mathbb C}}

\newcommand{\ii}{\mathrm{i}}
\newcommand{\ee}{\mathrm{e}}
\newcommand{\deq}{\mathrel{\mathop:}=}
\newcommand{\eqd}{\mathrel={\mathop:}}

\newcommand{\dd}{\mathrm{d}}

\newcommand{\bs}{\boldsymbol}

\newcommand{\la}{\langle}
\newcommand{\ra}{\rangle}

\newcommand{\nc}{\normalcolor}

\renewcommand{\mathbf}[1]{\bs{#1}}
 
\usepackage{amsmath}
\usepackage{amsfonts}
\usepackage{latexsym}
\usepackage{amsthm}
\usepackage{amsxtra}
\usepackage{amscd}
\usepackage{bbm}
\usepackage{mathrsfs}
\usepackage{bm}
\usepackage{tensor}

\renewcommand{\leq}{\leqslant}
\renewcommand{\le}{\leqslant}
\renewcommand{\geq}{\geqslant}
\renewcommand{\ge}{\geqslant}
\renewcommand{\epsilon}{\varepsilon}

\renewcommand{\b}[1]{\boldsymbol{\mathrm{#1}}} 
\newcommand{\bb}{\mathbb} 
\newcommand{\cal}{\mathcal}

\newcommand{\ul}[1]{\underline{#1} \!\,} 
\newcommand{\ol}[1]{\overline{#1} \!\,} 
\newcommand{\wh}{\widehat}
\newcommand{\wt}{\widetilde}

\DeclareMathOperator{\tr}{Tr}
\DeclareMathOperator{\sym}{Sym}
\DeclareMathOperator{\var}{Var}
\DeclareMathOperator{\supp}{supp}

\calclayout
\begin{document}

\begin{center}
\large\bf
Quantitative CLT for  linear eigenvalue statistics of Wigner matrices 
\end{center}

\vspace{0.5cm}
\renewcommand{\thefootnote}{\fnsymbol{footnote}}
\hspace{5ex}	
\begin{center}
 \begin{minipage}[t]{0.35\textwidth}
\begin{center}
Zhigang Bao\footnotemark[1]  \\
\footnotesize {HKUST}\\
{\it mazgbao@ust.hk}
\end{center}
\end{minipage}
\hspace{8ex}
\begin{minipage}[t]{0.35\textwidth}
\begin{center}
Yukun He\footnotemark[2]  \\ 
\footnotesize {University of Z\"{u}rich}\\
{\it yukun.he@math.uzh.ch}
\end{center}
\end{minipage}
\end{center}

\footnotetext[1]{Supported by Hong Kong RGC  GRF 16300618,  GRF 16301519, GRF 16301520 and NSFC 11871425.}
\footnotetext[2]{Supported by ERC Advanced Grant “Correlations in Large Quantum Systems”, and UZH Forschungskredit grant FK-20-113.}

\vspace{0.8cm}
\begin{center}
	\begin{minipage}{0.8\textwidth}\footnotesize{
			In this article, we establish a near-optimal convergence rate  for the CLT of linear eigenvalue statistics of $N\times N$ Wigner matrices, in Kolmogorov-Smirnov distance. For all test functions $f\in C^5(\bb R) $, we show that  the convergence rate is  either $N^{-1/2+\varepsilon}$ or $N^{-1+\varepsilon}$,  depending on the first Chebyshev coefficient of $f$ and the third moment of the diagonal matrix entries. The condition that distinguishes these two rates is necessary and sufficient. For a general class of test functions, we further identify matching lower bounds for the convergence rates. In addition, we identify an explicit, non-universal contribution in the linear eigenvalue statistics, which is  responsible for the slow  rate  $N^{-1/2+\varepsilon}$ for non-Gaussian ensembles. By removing this non-universal part, we show that the shifted linear eigenvalue statistics have the unified  convergence rate  $N^{-1+\varepsilon}$ for all test functions.  
		}
	\end{minipage}
\end{center}

\thispagestyle{headings}

\vspace{2em}

\section{Introduction and main result}

In Random matrix Theory (RMT), there are various limiting laws about the fluctuations of eigenvalue statistics.  
However, most of these laws were derived in the limiting form without a quantitative description on the speed of the weak convergence.   In this paper, we will establish a near-optimal  convergence rate of the CLT for the linear eigenvalue statistics, 
in Kolmogorov-Smirnov distance, for a fundamental Hermitian random matrix model, Wigner matrix, whose definition is detailed below.  

\begin{defn}[Wigner matrix] \label{def:Wigner}
	Let $h_d$ be a real random variable, and $h_o$ be a complex random variable. They satisfy
	\[
	\bb E h_d=\bb E h_o=0\,,\quad	\bb E |h_o|^2=1 \quad \mbox{and} \quad \bb E |h_{o}|^p+\bb E |h_{d}|^p \leq C_p
	\]
	for all fixed $p \in \bb N_+$. We set
	\[
	a_2 \deq \bb Eh_d^2\,,\quad a_3\deq \bb E h_d^3 \quad \mbox{and}\quad  m_4 \deq \bb E|h_o|^4\,.
	\]
	A \textit{Wigner matrix} is a  Hermitian  matrix $H=(H_{ij})_{i,j=1}^N\in \bb C^{N\times N}$ with independent upper triangular entries $H_{ij}(1\leq i \leq j \leq N)$, and
	\[
	H_{ii}\overset{d}{=}h_d/\sqrt{N}\,, \quad H_{ij}\overset{d}{=}h_o/\sqrt{N}, \quad \forall i < j\,.
	\]
	We distinguish the \textit{real symmetric} case $(\beta=1)$, where $h_o\in \bb R$ and $H_{ij}=H_{ji}$, from the \textit{complex Hermitian} case $(\beta=2)$, where $h_o\in \mathbb{C}$, $\bb Eh_o^2=0$ and $H_{ji}=\overline{H_{ij}}$. Set $s_4= m_4+\beta-4$.
\end{defn}

For a test function $f:\mathbb{R}\to \mathbb{R}$, the linear eigenvalue statistic (LES) of $H$  is defined as 
\begin{align*}
	\tr f(H)=\sum_{i=1}^Nf(\lambda_i), 
\end{align*}
where $\lambda_1\geq \lambda_2\geq \ldots\geq\lambda_N$ are the ordered eigenvalues of $H$.

\subsection{Reference review on non-quantitative and quantitative  CLTs for LES} 
The CLT for LES of Wigner matrices is a classical result in RMT; see e.g.\,\cite{LP09,SW13, BaiWangZhou, BY05, C-D01, Chatterjee09,  KKP, KKP2, Shcherbina}. It states that for test functions $f$ satisfying certain regularity assumptions and $\var \tr f(H)>c>0$, we have
\begin{align} \label{1.1}
	\frac{\mathrm{Tr} f(H)-\mathbb{E}\mathrm{Tr} f(H)}{\sqrt{\mathrm{Var}(\mathrm{Tr} f(H))}}\stackrel{d}\longrightarrow \cal N(0,1)\,.
\end{align}
Differently from the classical CLT for sum of i.i.d. order 1 random variables,  a prominent feature of  CLT for LES is that $\text{Var}(\text{Tr} f(H))$ is order $1$ for sufficiently regular $f$. This is essentially due to the strong correlation among eigenvalues.  
Similar results have been obtained also for sample covariance matrices \cite{BS05, Jonsson}, deformed Wigner matrices \cite{JL20,DF19}, random band matrices \cite{AZ06, G02, Shcherbina15, LS13, JSS16}, heavy tailed matrices \cite{BGM14}, polynomials in random matrices \cite{MS06, MSS07, CMSS06}, random matrices on compact groups \cite{Johansson97,DS94, DE01}, Hermite beta ensembles \cite{Johansson98, DE06},   and also non-Hermitian random matrices \cite{RS06, CES19}. We also refer to the references in these papers for related study.

The convergence rate is a natural question following the CLT, which provides a quantitative description of the weak convergence. A quantitative CLT is especially important in applications, since in reality the random matrices often have a large but given size $N$, and the limiting laws may be achieved in  a rather slow rate so that it may deviate significantly from the law for the non-asymptotic system of size $N$.  In the context of RMT, most of the references mentioned above provide non-quantitative CLTs only. To the best of our knowledge, the first  few works in this direction is on the random matrices on compact groups \cite{DS94, Stein95, Johansson97}. Especially, in \cite{Johansson97}, a super-exponential rate of convergence $O(N^{-cN})$, in Kolmogorov-Smirnov distance,  was obtained for the circular unitary ensembles, and an exponential rate $O(e^{-cN})$ was obtained for the circular real and quaternion ensembles. Such fast rates are essentially due to the Gaussian nature of the circular ensembles. We also refer to \cite{DS11, JLambert20, CJ21} for related studies on circular ensembles. In contrast, the study of the convergence rate of LES for Hermitian ensembles, emerged only very recently.   In \cite{LLW19}, the authors considered the LES of the $\beta$-ensembles with one-cut potentials and established a convergence rate of CLT, in quadratic Kantorovich distance. In \cite{BB19}, the authors studied the convergence rate of LES for the models GUE/LUE/JUE, in Kolmogorov-Smirnov distance; in particular, they obtained a rate of $O(N^{-1/5})$ for GUE. Both the work \cite{LLW19} and \cite{BB19} considered the invariant ensembles and the results do not seem to be optimal in general. Furthermore, the approaches used in \cite{LLW19} and \cite{BB19} are more analytical  than probabilistic, which are both based on the explicit formulas for the joint probability density functions of the eigenvalues. 
In a recent work \cite{ECS20}, a CLT for the generalized linear statistics $\text{Tr} f(H)A$ with deterministic matrix $A$ was established for $f\in H_0^2(\mathbb{R})$, on macroscopic and mesoscopic scales. Particularly, on macroscopic scale, the result \cite{ECS20} indicates a $N^{-\frac12}$ convergence rate, in the sense of moment.
In this work, we establish a near-optimal rate for the CLT of LES in a much stronger distance, Kolmogorov-Smirnov distance, for  Wigner matrices, via a  probabilistic approach. 
Observe that, in two toy cases, $f(x)=x$ and $f(x)=x^2$, the statistic $\text{Tr} f(H)$ is simply the sum of $c_1N$ and $c_2N^2$ independent random variables, respectively. Therefore, according to the  Berry-Esseen bound for the classical CLT of sum of independent random variables with third moments, we can easily conclude that the convergence rate of LES is of order $N^{-1/2}$ or  $N^{-1}$ when $f(x)=x$ or $f(x)=x^2$, respectively. From these toy examples, we can raise the following questions

\textit{Question 1}: Do the rates $O(N^{-\frac12})$ and $O(N^{-1})$ also apply to general test functions? 

\textit{Question 2}: Are $O(N^{-\frac12})$ and $O(N^{-1})$ the best possible convergence rates, i.e., can one obtain matching lower bounds?

\textit{Question 3}: If the answers to the previous questions are positive, what is the necessary and sufficient condition for the rate to be $O(N^{-1})$?

This paper answers these three questions. The main results will be detailed in Section \ref{s.main result}, and the proof strategy and novelties will be stated in Section \ref{s. strategy}.

\subsection{Main result} \label{s.main result}
Our aim is to provide a  quantitative rate  for the convergence \eqref{1.1}, or its variant with $\mathbb{E}\text{Tr} f(H)$ and $\text{Var} (\text{Tr}f(H))$ replaced by their estimates. Recall  the Kolmogorov-Smirnov distance of two real random variables $X$ and $Y$
\[
\Delta(X,Y)\deq \sup_{x \in \bb R}  |\mathbb P(X\leq x)-\mathbb P(Y \leq x)|\,.
\]
We shall always use $ Z$ to denote the standard Gaussian random variable ${\cal N}(0,1)$.  For $k \in \bb N$ and a real test function $g$ which is integrable  w.r.t. to the weight function $\frac{1}{\sqrt{4-x^2}}$ on $[-2,2]$, we set 
\begin{align}
	c_k^g:= \frac{1}{\pi}\int_{-\pi}^{\pi} g(2\cos\theta) \cos k\theta {\rm d} \theta=\frac{2}{\pi} \int_{-1}^1 g(2x) T_k(x)\frac{{\rm d} x}{\sqrt{1-x^2}},  \label{def of ck}
\end{align}
where $T_k(x)= \cos(k \cos^{-1}x)$ is the k-th Chebyshev polynomial of the first kind. 
Hence, $c_k^g$ can be regarded as the $k$-th coefficient of the Fourier-Chebyshev expansion of $g(2x), x\in [-1,1]$. Further, let $\mu_f,\sigma^2_f $ be as in \eqref{mu_f}, \eqref{sigma_f} below. For any fixed $\gamma \in \bb R$, we define
\begin{align}
	\sigma_{f,\gamma}^2\deq \sigma^2_f+\frac{1}{4}a_{2}(\gamma-1)^2(c_1^f)^2 
	\quad \mbox{and} \quad
	\cal Z_{f,\gamma}\deq \frac{\tr f(H)-\mu_f-\frac12\gamma c_1^f \tr {H}}{\sigma_{f,\gamma}}\,. \label{def of Z_{f,gamma}}
\end{align}
We shall regard $\frac12 c_1^f \tr {H}$  and $\tr f(H)-\frac12 c_1^f \tr {H}$ as the diagonal part and the off-diagonal part  of $\tr f(H)$, respectively. In the shifted LES $Z_{f,\gamma}$, we subtract $\gamma$-portion of the diagonal part from $\text{Tr} f(H)$. In particular, when $\gamma=0$, we have the original LES, while when $\gamma=1$, we have the pure off-diagonal part. Our main finding is that the diagonal part and the off-diagonal have different convergence rates towards Gaussian in general, and thus it is expected that the convergence rate of  $Z_{f,\gamma}$ depends on $(1-\gamma)$.  The dependence is nevertheless more subtle in the sense two more  factors will determine the convergence rate together with $1-\gamma$.  To state our result, we first introduce the following notation. 
\begin{equation} \label{cal X}
	\cal X\equiv \cal X(\gamma, f, H)\deq \begin{cases}
		0 &\mbox{if}  \quad (1-\gamma)c_1^f \bb E h_d^3=0\\
		1 &\mbox{otherwise}.
	\end{cases} 
\end{equation}
Since the result for linear function $f(x)=ax+b$ follows from the classical Berry-Esseen bound directly, we exclude this trivial case from our discussion. We may now state our main theorem.

\begin{thm}  \label{thmmain}
	Let $f\in C^5(\bb R)$ be independent of $N$ and suppose that $f$ is not linear. For any fixed $\kappa>0$, there exists fixed $C_{f,\kappa}>0$ such that 
	\begin{align*}
		\Delta(\cal Z_{f,\gamma},Z)\leq C_{f,\kappa}\big(\cal X N^{-1/2+\kappa}+N^{-1+\kappa}\big)\,. 
	\end{align*}
\end{thm}

\begin{rem} (i) Theorem \ref{thmmain} provides a positive answer for Question 1. It shows that up to an $N^{\kappa}$ factor, the rates $O(N^{-\frac12})$ and $O(N^{-1})$ apply to general test functions.
	
	(ii) From the definition of $\cal X$ in \eqref{cal X}, the conditions to have the rate $O(N^{-1})$ is three-fold. First, the slow rate $O(N^{-1/2})$ comes from the  diagonal part $\frac12c_1^f \tr H$ of the LES. Once we fully subtract this term from $\tr f(H)$, i.e.\,when $\gamma=1$, the remaining part of the LES will have a unified $O(N^{-1+\kappa})$ convergence rate.
	
	Second, in case $\gamma\ne 1$ but the test function $f$ satisfies $c_1^f=0$, the rate is again $O(N^{-1+\kappa})$. Especially, it recovers the rate for $\text{Tr} f(H)$ in the toy case $f(x)=x^2$.
	
	Third, if $\gamma\ne 1$ and $c_1^f\neq 0$, the diagonal part $\frac12c_1^f \tr H$ will play a role in the LES.  The object $\tr {H}$, is simply a sum of i.i.d.\,random variable, and in general it has a slow convergence rate $O(N^{-1/2})$ towards Gaussian distribution. This is true even if $\mathbb{E}h_d^3=0$, as it is easy to check the case when $\mathbb{P}(h_d=1)=\mathbb{P}(h_d=-1)=\frac12$.  However, Theorem \ref{thmmain} shows that if $\mathbb{E}h_d^3=0$, the convergence rate of our shifted LES will still degenerate to $O(N^{-1+\kappa})$. This is due to the fact that the Gaussianity of the off-diagonal distribution $\text{Tr}f(H)-\frac12 c_1^f \text{Tr} H$ can further smooth out the difference between the distribution of $\frac12 c_1^f \text{Tr} H$ and Gaussian, as long as $\mathbb{E}h_d^3=0$. This does not happen in case $f(x)=ax+b$, due to the absence of the off-diagonal part.  
	
	(iii) As we shall see in Remark \ref{rmk1.6} below, the condition $\cal X=0$ for the rate to be $O(N^{-1})$ is necessary and sufficient for general test function $f$. This answers Question 3.
\end{rem}

In order to verify the optimality of our upper bound for the convergence rate. In the sequel, we present a companion result on the lower bound.  
Let us denote 
\[
\mathring{\cal Z}_{f,\gamma} \deq \frac{\tr f(H)-\frac12\gamma c_1^f \tr {H}-\bb E \tr f(H)}{\sqrt{\var( \tr f(H)-\frac12\gamma c_1^f \tr {H})}}\,.
\]
For the lower bound of  the convergence rate, we study the above  quantity with mean $0$ and variance $1$, instead of  $\cal Z_{f,\gamma}$ in (\ref{def of Z_{f,gamma}}). Otherwise, one needs to exclude the possibility that the bias of the centralization or the scaling  may be responsible for the lower bound of the convergence rate. Since our lower bound is mainly used  to confirm that our upper bound is near-optimal, we do not aim for the lightest assumptions for the matrix and the test function in this part. We will further make  the following additional assumptions in order to simplify the discussion.  

\begin{assumption} \label{122910} We make the following additional assumptions on $H$ and $f$:
	
	{\rm (i)  (on $H$)}  In addition to the basic assumptions in Definition \ref{def:Wigner}, we further assume that the diagonal entries and off-diagonal entries of $H$ match those of GOE ($\beta=1$) or GUE ($\beta=2$) up to the second and fourth moments respectively. 
	
	{\rm (ii) (on $f$)}  We assume that $f(x):\mathbb{R}\to \mathbb{R}$ is analytic in a neighborhood of $[-2-\varepsilon, 2+\varepsilon]$ for some fixed $\varepsilon>0$, and we further assume that $|f(x)|$ does not grow faster than polynomials when $|x|\to \infty$. 
\end{assumption}
We remark here that the assumption on the growth of $|f(x)|$ is made to ensure the existence of $\mathbb{E}\text{Tr}f(H)$.

For any test function $f$, we introduce the notation
$
	f_\gamma(x)\deq f(x)-\frac{1}{2}\gamma c_1^f x.
$
Furthermore, define 
$
r_1^{f_\gamma} :=\frac{1}{8}(c_1^{f_\gamma})^3\mathbb{E}h_d^3\,,
$
and
\begin{align*}
	&r_{2,\beta}^{f_\gamma} := \beta^{-2}\sum_{\alpha, \tau, \gamma,\sigma, \psi=0}^{\infty}(\psi+1)\Big( c_{\tau-\alpha+\sigma-\gamma+2\psi+2}^{f_\gamma} c_{\alpha+\tau+1}^{f_\gamma}c_{\gamma+\sigma+1}^{f_\gamma}+c_{\sigma-\gamma}^{f_\gamma} c_{\tau-\alpha}^{f_\gamma} c_{\alpha+\tau+\gamma+\sigma+2\psi+4}^{f_\gamma}\notag\\
	&\qquad\qquad\qquad\qquad-c_{\tau-\alpha}^{f_\gamma}c_{\gamma+\sigma+1}^{f_\gamma} c_{\alpha+\tau-\gamma+\sigma+2\psi+3}^{f_\gamma}\Big)+\frac{3}{8}c_2^{f_\gamma}(c_1^{f_\gamma})^2\mathcal{C}_4(h_d),
\end{align*}
where $\beta\in \{1,2\}$ represents the symmetric class, $c_k^g$ is defined in (\ref{def of ck}) and here we also set $c_{-k}^g=c_k^g$ for $k\in \mathbb{N}$. 

\begin{thm} \label{thmlowerbound} Suppose that Assumption \ref{122910} holds. We have
	\begin{equation} \label{www}
		\mathbb{E} \mathring{\mathcal{Z}}_{f,\gamma}^3= (\var(\tr f_{\gamma}(H)))^{-\frac32}\big(-r_1^{f_\gamma}N^{-\frac12}+r_{2,\beta}^{f_\gamma} N^{-1}\big)+O(N^{-\frac32})\,.
	\end{equation}
	As a consequence, for any fixed $\kappa>0$, we have 
	\begin{align}
		\Delta(\mathring{\cal Z}_{f,\gamma},Z)\geq C'_{f,\kappa}\Big(|r_1^{f_\gamma}|N^{-1/2-\kappa}+|r_{2,\beta}^{f_\gamma}|N^{-1-\kappa}\Big) \label{121101}
	\end{align}
	for some constant $C'_{f,\kappa}$, when $N$ is sufficiently large. 
\end{thm}
\begin{rem} \label{rmk1.6}
	Observe that 
	$
	r_1^{f_\gamma}=\frac{1}{8}(1-\gamma)^3(c_1^f)^3 \mathbb{E}h_d^3.
	$
	Apparently, we can replace $r_1^{f_\gamma}$ by $\cal X$ in (\ref{cal X}) when $f$ is not linear.  Hence, our lower bound matches the upper bound (up to $N^{2\kappa}$) in case $r_1^{f_\gamma}\neq 0$. Further, $r_2^{f_\gamma}$ is nonzero in general. For instance, if $f(x)=T_k(\frac{x}{2})$, the $k$-th  Chebyshev polynomial of the first  kind, it is easy to check 
	$r_1^{f_\gamma}=0$ and $r_{2,\beta}^{f_\gamma}=r_{2,\beta}^f= \beta^{-2}(\frac12 k^3+\frac{1}{6}k)$ if  $k$ is even and $k\geq 4$. In this case, our lower bound also matches the upper bound (up to an $N^{2\kappa}$ factor).  As a result, Theorem \ref{thmlowerbound} gives a positive answer for Question 2, that $O(N^{-\frac12})$ and $O(N^{-1})$ are indeed the best possible convergence rates in general. In addition, it also shows that $\cal X=0$ is a necessary and sufficient condition for the general linear statistics $\tr f(H)$ to converge to Gaussian with speed $O(N^{-1})$, which answers Question 3, up to an $N^{2\kappa}$ factor. 
	
\end{rem}

\begin{rem}
	We remark here that our discussion on lower bound can be extended to  general setting without the moment matching condition in Assumption \ref{122910} (i), and the regularity assumption on $f$ in (ii) can also be largely weakened. However, on one hand, the calculation and presentation under more general  assumptions will be much more involved, and on the other hand, the result under more general  assumptions will not be too much more informative for the purpose of checking the optimality of our upper bound.  Hence, in order to simplify the presentation,  we are not trying to optimize the conditions in the lower bound part. 
\end{rem}

\subsection*{Organization} The paper is organized as follows. In Section \ref{s. strategy}, we provide an outline of our proofs, with a highlight on heuristics and novelties. In Section \ref{s.pre}, we state some preliminaries including known estimates or basic notions for the subsequent sections. Section \ref{s.pfmainthm} is devoted to the proof of our main result, Theorem  \ref{thmmain}, based on Proposition \ref{lem 2148}, Lemma \ref{lem1558}, and Proposition \ref{lem. tf=0}. Proposition \ref{lem 2148} and Lemma \ref{lem1558} will be proved in Section \ref{sec6} and Appendix \ref{app B}, respectively. The main technical result, Proposition \ref{lem. tf=0}, will be proved in Sections \ref{sec4} and \ref{sec5}. In Section \ref{s. lowbound}, we prove the lower bound, i.e., Theorem \ref{thmlowerbound}. Some other technical results are proved in the appendices.

\subsection*{Notations and Conventions}
Throughout this paper, we regard $N$ as our fundamental large parameter. Any quantities that are not explicit constant or fixed may depend on N; we almost always omit the argument $N$ from our notation. We use $\|A\|$ to denote the operator norm of a matrix $A$ and use $\|\mathbf{u}\|_2$ to denote the $L^2$-norm of a vector $\mathbf{u}$. We use $c$ to denote some generic (small) positive constant, whose value may change from one expression to the next. Similarly, we use $C$ to denote some generic (large) positive constant. For $A \in \bb C$, $B>0$ and parameter $a$, we use $A = O_a(B)$ to denote $|A|\leq 
C_aB$ with some positive constant $C_a$ which may depend on $a$.

\subsection*{Acknowledgment} We would like to thank Xiao Fang and Gaultier Lambert for helpful discussion.

\vspace{2em}

\section{Heuristics, proof strategy, and novelties} \label{s. strategy}

In this section, we provide an outline of our proofs, with a highlight on the heuristics and novelties. 
Our starting point is the Helffer-Sj\"{o}strand formula in Lemma \ref{HS}, which allows us to rewrite the LES into 
\begin{align*}
	\langle \text{Tr} f(H)\rangle=\frac{1}{\pi}  \int_{\mathbb{C}} \frac{\partial}{\partial \bar{z}}\tilde{f}(z)\langle \text{Tr}G(z)\rangle{\rm d}^2 z\,,
\end{align*}
where $G(z)\deq (H-z)^{-1}$ is the Green function, $\langle\xi \rangle\deq\xi-\mathbb{E}\xi$ (c.f. \eqref{2.111}), and $\tilde{f}$ is the almost-holomorphic extension of $f$ defined in (\ref{def of f tilde}) below.

\textbf{Decomposition into diagonal and off-diagonal parts}.
The first key  observation is that the LES can be decomposed into two parts, which rely on the diagonal and off-diagonal entries of $H$, respectively. In general, these two parts have different convergence rates towards Gaussian. More precisely, let $\wh{H}=((1-\delta_{ij})H_{ij})_{N,N}$ be the off-diagonal part of $H$ and $\wh{G}(z)=(\wh{H}-z)^{-1}$ be its Green function. By Proposition \ref{lem 2148}, one can approximately write 
\begin{align*}
	\langle	\text{Tr} f(H)\rangle&=\frac{1}{\pi}  \int_{\mathbb{C}} \frac{\partial}{\partial \bar{z}}\tilde{f}(z)(\langle\tr\wh{G}(z)\rangle+m'(z)\tr H){\rm d}^2 z+\text{Error}\\
	&=\frac{1}{\pi}  \int_{\mathbb{C}} \frac{\partial}{\partial \bar{z}}\tilde{f}(z)\langle\tr\wh{G}(z)\rangle{\rm d}^2 z+\frac{1}{2}c_1^f\tr H+\text{Error}\eqd Z_o+\frac{1}{2}c_1^f\tr H+\text{Error}\,,
\end{align*}
where $m(z)$ and $c_1^f$ are defined in (\ref{2.5}) and \eqref{def of ck} respectively. From the classical Berry-Esseen bound, one knows that $\tr H$ approaches  Gaussian with a rate $O(N^{-1/2})$ in general. On the other hand, $Z_o$ is contributed ``equally" by $ O(N^2)$ independent random variables, which makes it possible to expect for a convergence rate of $O(N^{-1})$. Heuristically, by Schur complement,  the leading part of $\la \wh{G}_{ii} \ra$ is proportional to the centered quadratic form $x_i^*\wh{G}^{(i)}x_i-\frac{1}{N}\text{Tr}\widehat{G}^{(i)}$, 
where $x_i$ is the $i$-th column of $H$ with $H_{ii}$ removed, and $\wh{G}^{(i)}=(\wh{H}^{(i)}-z)^{-1}$. Here $\wh{H}^{(i)}$ is the minor of $\wh{H}$ with $i$-th row and column crossed out.   It is known from 
\cite{GT99, GT02} that a single quadratic form itself is already close to Gaussian up to a $O(\frac{1}{\sqrt{N}})$ Kolmogorov-Smirnov distance. Summing up $N$ such quadratic forms may then further reduce  this distance to $O(\frac{1}{N})$. This indicates the faster rate for $Z_o$. 

The above heuristic reasoning  motivates us to consider the shifted LES $\cal Z_{f,\gamma}$ in (\ref{def of Z_{f,gamma}}). When $\gamma=1$ or $c_1^f=0$, the leading contribution of the diagonal part of $H$ to $\cal Z_{f,\gamma}$ is removed, and one expects the convergence rate $O(N^{-1})$. In case $(1-\gamma)c_1^f\ne 0$ but $\mathbb{E}h_d^3=0$, the mechanism for gaining a $N^{-1}$ rate is a bit more subtle. Nevertheless,  one can compare this case with the toy model
$Z+\tr H$, where $Z$ is a Gaussian random variable that can be regarded as a replacement of the off-diagonal part up to a $N^{-1}$ error in distribution. A simple estimate of the  characteristic function for this toy model leads to the dependence of the convergence rate on $\mathbb{E}h_d^3$. For the above reasons, we introduce the parameter $\cal X$ in \eqref{cal X}, which identifies that  the slow convergence rate comes  from the diagonal part of $H$.

\textbf{Near-optimal estimate of the off-diagonal part}. 
The previous heuristic reasoning gives the correct prediction, it is nevertheless highly nontrivial to carry out rigorously.  Especially, the entries $\widehat{G}_{ii}$'s are correlated. A priori, there is no obvious evidence that summing up $N$ of them can reduce convergence rate by $\frac{1}{\sqrt{N}}$. 
By Esseen's inequality Lemma \ref{lem:essen}, in order to obtain the fast convergence of the off-diagonal part $Z_o$, the main step of our proof is to compute the characteristic function $\psi(t)\deq\bb E \exp(\ii t Z_o)$ up to a precision of $O_\prec(N^{-1})$ for all $t\in [0,N^{1-\varepsilon}]$. More precisely, we need to show that
\[
\Big|\psi(t)-\exp\Big(\frac{-\sigma_{f}^2t^2}{2}\Big)\Big| \leq C N^{-1+\varepsilon}
\] 
for all $t \in [0,N^{1-\varepsilon}]$. In order to do this, we shall show that $\psi$ satisfies the differential equation
\begin{align*}
	\psi'(t)=(-\sigma_{f}^2t+b(t))\psi(t)+\mathcal{E}(t) \,.
\end{align*}
The main difficulty lies at obtaining optimal estimate for the error terms $b(t)$, $\mathcal{E}(t)$, namely
\begin{align}
	b(t)\prec \frac{t^2}{N}, \qquad \mathcal{E}(t)\prec \frac{t+1}{N} \label{3.222}
\end{align}
for all $t \in [0,N^{1-\varepsilon}]$. Let us define the integration operator $\{\cdot\}_n\equiv \{\cdot\}_{n,f}$ as in \eqref{int f} below, such that $\psi'(t)=\{\ii \bb E \langle \tr \widehat{G} \rangle \exp(\ii t Z_o)\}$. Our starting point is 
\[
z\bb E \langle \tr \widehat{G} \rangle \exp(\ii t Z_o)={\sum_{i,j
}}^{\hspace{-0cm}*}\bb E \big[H_{ij}\widehat{G}_{ji} \langle \exp(\ii t Z_o)\rangle\big]\,,
\]
and we expand the RHS  to get a self-consistent equation, using the cumulant expansion formula in Lemma \ref{lem:cumulant_expansion}. By doing so, we will  arrive at 
\begin{align*}
	\psi'(t)&= -\sigma_{f}^2 \psi(t)+\sigma_{f,1}^2\psi(t)\notag\\
	&\qquad+\sum_{k\geq2} \ii \bigg\{\frac{s_{k+1} }{k!} \frac{1}{N^{(k+1)/2}} \frac{1}{z+2m(z)}{\sum_{i,j}}^* \bb E \frac{\partial^k( \widehat{G}_{ji} \langle \exp(\ii t Z_o) \rangle)}{\partial H_{ij}^k}\bigg\}+\text{error}
	\\ 
	&\eqd -\sigma_{f}^2 \psi(t)+\sigma_{f,1}^2\psi(t)+\sum_{k\geq2}\cal L_k+\text{error}\,,
\end{align*}
where $s_k$'s are defined in (\ref{022701}) and
\[
\sigma_{f,1}^2=\frac{s_4}{2\pi^2}\Big(\int_{-2}^2 f(x)\frac{2-x^2}{\sqrt{4-x^2}} {\rm d}x\Big)^2\,.
\]
The term $\sigma_{f,1}^2\psi(t)$ shall cancel the leading term in $\cal L_3$. From a straightforward computation, it can be proved, as previously in \cite{LS18} that 
\[
\cal L_k+\delta_{k4}\sigma_{f,1}^2\psi(t)= O_{\prec} (1+|t|^5)/\sqrt{N}\,,
\]
which apparently falls far short of our needs. In order to get the sharp estimate \eqref{3.222}, we need a much more careful treatment of $\cal L_k$. For brevity, in the remaining discussion of this section, for a sequence of  parameters $z_a$'s in $\mathbb{C}$, we write $\widehat{G}_a=\widehat{G}(z_a)$.  By the differential rule \eqref{diff} and local laws on Green function entries (c.f. Theorem \ref{refthm1}, Lemmas \ref{prop4.4} and \ref{lem0diag}), it will be seen that one type of the terms in $\cal L_k$ which is difficult to estimate is of the form 
\begin{align} 
	&\qquad\frac{t^k}{N^{(k+1)/2}}\bigg\{{\sum_{i,j}}^* \bb E (\widehat{G}_1)_{ij}(\widehat{G}^2_2)_{ji} \cdots (\widehat{G}^2_{k+1})_{ji} \exp(\ii t Z_o) \bigg\}_{k+1}  \label{031601} \\
	&=\, \frac{t^k}{N^{(k+1)/2}}\bigg\{{\sum_{i,j}}^* \bb E \Big\langle (\widehat{G}_1)_{ij}(\widehat{G}^2_2)_{ji} \cdots (\widehat{G}^2_{k+1})_{ji} \Big\rangle \exp(\ii t Z_o) \bigg\}_{k+1} \notag \\
	& \quad+\frac{t^k}{N^{(k+1)/2}}\bigg\{{\sum_{i,j}}^* \bb E (\widehat{G}_1)_{ij}(\widehat{G}^2_2)_{ji} \cdots (\widehat{G}^2_{k+1})_{ji}  \bigg\}_{k+1}  \psi(t) \eqd \cal L_{k,1}+O_{\prec}\Big(\frac{t^k}{N^{k-1}}\Big) \cdot \psi(t)\,,\notag 
\end{align}
where $\sum_{i,j}^*=\sum_{i\neq j}$ (c.f. (\ref{sum*})) and in the last step we used Lemma \ref{lem0diag} to estimate the off-diagonal entries of the Green functions and Lemma \ref{lem:integration} for the integral.  We emphasize that what we actually need to deal with is a class of more general and complicated variants of (\ref{031601}). Nevertheless, for simplicity, we focus on the toy case in (\ref{031601}) to illustrate the main mechanism. 
As $k\geq 2$, and we have the crucial condition $t\leq N^{1-\varepsilon}$, we see that
\[
O_{\prec}\Big(\frac{t^k}{N^{k-1}}\Big) \cdot \psi(t)=O_{\prec}\Big(\frac{t^2}{N}\Big) \cdot \psi(t)
\]
as desired. Hence it remains to show $\cal L_{k,1}=O_{\prec}(t^kN^{-k})=O_{\prec}(tN^{-1})$. By Cauchy-Schwarz inequality, it suffices to prove
\begin{equation} \label{666}
	\cal L_{k,2}\deq\bb E \bigg\{{\sum_{i,j}}^*\Big\langle (\widehat{G}_1)_{ij}(\widehat{G}^2_2)_{ji} \cdots (\widehat{G}^2_{k+1})_{ji} \Big\rangle\bigg\}^2_{k+1}=O_{\prec}(N^{-k+1})\,.
\end{equation}
A direct application of Lemma \ref{lem0diag} only leads to $\cal L_{k,2}=O_{\prec}(N^{-k+3})$, and thus we need to gain an additional  factor of $N^{-2}$. In order to achieve this, we need to exploit the smallness induced by the centering operator $``\langle \cdot \rangle"$ in $\cal L_{k,2}$. Such a mechanism of gaining additional smallness will be referred to  as {\it fluctuation averaging of the off-diagonal entries} ( of the Green function) in the sequel.

\textbf{Fluctuation averaging of the off-diagonal entries of the Green function}.
The fluctuation averaging of the diagonal entries of the Green function is well-understood by the local law. For instance, by Lemma \ref{prop4.4}, one can show that
\begin{equation} \label{3333}
	\bb E \bigg\{{\sum_{i}}^*\Big\langle (\widehat{G}_1)_{ii}(\widehat{G}^2_2)_{ii} \cdots (\widehat{G}^2_{k+1})_{ii} \Big\rangle\bigg\}^2_{k+1}=O_{\prec}(1)\,.
\end{equation}
On the other hand, fluctuation averaging for the off-diagonal entries of the Green function has not been fully studied in the literature. To prove \eqref{666}, let us denote
\[
Y_{k+1}\deq \bigg\{{\sum_{i,j}}^*(\widehat{G}_1)_{ij}(\widehat{G}^2_2)_{ji} \cdots (\widehat{G}^2_{k+1})_{ji}\bigg\}_{k+1}\,,
\] 
and the leading contribution of $\bb E \langle Y_{k+1}\rangle^2$ is contained in 
\begin{align*}
	&\bb E \bigg\{{\sum_{l\ne i}\sum_{i,j}}^* {H}_{il}(\widehat{G}_1)_{lj}(\widehat{G}^2_2)_{ji} \cdots (\widehat{G}^2_{k+1})_{ji} \bigg\}_{k+1}\langle Y_{k+1}\rangle \\
	=\ &\sum_{n\geq 1}\frac{s_{n+1}}{n!} \frac{1}{N^{(n+1)/2}} {\sum_{l\ne i}\sum_{i,j}}^*\bb E \bigg\{ \frac{\partial^n ((\widehat{G}_1)_{lj}(\widehat{G}^2_2)_{ji} \cdots (\widehat{G}^2_{k+1})_{ji})}{H_{il}^n}\bigg\}_{k+1}\langle Y_{k+1}\rangle \\
	&\hspace{-0.5cm}+	\sum_{n\geq 1}\sum_{r=1}^n{n \choose r}\frac{s_{n+1}}{n!} \frac{1}{N^{(n+1)/2}} {\sum_{l\ne i}\sum_{i,j}}^* \bb E\bigg\{ \frac{\partial^{n-r} ((\widehat{G}_1)_{lj}(\widehat{G}^2_2)_{ji} \cdots (\widehat{G}^2_{k+1})_{ji})}{H_{il}^{n-r}}\frac{\partial^r \langle Y_{k+1}\rangle }{\partial H_{il}^r}\bigg\}_{k+1}\\
	\eqd&\  \sum_{n\geq 1}\cal L^{(1)}_n+\sum_{n\geq 1} \cal L^{(2)}_n\,,
\end{align*}
where we expand the LHS of above via Lemma \ref{lem:cumulant_expansion}.  By repeatedly applying the cumulant expansion formula to $\mathcal{L}_n^{(1)}$, one can express it as a sum of two types of terms. The first type of terms contain enough off-diagonal Green function entries so that they can be neglected directly by applying the local laws in Theorem \ref{refthm1}, Lemmas \ref{prop4.4} and \ref{lem0diag}, thanks to the smallness of off-diagonal Green function entries. The other type of terms contain no off-diagonal entry except for those in the second $\la Y_{k+1}\ra$ factor. Thanks to the elementary fact $\mathbb{E}\xi \la Y_{k+1}\ra=\mathbb{E}\la \xi\ra\la Y_{k+1}\ra$ for any word $\xi$ in diagonal Green function entries, we can apply  the fluctuation averaging of diagonal entries such as (\ref{3333}). This mechanism will then be enough to bound $\mathcal{L}_n^{(1)}$.

The terms in $\cal L^{(2)}_n$ are more complicated, as the centering of $Y_{k+1}$ is destroyed by the derivatives. We are then forced to deal with the joint behavior of two $Y_{k+1}$. For example, one of the leading terms in $\cal L_{1}^{(2)}$ is of the form
\begin{align*}
-\frac{1}{N}{\sum_{i,j}}^*{\sum_{u,v}}^* \bb E \big\{(\widehat{G}^2_2)_{ji} \cdots (\widehat{G}^2_{k+1})_{ji}(\widehat{G}_1\widehat{G}_{k+2})_{uj}(\widehat{G}_{k+2})_{iv}(\widehat{G}^2_{k+3})_{uv} \cdots (\widehat{G}^2_{2k+2})_{uv}\big\}_{2k+2}\,.
\end{align*}
By Lemma \ref{lem0diag}, the above can only be estimated by $O_{\prec}(N^{-k+2})$, which means we are still in short by a factor $N^{-1}$. In order to exploit extra smallness for this types of terms, we identify several special cases for the monomials of the Green functions in Section \ref{sec6.1} below. More specifically, we will introduce the notion of {\it lone factor}, which is an off-diagonal Green function entry $\widehat{G}_{ij}$ whose index $(i,j)$ is not shared by the other entries in the monomial. Heuristically, a lone factor is weakly correlated with the other Green function entries in the monomial and thus it can bring additional smallness when one apply expectation to the monomial. We then further discuss the case when there are two lone factors $\widehat{G}_{ij}, \widehat{G}_{uv}$ with $\{i,j\}\cap \{u,v\}=\emptyset$ or the case there is a trio $\widehat{G}_{ij}, \widehat{G}_{iu}, \widehat{G}_{iv}$ with distinct $j,u,v$.  
We prove the improved estimates for these cases, which finally enable us to conclude \eqref{666}. The heuristics of the lone factor and related notions are explained in Remark \ref{remK} below.

Observe that the size of a sum of monomials in Green function entries essentially depends on a few key parameters. For instance, first,  if an index shows up exactly twice in the monomial, the sum over such an index will reduce the number of terms by $N$. Second, each off-diagonal Green function entry can contribute an $\frac{1}{\sqrt{N}}$ factor due to the local law. Third, since we will need to work in the regime $\eta=\Im z\gtrsim N^{-\frac14}$ due to the assumption $f\in C^5(\mathbb{R})$, we also need to monitor the power of $\eta^{-1}$, which may contribute additional $N$-factors. There are several other factors crucial for the estimate, such as the $t$-factor which may be large since we are working with $t\in [0,N^{1-\varepsilon}]$. Therefore, we will introduce a uniform bookkeeping system to keep tracking on the evolution of the above key parameters during the cumulant expansions. We refer to the notion {\it abstract polynomial} and its  key parameters in Section \ref{sec8} for details.

\textbf{The lower bound.}
Finally, in order to obtain a matching lower bound of the convergence rate, we turn to estimate the third moment of the centered LES, i.e. $\mathbb{E}[(\text{Tr} f(H)-\mathbb{E} \text{Tr} f(H)-\frac12 \gamma c_1^f \text{Tr} H)]^3$. It turns out the third moment has precise leading term $r_1\mathcal{X}N^{-\frac12}+r_2 N^{-1}$, where $r_1$ and  $r_2$ are constants depending on $f$ and they are nonzero in general. This suggests that in general, we cannot approximate the LES with any Gaussian random variable, with a  precision better than $O(\cal X N^{-1/2}+N^{-1})$. A simple argument using the eigenvalue rigidity implies that the convergence rate has a lower bound  $|r_1|\mathcal{X}N^{-\frac12-\kappa}+|r_2| N^{-1-\kappa}$ in Kolmogorov-Smirnov distance. 

The computation of the third moment of the LES boils down to that of the three point function of Green functions, which is also derived by the cumulant expansion. Using a naive estimate by Theorem \ref{refthm1}, the third moment can only be bounded by $O(1)$. In order to reach the true scale $O((\cal X N^{-1/2}+N^{-1})$, we shall discover several nontrivial cancellations to obtain a more precise estimate than what has been done in previous works.

\section{Preliminaries} \label{s.pre}

Throughout  the paper, for an $N\times N$ matrix $M$, we write $\ul{M}\deq \frac{1}{N} \tr M$, and we abbreviate $M^{*n}\deq (M^{*})^n$, $M^{*}_{ij}\deq (M^{*})_{ij} = \ol M_{ji}$, $M_{ij}^n\deq (M_{ij})^n$.  We emphasize here that $M_{ij}^n$ is different from $(M^n)_{ij}$ in general, where the latter apparently means the $(i,j)$ entry of $M^n$. We use 
\begin{equation} \label{2.111}
	\langle X \rangle\deq X-\bb E X
\end{equation} 
to denote the centering of the random variable $X$. For any positive integer $n$, we use the notation $\llbracket 1, n\rrbracket\deq \{1,2,...,n\}$.  For indices $i_1,...,i_k\in\llbracket 1, N\rrbracket$, we use
\begin{equation}\label{sum*}
	{\sum_{i_1,...,i_k}}^{\hspace{-0.15cm}*}
\end{equation}
to denote the sum over all $k$-tuples $(i_1,...,i_k)\in \llbracket 1, N\rrbracket^k$ with distinct indices.
We further introduce the following notations for the conditional expectations
\begin{equation} \label{2.1}
	\bb E_{\rm o}(\cdot) \deq \bb E(\cdot | H_{ii}, i=1,...,N) \quad \mbox{and} \quad  \bb E_{\rm d}(\cdot) \deq \bb E(\cdot | H_{ij}, i,j=1,...,N, i \ne j)\,.
\end{equation}
Let $z \in \bb C \backslash \bb R$. We denote the \emph{Green function of $H$} and its Stieltjes transform  by
\begin{equation*} 
	G \equiv G(z)\deq (H-z)^{-1}\quad  \mbox{and} \quad s(z) \deq \frac{1}{N} \tr G(z)\,=\ul{G(z)}.
\end{equation*}
Let $\rho_{sc}$ be the semicircle density, i.e.
$
\rho_{sc}(x) =\frac{1}{2\pi} \sqrt{(4-x^2)_+}.
$
The Stieltjes transform $m$ of the semicircle density $\rho_{sc}$ is given by
\begin{equation} \label{2.5}
	m(z)\deq \int \frac{\rho_{sc}(x)}{x-z}\,\dd x=\frac{1}{2\pi}\int_{-2}^{2}\frac{\sqrt{4-x^2}}{x-z}\dd x\,,
\end{equation}
so that
$
m(z)=\frac{-z+\sqrt{z^2-4}}{2}.
$
Here the square root $\sqrt{z^2-4}$ is chosen with a branch cut in the segment $[-2,2]$ so that $\sqrt{z^2-1}\sim z$
as $z\to \infty$. For $z=E+\ii \eta$,  $\eta>0$, it is easy to check
\begin{equation*}
	m(z)=O(1) \quad \mbox{ and }\quad m'(z)=O\Big(\big(|E^2-4|+\eta\big)^{-1/2}\Big)\,.
\end{equation*}
For $f\in C^1(\bb R)$, we recall the parameters defined in Definition  \ref{def:Wigner} and further set
\begin{align} \label{mu_f}
		&\mu_f\deq N\int_{-2}^2 f(x)\rho_{sc}(x) {\rm d}x-\frac{1}{2\pi} \Big(\frac{2}{\beta}-1\Big)\int_{-2}^2 \frac{f(x)}{\sqrt{4-x^2}}{\rm d}x+\frac{1}{2\beta}(f(2)+f(-2)) \notag\\
		&+\frac{a_2-2\beta^{-1}}{2\pi}\int_{-2}^2 f(x) \frac{2-x^2}{\sqrt{4-x^2}}{\rm d}x+\frac{s_4}{2\pi} \int_{-2}^2 f(x)\frac{x^4-4x^2+2}{\sqrt{4-x^2}}{\rm d}x \notag\\
		&+\frac{a_3}{4 \pi \sqrt{N}}\int_{-2}^2 f(x)\frac{x^3-2x}{\sqrt{4-x^2}}\dd x\,,  
\end{align}
and 
\begin{equation} \label{sigma_f}
	\begin{aligned}
		\sigma^2_f\deq&\frac{1}{2\beta\pi^2}\int_{-2}^2\int_2^2\frac{(f(y)-f(x))^2}{(x-y)^2}\frac{4-xy}{\sqrt{4-x^2}\sqrt{4-y^2}}{\rm d}x{\rm d}y\\
		&-\frac{1}{2\beta\pi^2}\Big(\int_{-2}^2 f(x)\frac{x}{\sqrt{4-x^2}}{\rm d}x\Big)^2
		+\frac{s_4}{2\pi^2}\Big(\int_{-2}^2 f(x)\frac{2-x^2}{\sqrt{4-x^2}} {\rm d}x\Big)^2\,.
	\end{aligned}
\end{equation}
Note that by \cite[Equation (1.5)]{BY05}, we can rewrite the variance into
\begin{equation*} 
	\sigma^2_f=\frac{1}{2\beta}\sum_{k=2}^{\infty} k(c_k^f)^2+\frac{s_4}{2} (c_2^f)^2\,,
\end{equation*}
which implies that $\sigma_f^2 \geq 0$, and $\sigma_f^2=0$ if and only if $f$ is linear.

Let $h$ be a real-valued random variable with finite moments
of all order. We use $\cal C_n(h)$ to denote the $n$th cumulant of $h$, i.e.
\[
{\cal C}_n(h)\deq (-\ii)^n \cdot\big(\partial_{\lambda}^n \log \bb E \ee^{\ii \lambda h}\big) \big{|}_{\lambda=0}\,.
\]
The following is a simple consequence of Definition \ref{def:Wigner}.

\begin{lem} \label{lemh}
	For every fixed $n \in \bb N$ we have
	\begin{equation*}
		\cal C_{n}(H_{ij})=O_{n}(N^{-n/2})
	\end{equation*}
	uniformly for all $i,j$.
\end{lem}

We will need the following expansion formula, due to Andrew Barbour \cite{Barbour}. It was first applied to random matrix theory in \cite{KKP2}, and it has been widely used by the random matrix community in recent works, e.g.\,\cite{HK,LeeSchnelli18,HKR,EKS19,HLY17,HK2,H19}. A proof of a slightly different version can be found in \cite[Lemma 2.4]{HKR}.

\begin{lem}[Barbour's cumulant expansion formula] \label{lem:cumulant_expansion}
	Let $\cal F:\R\to\C$ be a smooth function, and denote by $\cal F^{(n)}$ its $n$th derivative. Then, for every fixed $\ell \in\N$, we have 
	\begin{equation}\label{eq:cumulant_expansion}
		\mathbb{E}\big[h\cdot \cal F(h)\big]=\sum_{n=0}^{\ell}\frac{1}{n!}\mathcal{C}_{n+1}(h)\mathbb{E}[\cal F^{(n)}(h)]+\cal R_{\ell+1},
	\end{equation}	
	assuming that all expectations in \eqref{eq:cumulant_expansion} exist, where $\cal R_{\ell+1}$ is a remainder term (depending on $\cal F$ and $h$), such that for any $s>0$,
	\begin{equation*} 
		\cal R_{\ell+1} = O(1) \cdot \bigg(\E\sup_{|x| \le |h|} \big|\cal F^{(\ell+1)}(x)\big|^2 \cdot \E \,\big| h^{2\ell+4} \mathbf{1}_{|h|>s} \big| \bigg)^{1/2} +O(1) \cdot \bb E |h|^{\ell+2} \cdot  \sup_{|x| \le s}\big|\cal F^{(\ell+1)}(x)\big|\,.
	\end{equation*}
\end{lem}

Next we introduce the notions of stochastic domination.
\begin{defn} [Stochastic domination]  \label{def2.1}
	Let
	\begin{equation*}
		\mathsf{X}=(\mathsf{X}^{(N)}(u):  N \in \mathbb{N}, \ u \in \mathsf{U}^{(N)}), \   \mathsf{Y}=(\mathsf{Y}^{(N)}(u):  N \in \mathbb{N}, \ u \in \mathsf{U}^{(N)}),
	\end{equation*}
	be two families of random variables, where $\mathsf{Y}$ is nonnegative, and $\mathsf{U}^{(N)}$ is a possibly $N$-dependent parameter set. 
	
	We say that $\mathsf{X}$ is stochastically dominated by $\mathsf{Y},$ uniformly in $u,$ if for all small $\epsilon>0$ and large $ D>0$, we have 
	\begin{equation*}
		\sup_{u \in \mathsf{U}^{(N)}} \mathbb{P} \Big( \big|\mathsf{X}^{(N)}(u)\big|>N^{\epsilon}\mathsf{Y}^{(N)}(u) \Big) \leq N^{- D},
	\end{equation*}   
	for large enough $N \geq  N_0(\epsilon, D).$  If $\mathsf{X}$ is stochastically dominated by $\mathsf{Y}$, uniformly in $u$, we use the
	notation $\mathsf{X} \prec \mathsf{Y}$ , or equivalently $\mathsf{X} = O_{\prec}(\mathsf{Y})$. Note that in the special case when $\mathsf{X}$ and $\mathsf{Y}$ are deterministic, $\mathsf{X} \prec\mathsf{Y}$
	means that for any given $\varepsilon>0$,  $|\mathsf{X}^{(N)}(u)|\leq N^{\epsilon}\mathsf{Y}^{(N)}(u)$ uniformly in $u$, for all sufficiently large $N\geq N_0(\epsilon)$.

	Throughout this paper, the stochastic domination will always be uniform in
	all parameters (mostly are matrix indices and the spectral parameter $z$) that are not explicitly fixed.
\end{defn}

We have the following elementary result about stochastic domination.

\begin{lem} \label{prop_prec} Let
	\begin{equation*}
		\mathsf{X}_i=(\mathsf{X}^{(N)}_i(u):  N \in \mathbb{N}, \ u \in \mathsf{U}^{(N)}), \   \mathsf{Y}_i=(\mathsf{Y}_i^{(N)}(u):  N \in \mathbb{N}, \ u \in \mathsf{U}^{(N)}),\quad i=1,2
	\end{equation*}
	be families of  random variables, where $\mathsf{Y}_i, i=1,2,$ are nonnegative, and $\mathsf{U}^{(N)}$ is a possibly $N$-dependent parameter set.	Let 
	\begin{align*}
		\Phi=(\Phi^{(N)}(u): N \in \mathbb{N}, \ u \in \mathsf{U}^{(N)})
	\end{align*}
	be a family of deterministic nonnegative quantities. We have the following results:
	\begin{enumerate}
		\item[(i)] If $\mathsf{X}_1 \prec \mathsf{Y}_1$ and $\mathsf{X}_2 \prec \mathsf{Y}_2$ then $\mathsf{X}_1+\mathsf{X}_2 \prec \mathsf{Y}_1+\mathsf{Y}_2$ and  $\mathsf{X}_1 \mathsf{X}_2 \prec \mathsf{Y}_1 \mathsf{Y}_2$.
		
		\item[(ii)] Suppose ${X}_1 \prec \Phi$, and there exists a constant $C>0$ such that  $|\mathsf{X}_1^{(N)}(u)| \leq N^{C}\Phi^{(N)}(u)$ a.s. uniformly in $u$ for all sufficiently large $N$. Then $\E \mathsf{X}_1 \prec \Phi$.  
	\end{enumerate}
\end{lem}

We further define the following relation between random variables that will facilitate our discussion.
\begin{defn}[Equivalent class in the Kolmogorov-Smirnov distance]
	Fix $M>0$. Two families of real random variables $X=(X_N)$,$Y=(Y_N)$ are said to be equivalent, with parameter $M$, if 
	\begin{equation*}
		\mathbb P(|X_N-Y_N|\geq CN^{-M}) \leq CN^{-M}
	\end{equation*}
	for some constant $C>0$. We denote it by
	$
	X\sim_{M} Y.
	$
\end{defn}

It is easy to check that ``$\sim_M$" defines an equivalent relation. The next lemma shows a relation between $\sim_M$ and the Kolmogorov-Smirnov distance. 

\begin{lem} \label{lem:KS comparison}
	Recall that $Z \overset{d}{=}\cal N(0,1)$. Suppose two families of real random variables $X=(X_N)$, $Y=(Y_N)$ satisfy $X\sim_M Y$
	for some fixed $M>0$.  Then
	\[
	\Delta(X,Z)=\Delta(Y,Z)+O(N^{-M})\,.
	\]
\end{lem}
\begin{proof}
	Let $x\in \bb R$. Since $X\sim_M Y$, we see that
	\[
	\mathbb P(Y\leq x-|\varepsilon|)+O(N^{-M})\leq\mathbb P(X\leq x) \leq \mathbb P(Y\leq x+|\varepsilon|)+O(N^{-M})
	\]
	for some $\varepsilon=O(N^{-M})$ independent of $x$. The proof then follows by using the triangle inequality and the fact that the distribution function of $Z$ is Lipschitz continuous.
\end{proof}

We now state the local semicircle law for Wigner matrices from \cite{EYY,KY13}.
\begin{thm}[Isotropic local semicircle law] \label{refthm1}
	Let $H$ be a Wigner matrix satisfying Definition \ref{def:Wigner}, and define the spectral domains
	$$ 
	{\bf S}  \deq  \{E+\mathrm{i}\eta: |E| \le 10, 0 <  \eta \le 10 \} \quad \mbox{and} \quad \mathbf{S}^{o}(c)\deq\big\{ E+\ii \eta\in \b S: |E|\geq 2+N^{-2/3+c}\big\}
	$$
	for fixed $c>0$. Then for deterministic $\b u, \b v \in \bb C^{n}$ with $\|\b u\|_2=\|\b v\|_2=1$, we have the bounds
	\begin{equation*}  
		\langle \b u, G(z) \b v \rangle -m(z) \langle \b u, \b v \rangle  \prec \sqrt{\frac{\im m(z)}{N\eta}}+\frac{1}{N\eta} \quad \mbox{and} \quad |s(z)-m(z)| \prec \frac{1}{N\eta}\,,
	\end{equation*} 
	uniformly in $z =   
	E+\mathrm{i}\eta \in {\bf S}$. Moreover, outside the bulk of the spectrum, we have the stronger estimates
	\begin{equation*}
		\langle \b u, G(z) \b v \rangle -m(z) \langle \b u, \b v \rangle  \prec \frac{1}{\sqrt{N}(\eta+|E^2-4|)^{1/4}} \quad \mbox{and} \quad
		|\ul{G}(z)-m(z)|\prec  \frac{1}{N(\eta+|E^2-4|)} \label{19083101}
	\end{equation*}
	uniformly for $z=E+\ii \eta \in \mathbf{S}^{o}(c)$. 
\end{thm}

One standard consequence of the above local law is the following bound on the spectral norm of $H$.
\begin{cor} \label{cor:spectral norm}
	We have
	\[
	\|H\|-2 \prec N^{-2/3}\,.
	\]
\end{cor}

For any fixed $c>0$, we define
\begin{equation} \label{sc+}
	\b S_c \deq \{E+\ii \eta: |E|\leq 10, |\eta|\geq N^{-1+c}\} \quad \mbox{ and } \quad \b S_c^+ \deq \{E+\ii \eta: |E|\leq 10, \eta\geq N^{-1+c}\} \,.
\end{equation}
The following lemma is a preliminary estimate on $G$. It provides a priori bounds on entries of powers of $G$ which are significantly better than those obtained by a direct application of the local semicircle law. The proof is postponed to Appendix \ref{appD}.

\begin{lem} \label{prop4.4} 
	Let $H$ be a Wigner matrix satisfying Definition \ref{def:Wigner}. Fix $c>0$ and $l \in \bb N_+$.  For $j\in \{1,2,...,l\}$, we define $G_j \equiv G(z_j)=(H-z_j)^{-1}$, where $z_j=E_j+\ii \eta_j \in \b S_c$. Suppose $|\eta_1| \leq |\eta_2|\leq \cdots \leq |\eta_l|$. We have
	\begin{equation} \label{2.12}
		\tr\big( G_1^{k_1}\cdots G_l^{k_l}\big)-\bb E \tr\big( G_1^{k_1}\cdots G_l^{k_l}\big) \prec  \frac{1}{\big|\eta_1^{k_1} \cdots \eta_l^{k_l}\big|}
	\end{equation}
	and
	\begin{equation} \label{2.13}
		\sum_{i}(G_1^{k_1})_{ii}\cdots (G^{k_l}_l)_{ii}-Nm^{(k_1-1)}(z_1) \cdots m^{(k_l-1)}(z_l)\prec \frac{1}{\big|\eta_1^{k_1}\eta_2^{k_2-1} \cdots \eta_l^{k_l-1}\big|}
	\end{equation}
	as well as
	\begin{equation} \label{2.14}
		\big( G_1^{k_1}\cdots G_l^{k_l}\big)_{ij}-\delta_{ij} m_{l}((z_1,k_1),...,(z_l,k_l))\prec \frac{1}{\sqrt{N|\eta_1|}} \frac{1}{\big|\eta_1^{k_1-1} \eta_2^{k_2} \cdots \eta_l^{k_l}\big|}
	\end{equation}
	for any fixed $k_1,...,k_l \in \bb N_+$. Here 
	\begin{align}
	m_{l}((z_1,k_1),...,(z_l,k_l))\deq \int \frac{\varrho_{sc}(x)}{(x-z_1)^{k_1} \cdots (x-z_l)^{k_l}} \dd x \prec \frac{1}{\big|\eta_1^{k_1-1} \eta_2^{k_2} \cdots \eta_l^{k_l}\big|}\,.
	\label{031901}
	\end{align}
\end{lem}

\begin{rem} We emphasize that (\ref{2.12})-(\ref{2.14}) are particular cases of equations (52) and (54) in \cite{ECS20}; also see Theorem 4.1 of \cite{ECS20a} and Theorem 3.4 of \cite{ECS21} for more general versions, which are called {\it multi-resolvent local laws} therein. Further, we also emphasize that the quantity in (\ref{031901}) is a special case of the one considered in equation (3.8) of \cite{ECS21} (also see (3.25) therein). But due the simplicity of the special forms  in Lemma \ref{prop4.4}, we state a short proof in Appendix \ref{appD} for readers' convenience.
\end{rem}

Next we use $\widehat{H}$ to denote the Wigner matrix with zero diagonal entries, i.e.
$
\widehat{H}_{ij}=(1-\delta_{ij})H_{ij}
$
for all $i,j \in \llbracket 1, n\rrbracket$. We define the corresponding Green function and  Stieltjes transform by
\begin{equation} \label{def_hatG}
	\widehat{G} \equiv \widehat{G}(z)\deq (\widehat{H}-z)^{-1}\quad  \mbox{and} \quad \widehat{s}(z) \deq \frac{1}{N} \tr \widehat{G}(z)\,=\ul{\widehat{G}(z)}.
\end{equation}
We have the following result.

\begin{lem} \label{lem0diag}
	Theorem \ref{refthm1} and Lemma \ref{prop4.4} remain valid when we replace $H$ and $G$ by $\widehat{H}$ and $\widehat{G}$, respectively.
\end{lem}
\begin{proof}
	The statement is trivially true, as by Definition \ref{def:Wigner}, $\widehat{H}$ is again a Wigner matrix, with $a_2=0$.
\end{proof}

Set the matrix 
\begin{equation} \label{HHH}
	H_{\mathrm d}\deq H-\widehat{H}\,,
\end{equation}
which  consists of the diagonal entries of $H$. We have the following estimate that can be easily deduced from Lemma \ref{lem0diag}.

\begin{lem} \label{lemGHG}
	Fix $\omega \in \bb N$.	We have
	\[
	(\widehat{G}(z)H_{\mathrm d}\widehat{G}^{1+\omega}(z))_{ij}\prec \Big(\frac{1}{N|\eta|}+\frac{\delta_{ij}}{\sqrt{N|\eta|}}\Big)\frac{1}{|\eta|^{\omega}}
	\]
	uniformly for $z \in \b S_c$.
\end{lem}

We conclude this section with two classical results.

\begin{lem}[Helffer-Sj\"{o}strand formula] \label{HS}
	Let $\phi \in C^{1}(\bb C)$ such that $\phi(z) = 0$ for large enough $|\re z|$. Then for any $\lambda \in \mathbb{R}$ we have
	\begin{align*} 
		\phi(\lambda)=\frac{1}{\pi}\int_{\mathbb{C}}\frac{\partial_{\bar{z}}\phi(z)}{\lambda-z}\,{\rm d}^2z\,, 
	\end{align*}
	where $\partial_{\bar{z}}:= \frac{1}{2}(\partial_x+\mathrm{i}\partial_y)$ is the antiholomorphic derivative and ${\rm d}^2 z$ is the Lebesgue measure on $\mathbb{C}$.
\end{lem}

\begin{lem}[Esseen inequality] \label{lem:essen}
	Let $Y$ be a real random variable with characteristic function $\psi(t)=\bb E \mathrm{e}^{\ii t Y}$. Let $\Phi(x)$ be the distribution function of $Z\overset{d}{=}\cal N(0,1)$.  There exists a constant $C>0$ such that for any $\sigma,T>0$, we have
	\[
	\Delta(Y\sigma^{-1},Z)=\sup_{x \in \bb R}|\bb P(Y \sigma^{-1}\leq x) -\Phi(x)|\leq C \int_0^T \frac{|\psi(t\sigma^{-1})-\exp(-t^2/2)|}{t} \,\dd t +\frac{C}{T}
	\]
\end{lem}

\vspace{2em}

\section{Proof of Theorem \ref{thmmain}} \label{s.pfmainthm}
For the rest of this paper we fix $\kappa>0$ as in Theorem \ref{thmmain}. Let $f$ be as in Theorem \ref{thmmain}. It suffices to assume that $f$ has compact support,  since we can construct $\widehat{f} \in C^5(\bb R)$ such that
\[
\widehat{f}(x)=\begin{cases}
	f(x) & \mbox{if} \quad  |x|\leq 3,\\
	0   & \mbox{if} \quad   |x| \geq 4\,.
\end{cases}
\] 
Note that $f=\widehat{f}$ on $[-3,3]$, and by Corollary \ref{cor:spectral norm} we have
\[
\mathbb P(\tr f(H)\ne \tr \widehat{f}(H) )\leq \mathbb P(\|H\|\geq 3)\leq CN^{-10}\,.
\]
Thus
$\cal Z_{f,\gamma} =\cal Z_{\widehat{f},\gamma}$ with probability at least $1-CN^{-10}$, which together with Lemma \ref{lem:KS comparison} implies
\[
\Delta(\cal Z_f,Z)=\Delta(\cal Z_{\widehat{f}},Z)+O(N^{-10})\,.
\] 
From now on we shall always assume $\supp f \subset [-4,4]$, and as a consequence $f$ is bounded. We have the following result, whose proof is postponed to Appendix \ref{appA}.

\begin{lem} \label{lem:exp tr f(H)}
	Let $f\in C^5_c(\bb R)$ and $\mu_f$ be as in \eqref{mu_f}. We have
	\[
	\mathbb E \tr f(H) -\mu_f =O_\prec(N^{-1})\,.
	\]
\end{lem}

By Lemma \ref{lem:exp tr f(H)} we see that
\[
\langle {\cal Z}_{f,\gamma} \rangle = \cal Z_{f,\gamma}-\bb E \cal Z_{f,\gamma}=\frac{\tr f(H)-\bb E \tr f(H)-\frac12\gamma c_1^f \tr {H}}{\sigma_{f,\gamma}}
\]
satisfies $\langle {\cal Z}_{f,\gamma} \rangle \sim_{1-\kappa} {\cal Z}_{f,\gamma}  $. Hence,  Lemma \ref{lem:KS comparison} implies
\begin{equation} \label{z0}
	\Delta({\cal Z}_{f,\gamma} ,Z)=\Delta(\langle {\cal Z}_{f,\gamma} \rangle ,Z)+O(N^{-1+\kappa})\,.
\end{equation}
By Lemma \ref{HS}, we can write 
the LES as
\begin{align}
\text{Tr} f(H)= \frac{1}{\pi} \int_{\mathbb{C}} \frac{\partial }{\partial \bar{z}} \tilde{f}(z)\text{Tr} G(z) {\rm d}^2 z\,, \quad \mbox{where} \quad	\tilde{f}(z)\deq \sum_{k=0}^4\frac{(\ii y)^k}{k!}f^{(k)}(x) \chi(y)\,, \label{def of f tilde}
\end{align}
and $\chi\in C_c^{\infty}(\bb R)$ satisfies $\chi(y)=1$ for $|y|\leq 1$ and $\chi(y)=0$ for $|y|\geq 2$. Define the region
\begin{equation} \label{ddd}
	\b D\deq\{x+\ii y: |x|\leq 10, |y|\geq N^{-1/4}\}\,.
\end{equation}
By Theorem \ref{refthm1}, we see that $\langle \tr G(z)\rangle \prec |y|^{-1}$ uniformly for $z \in \b S$. Thus
\begin{equation}  \label{4.2}
	\begin{aligned}
		\frac{1}{\pi} \int_{\bb C \backslash\b D} \frac{\partial }{\partial \bar{z}} \tilde{f}(z)\langle\text{Tr} G(z)\rangle  {\rm d}^2 z=-\frac{\ii }{48\pi} \int_{\bb C \backslash\b D} y^4 f^{(5)}(x)\langle\text{Tr} G(z)\rangle{\rm d}^2 z\prec N^{-1}\,,
	\end{aligned}
\end{equation}
where we used the fact that $\supp f\subset [-4,4]$.  Let
\begin{equation} \label{4.3}
	\widetilde{\cal Z}_{f,\gamma}\deq \frac{1}{\pi\sigma_{f,\gamma}} \int_{\b D} \frac{\partial }{\partial \bar{z}} \tilde{f}(z)\langle\text{Tr} G(z) \rangle{\rm d}^2 z-\frac{\frac12\gamma c_1^f\tr {H}}{\sigma_{f,\gamma}}\,.
\end{equation}
Note that by the symmetry of the domain $\b D$, the random variable $\widetilde{\cal Z}_{f,\gamma}$ is real. From \eqref{4.2} we know that $|\widetilde{\cal Z}_{f,\gamma}-\langle {\cal Z}_{f,\gamma} \rangle|\prec  N^{-1}$. By Lemma \ref{lem:KS comparison} and \eqref{z0} we have  
\begin{equation} \label{3.33}
	\Delta(\cal Z_{f,\gamma},Z)=\Delta(\widetilde{\cal Z}_{f,\gamma},Z)+O( N^{-1+\kappa})\,.
\end{equation}
The first key step of our proof is the following decomposition, which splits off the contribution of the diagonal entries. The proof is postponed to Appendix \ref{sec6}.

\begin{pro} \label{lem 2148}
	We have
	\[
	\langle\tr G(z)\rangle -\langle\tr \widehat{G}(z)\rangle+\sum_i \langle(\widehat{G}^2)_{ii}\rangle H_{ii} +m'(z) \tr H-m'(z)m(z)\Big(\sum_i H_{ii}^2-a_2\Big)  \prec \frac{1}{N|\eta|^3}
	\]
	uniformly for $z \in \b S_c$.
\end{pro}

By Green's formula, we have 
\begin{multline*}
	\frac{1}{\pi\sigma_{f,\gamma}} \int_{\b D} \frac{\partial }{\partial \bar{z}} \tilde{f}(z)\Big(\langle\text{Tr} \widehat{G}(z) \rangle-\sum_i \langle(\widehat{G}^2)_{ii}\rangle H_{ii}-m'(z)\tr H+m'(z)m(z)\Big(\sum_i H_{ii}^2-a_2\Big)\Big){\rm d}^2 z\\
	-\frac{\frac12\gamma c_1^f\tr {H}}{\sigma_{f,\gamma}}=\frac{1}{\pi\sigma_{f,\gamma}} \int_{\b D} \frac{\partial }{\partial \bar{z}} \tilde{f}(z)\langle\text{Tr} \widehat{G}(z) \rangle{\rm d}^2 z
	-\frac{1}{\pi\sigma_{f,\gamma}} \int_{\b D} \frac{\partial }{\partial \bar{z}} \tilde{f}(z)\sum_i \langle(\widehat{G}^2)_{ii}\rangle H_{ii} {\rm d}^2 z\\
	+\frac{1}{\sigma_{f,\gamma}}\Big(\frac12(1-\gamma)c_1^f\tr {H}+\frac12 c_2^f \Big(\sum_{i}H_{ii}^2-a_2\Big)\Big)
	\eqd \widehat{Z}_{f,\gamma,1}+\widehat{Z}_{f,\gamma,2}+\widehat{Z}_{f,\gamma,3} \eqd  \widehat{Z}_{f,\gamma}
\end{multline*}
Since $\frac{\partial }{\partial \bar{z}} \tilde{f}(z)=-\frac{ \ii y^3}{12}f^{(4)}(x)$ for $|y|\leq 1$ and $f$ is compactly supported, we can easily see from \eqref{4.3} and Proposition \ref{lem 2148} that
$
\widehat{Z}_{f,\gamma} \sim_{1-\kappa} \widetilde{\cal Z}_{f,\gamma}.
$
By Lemma \ref{lem:KS comparison} and \eqref{3.33} we have
\begin{equation} \label{brick}
	\Delta(\widehat{Z}_{f,\gamma},Z)=\Delta({\cal Z}_{f,\gamma},Z)+O(N^{-1+\kappa})\,.
\end{equation}
In the squeal, we denote
$
\varphi_k (t)\deq \exp( \ii t \sigma_{f,\gamma} \widehat{{\cal Z}}_{f,\gamma,k}) 
$
for $k=1,2,3$, and
\begin{equation}
	X_i\deq \Big(-\frac{1}{\pi} \int_{\mathbf{D}} \frac{\partial}{\partial \bar{z}} \tilde{f}(z)\la (\wh{G}^2)_{ii}\ra {\rm d}^2 z H_{ii}+\frac{1}{2}c_2^f (H_{ii}^2-a_2N^{-1})\Big)  \label{def of Xi} 
\end{equation} 
for $i \in \llbracket 1,N \rrbracket$. Let
\begin{align}
	\varphi_4(t)\deq\exp\Big(-\frac{t^2}{2}\sum_i \mathbb{E}_d X_{i}^2+\frac{(\ii t)^3}{6}\sum_{i}\mathbb{E}_d X_{i}^3\Big)\,. \label{eq of phi 4}
\end{align}
Note that by the definition of $\b D$, $X_i$'s are real random variables, and as a result $|\varphi_4(t)|\leq 1$. We have the following result concerning the diagonal entries of $H$.

\begin{lem} \label{lem1558}
	Fix $c>0$.  Recall the definition of $\cal X$ from \eqref{cal X}, and the definition of $\bb E_{\rm d}$ from \eqref{2.1}. We have
	\begin{align} 
		\bb E_{\rm d}\big[\varphi_{2}(t)\varphi_{3}(t)\big]
		=\exp\Big(-\frac{a_2(1-\gamma)^2(c_1^f)^2t^2}{8}\Big)\varphi_4(t)+O_{\prec} (N^{-1}+ \cal X N^{-1/2
		})  \label{012110}
	\end{align}
	uniformly for $t \in [0,N^{(1-c)/2}]$.
\end{lem}
\begin{proof}
	See Appendix  \ref{app B}.
\end{proof}

The second key step of our proof is the following estimate of the characteristic function, especially when $t$ gets large. The proof  will be our major technical task, and it will be postponed to Section \ref{sec4} below.

\begin{pro} \label{lem. tf=0} Fix $c>0$, and recall the definition of $\bb E_{\rm o}$ from \eqref{2.1}. Let us denote
	\[
	\phi_\xi(t)\deq \bb E_{\rm o}[\varphi_1(t)(\xi \varphi_{2}(t)+(1-\xi)\varphi_4(t))]
	\]
	with parameter $\xi \in \{0,1\}$. We have
	\begin{equation} \label{ODE}
		\phi_\xi'(t)=(-\sigma_f^2 t+b_\xi(t)) \phi_\xi(t)+ \cal E_\xi(t)\,,
	\end{equation}
	where $b_\xi(t)$ is deterministic satisfying 
	\begin{equation} \label{3.99}
		b_\xi(t) \prec \frac{t^2}{N}\,, \quad \mbox{and} \quad \bb E|\cal E_\xi(t)| \prec \frac{t+1}{N}+\frac{\xi}{\sqrt{N}}
	\end{equation}
	uniformly for $t \in [0,N^{(1-c)/(2-\xi)}]$.
\end{pro}

Proposition \ref{lem. tf=0} leads to the following corollary.

\begin{cor}  \label{cor3.5}
	Fix $c>0$. There exists a constant $C\equiv C(f,c)>0$ such that 		
	\begin{equation}  \label{coe4.5-1}
		\Big|\bb E \varphi_1(t)\varphi_4(t)-\exp\Big(-\frac{\sigma_{f}^2t^2}{2}\Big)\Big|\leq CN^{-1+c}
	\end{equation}
	uniformly for all $t \in [0,N^{(1-c)/2}]$, and
	\begin{equation}  \label{coe4.5-2}
	\Big|\bb E\big[\varphi_1(t)\varphi_2(t)\varphi_3(t)\big]\Big| \leq CN^{-1+c}
	\end{equation}
	uniformly for all $t \in [N^{(1-c)/2},N^{1-c}]$.
\end{cor}

\begin{proof}
Let 
$
B_\xi(t)\deq \int_0^t b_\xi(u) \dd u,
$
and it is easy to see that 
\begin{equation} \label{3.9}
	B_\xi(t)-B_\xi(s)=O_{\prec}((t-s)t^2N^{-1})\,,   \quad 0 \leq s \leq t \leq N^{(1-c)/(2-\xi)}\,.
\end{equation} 
The ODE \eqref{ODE} is solved by
\begin{equation} \label{3.10}
	\begin{aligned}
		\phi_{\xi}(t) =&\,\phi_\xi (t_0) \exp\Big(-\frac{\sigma_f^2 (t^2-t_0^2)}{2}+B_\xi(t)-B_\xi(t_0)\Big)\\
		&+\int_{t_0}^t\exp\Big(  -\frac{\sigma_{f}^2(t^2-s^2)}{2}+B_\xi(t)-B_\xi(s)\Big)  \cal E_\xi(s) \dd s
	\end{aligned}
\end{equation}
for $0 \leq t_0 \leq t \leq N^{(1-c)/(2-\xi)}$.

First, we prove \eqref{coe4.5-1}. Note that when $\xi=0$, the error term $\cal E_0(t)$ in \eqref{ODE} is deterministic. Using \eqref{3.10} for $\xi=0$ and $t_0=0$, we have
\begin{equation*}
	\bb E \varphi_1(t)\varphi_4(t)=\phi_0(t) =\exp\Big(-\frac{\sigma_f^2 t^2}{2}+B_0(t)\Big)
	+\int_{0}^t\exp\Big(  -\frac{\sigma_{f}^2(t^2-s^2)}{2}+B_0(t)-B_0(s)\Big)  \cal E_0(s) \dd s\\
\end{equation*}
for $t \in [0,N^{(1-c)/2}]$. By \eqref{3.9} and $\sigma_f >c_*>0$, we have
\[
\exp\Big(-\frac{\sigma_f^2 t^2}{2}+B_0(t)\Big)=\exp\Big(-\frac{\sigma^2_ft^2}{2}\Big)+O_{\prec}(N^{-1})\,,
\]
and
\[
\exp\Big(  -\frac{\sigma_{f}^2(t^2-s^2)}{2}+B_0(t)-B_0(s)\Big) \leq \exp\Big(-\frac{\sigma^2_f(t^2-s^2)}{4}\Big)
\]
for $0 \leq s \leq t \leq N^{(1-c)/2}$.
Together with \eqref{3.99} we get
\begin{equation*}
	\bb E \varphi_1(t)\varphi_4(t)=\exp\Big(-\frac{\sigma^2_ft^2}{2}\Big)+O_{\prec}(N^{-1})+O_{\prec }(N^{-1})\cdot  \int_{0}^t\exp\Big(  -\frac{\sigma_{f}^2(t^2-s^2)}{4}\Big) (s+1)\,\dd s\,.
\end{equation*}
When $t \in [0, \log N]$, it is easy to see from the above that 
\begin{equation}
	\bb E \varphi_1(t)\varphi_4(t)=\exp\Big(-\frac{\sigma^2_ft^2}{2}\Big)+O_{\prec}(N^{-1})\,. \label{012106}
\end{equation}
When $t \in [\log N,N^{(1-c)/2}]$, we have
$
\exp\big( -\sigma_{f}^2(t^2-s^2)/4\big)=O(N^{-10})
$
for $s \in [0, t-(\log N)^2 t^{-1}]$. Thus
\begin{align}
	\bb E \varphi_1(t)\varphi_4(t)
	&=\exp\Big(-\frac{\sigma^2_ft^2}{2}\Big)+O_{\prec }(N^{-1})\cdot  \int_{t-(\log N)^2t^{-1}}^t\exp\Big(  -\frac{\sigma_{f}^2(t^2-s^2)}{4}\Big) (s+1) \,\dd s+O_{\prec}(N^{-1})\notag\\
	&=\exp\Big(-\frac{\sigma^2_ft^2}{2}\Big)+O_{\prec }(N^{-1})\cdot  (\log N)^2t^{-1}\cdot (t+1)+O_{\prec}(N^{-1})\notag\\
	&=\exp\Big(-\frac{\sigma^2_ft^2}{2}\Big)+O_{\prec}(N^{-1}) \label{012105}
\end{align}

This finishes the proof of  \eqref{coe4.5-1}.

Next, we prove \eqref{coe4.5-2}. Note that $|\varphi_3|, |\phi_{\xi}|\leq 1$. We have, using \eqref{3.10} with $\xi=1$ and $t_0=N^{1/2-3c/4}$ that
\begin{multline*}
	\Big|\bb E\big[\varphi_1(t)\varphi_2(t)\varphi_ 3(t)\big]\Big| \leq \bb E |\phi_1(t)| \\
	\leq\Big|\exp\Big(-\frac{\sigma_f^2 (t^2-t_0^2)}{2}+B_1(t)-B_1(t_0)\Big)\Big|+\int_{t_0}^t\Big|\exp\Big(  -\frac{\sigma_{f}^2(t^2-s^2)}{2}+B_1(t)-B_1(s)\Big)\Big|\bb E  |\cal E_1(s)| \dd s\,.
\end{multline*}
By \eqref{3.9}, $\sigma_f>c_*>0$, and $N^{1-c}\geq t \geq N^{(1-c)/2}=t_0N^{c/4}$, we have
\[
\Big|\exp\Big(-\frac{\sigma_f^2 (t^2-t_0^2)}{2}+B_1(t)-B_1(t_0)\Big)\Big|=O(N^{-10})\,.
\]
By \eqref{3.99},
$
\bb E|\cal E_1(s)| \prec \frac{s+1}{N}+\frac{1}{\sqrt{N}} \leq \frac{3s}{N^{1-3c/4}}
$
for $s\geq t_0$. Thus
\begin{equation*}
	\begin{aligned}
		&\quad\int_{t_0}^t\Big|\exp\Big(  -\frac{\sigma_{f}^2(t^2-s^2)}{2}+B_1(t)-B_1(s)\Big)\Big|\bb E  |\cal E_1(s)| \dd s \\
		&\prec \frac{1}{N^{1-3c/4}} \int_{t_0}^t\exp\Big(  -\frac{\sigma_{f}^2(t^2-s^2)}{4}\Big) s\,\dd s\\
		&=	O_{\prec}(N^{-10})+\frac{1}{N^{1-3c/4}} \int_{t-(\log N)^2t^{-1}}^t\exp\Big(  -\frac{\sigma_{f}^2(t^2-s^2)}{4}\Big) s\,\dd s
		\prec N^{-1+3c/4}\,.
	\end{aligned}
\end{equation*}
The above three relations imply the desired result \eqref{coe4.5-2}.
\end{proof}

Now, with the above estimates on characteristic functions, we can prove Theorem \ref{thmmain} in the sequel. 

\begin{proof}[Proof of Theorem \ref{thmmain}]	
	Let us denote
	\[
	{\psi}_f(t) \deq \bb E  \exp( \ii t \sigma_{f,\gamma} \widehat{{\cal Z}}_{f,\gamma})\,. 
	\]
	Setting $c=\kappa/2$ in Lemma \ref{lem1558} and \eqref{coe4.5-1}, we have
	\begin{equation} \label{3.2}
		\begin{aligned}
			&\ \  \ \ \Big|{\psi}_f(t)-\exp\Big(-\frac{\sigma_{f,\gamma}^2t^2}{2}\Big)\Big|=\Big|\bb E[\varphi_1(t)\bb E_{\rm d}[ \varphi_2(t)\varphi_3(t)]]-\exp\Big(-\frac{\sigma_{f,\gamma}^2t^2}{2}\Big)\Big|\\
			&\leq \Big|\bb E[\varphi_1(t)\varphi_4(t)]\exp\Big(-\frac{a_2(1-\gamma)^2(c_1^f)^2t^2}{8}\Big)-\exp\Big(-\frac{\sigma_{f,\gamma}^2t^2}{2}\Big)\Big|\\%
			&\ \ \ + \bb E\bigg|\bb E_{\rm d}\big[\varphi_{2}(t)\varphi_{3}(t)\big]-\exp\Big(-\frac{a_2(1-\gamma)^2(c_1^f)^2t^2}{8}\Big)\varphi_4(t)\bigg|\leq  C(N^{-1+\kappa/2}+ \cal X N^{-1/2+\kappa/2})
		\end{aligned} 
	\end{equation}
	uniformly for all $t \in [0,N^{1/2-\kappa/4}]$. By \eqref{coe4.5-2} and $\sigma_f>c_*>0$, we have
	\begin{equation} \label{maomao}
		\Big|{\psi}_f(t)-\exp\Big(-\frac{\sigma_{f,\gamma}^2t^2}{2}\Big)\Big| \leq C N^{-1+\kappa/2}	
	\end{equation} 
	uniformly for $t \in [N^{1/2-\kappa/4},N^{1-\kappa/2}]$. By the boundness of $f$ and Definition \ref{def:Wigner}  we have
	\begin{equation} \label{3.3}
		\Big|{\psi}_f(t)-\exp\Big(-\frac{\sigma_{f,\gamma}^2t^2}{2}\Big)\Big|\leq |{\psi}_f(t)-1|+\Big|\exp\Big(-\frac{\sigma_{f,\gamma}^2t^2}{2}\Big)-1\Big| \leq CtN
	\end{equation}
	uniformly in $N$ and $t \in[0,N^{-2}]$. Let us apply Lemma \ref{lem:essen} for $Y=\sigma_{f,\gamma} \widehat{\cal Z}_{f,\gamma}$ and  $T=N^{1-\kappa/2}$. We have
	\begin{multline} \label{34}
		\Delta(\widehat{\cal Z}_{f,\gamma},Z) \leq C \int_0^{N^{1-\kappa/2}} \frac{|{\psi}_f(t\sigma_{f,\gamma}^{-1})-\exp(-t^2/2)|}{t} \,\dd t +CN^{-1+\kappa/2}\\
		=C \int_0^{N^{-10}} \frac{|{\psi}_f(t\sigma_{f,\gamma}^{-1})-\exp(-t^2/2)|}{t} \,\dd t +C \int_{N^{-10}}^{N^{1-\kappa/2}} \frac{|{\psi}_f(t\sigma_{f,\gamma}^{-1})-\exp(-t^2/2)|}{t} \,\dd t\\ +CN^{-1+\kappa/2}
		\leq  C(N^{-1+\kappa}+ \cal X N^{-1/2+\kappa})\,,
	\end{multline}
	where in the last step we used \eqref{3.2} -- \eqref{3.3}. Combining \eqref{brick} and \eqref{34} we finish the proof.
\end{proof}

\vspace{2em}

\section{Proof of Proposition \ref{lem. tf=0}} \label{sec4}

For the rest of this paper we shall always assume $H$ is a real symmetric Wigner matrix. Using the complex cumulant expansion formula \cite[Lemma 7.1]{HK}, our argument can be easily extended to the complex Hermitian case, and for conciseness we shall omit the details.

In Sections \ref{sec4.1} -- \ref{sec4.6} we prove \eqref{ODE} for $\xi=1$, i.e. computing
$
\phi_1(t)=\bb E_{\rm o}\varphi_1(t)\varphi_2(t);
$ 
the case $\xi=0$ will be proved in  \ref{sec4.4}. In Sections \ref{sec4} and  \ref{sec5}, with certain abuse of notations, we will simply write
\[
\bb E \equiv \bb E_{\rm o}\,, \quad \langle X\rangle \equiv X-\bb E_{\rm o} X 
\]
and we denote
$\zeta(t)\deq \varphi_1(t)\varphi_2(t), \widehat{\zeta}(t)\deq \varphi_1(t)\varphi_4(t).
$
Further, set
\begin{align}
	s_n\deq \cal C_n(\sqrt{N}H_{12})  \quad \mbox{and} \quad a_n\deq  \bb E (\sqrt{N}H_{11})^n \label{022701}
\end{align}
for all fixed $n \geq 3$. Note that this definition is coherent with Definition \ref{def:Wigner}. For $n \in \bb N_+$, we shall use the integration operator $\{\cdot\}_n\equiv \{\cdot\}_{n,f}$, defined by
\begin{equation} \label{int f}
	\{g\}_n\equiv \{g(z_1,...,z_n)\}_n\deq \frac{1}{\pi^n}\int_{\b D^n} \Big(\frac{\partial}{\partial \bar{z}_1} \tilde{f}(z_1)\Big)\cdots\Big(\frac{\partial}{\partial \bar{z}_n} \tilde{f}(z_n)\Big) g(z_1,z_2,...,z_n) \dd^2 z_1\cdots \dd z_n^2
\end{equation}
for $g:\bb C ^n\to \bb C$. We also abbreviate $\{\cdot\} \equiv \{\cdot \}_1$.  Note that we also conventionally use $\{\cdot\}$ to group elements of a generic set. But this abuse of notation is harmless since it is easy to tell the difference between a set and an integral in the following context.  The following is an elementary estimate, which is essentially due to the fact that $f \in C^5(\bb R)$. 

\begin{lem} \label{lem:integration}
	Suppose $g:\bb C \to \bb C$ satisfies
	$
	g(z) \prec |\eta|^{-5}
	$
	uniformly for $z=E+\ii \eta \in \b D$, then we have
	$
	\{g\} \prec 1.
	$
\end{lem}

\subsection{The first expansion} \label{sec4.1}
We start with
\[
\phi_1'(t)=\frac{\ii }{\pi} \int_{\b D} \frac{\partial }{\partial \bar{z}} \tilde{f}(z)\bb E \big[\big(\big\langle\text{Tr} \widehat{G}(z) \rangle -\sum_i\langle (\widehat{G}^2(z))_{ii} \rangle H_{ii}\big)\zeta(t)\big] {\rm d}^2 z\,.
\]
It is easy to see from Lemma \ref{lem0diag} that
\[
\sum_{i}  \langle (\widehat{G}^2(z))_{ii} \rangle H_{ii} \prec \frac{1}{\sqrt{N|\eta|^3}}
\]
uniformly for $z \in \b D$, and together with Lemma \ref{lem:integration} we get
\[
\phi_1'(t)=\ii \bb E \{\langle \tr \widehat{G} \rangle\} \zeta(t)+O_{\prec}(N^{-1/2})\,.
\]
It suffices to check the first term on RHS of the above. By the resolvent identity
$
z\widehat{G}(z)=\widehat{H}\widehat{G}-I,
$
we have
\[
z\bb E \langle \tr \widehat{G} \rangle \zeta(t)={\sum_{i,j
}}^{\hspace{-0cm}*}\bb E H_{ij}\widehat{G}_{ji} \langle \zeta(t) \rangle \,,
\]
where we recall the notation of distinct summation ${\sum}^*$ from \eqref{sum*}. We compute the RHS of the above using Lemma \ref{lem:cumulant_expansion}, with $h=H_{ij}$ and $\cal F(h)\equiv \cal F_{ij}(\widehat{H})=\widehat{G}_{ji} \langle \zeta(t) \rangle$, which leads to
\begin{equation*} 
	\begin{aligned} 
		z\bb E \langle \tr \widehat{G} \rangle \zeta(t)
		&=\frac{1}{N} {\sum_{i,j
		}}^{\hspace{-0cm}*} \bb E  \frac{\partial \widehat{G}_{ji}}{\partial H_{ij}} \langle \zeta(t) \rangle  +\frac{1}{N} {\sum_{i,j
		}}^{\hspace{-0cm}*}\bb E\widehat{G}_{ji}\frac{\partial \langle \zeta(t)\rangle}{\partial H_{ij}}+
		\sum_{k=2}^{\ell}\bb E L_k +{\sum_{i,j}}^*\bb E \cal  R^{(ji)}_{\ell+1}\\
		&\ \eqd  (a)+(b)+ \sum_{k=2}^\ell  \bb EL_k+{\sum_{i,j}}^*\bb E\cal R^{(ji)}_{\ell+1}\,,
	\end{aligned} 
\end{equation*}
where
\begin{equation} \label{3.14} 
	L_k=\frac{s_{k+1}}{k!}\frac{1}{N^{(1+k)/2}}  {\sum_{i,j
	}}^{\hspace{-0cm}*} \frac{\partial^k( \widehat{G}_{ji} \langle \zeta(t) \rangle)}{\partial H_{ij}^k}\,.  
\end{equation}
Here $l$ is a fixed positive integer to be chosen later, and $\cal R_{l+1}^{(ji)}$ is a remainder term defined analogously to $\cal R_{l+1}$ in (\ref{eq:cumulant_expansion}). 
Using the differential rule 
\begin{equation} \label{diff}
	\frac{\partial \widehat{G}_{ab}}{\partial H_{ij}}=-\widehat{G}_{ai}\widehat{G}_{jb}-\widehat{G}_{aj}\widehat{G}_{ib}\,, \quad  i \ne j\,,
\end{equation}
we get
\begin{equation*} 
	\begin{aligned}
		(a)&= N^{-1} {\sum_{i,j}}^* \bb E (-\widehat{G}_{ij}\widehat{G}_{ij}-\widehat{G}_{ii}\widehat{G}_{jj})\langle \zeta(t)\rangle\\
		&=N^{-1} \sum_{i,j} \bb E (-\widehat{G}_{ij}\widehat{G}_{ij}-\widehat{G}_{ii}\widehat{G}_{jj})\langle \zeta(t)\rangle+2N^{-1}\sum_i \bb E \langle \widehat{G}_{ii}^2 \rangle\zeta(t)\\
		&= -\bb E \langle\underline{\widehat{G}^2} \rangle \zeta(t) - N\bb E  \langle \underline{\widehat{G}} \rangle^2\zeta(t)-2\bb E \underline{\widehat{G}}\bb E \langle \tr {\widehat{G}} \rangle \zeta(t)+ N\bb E \zeta(t)\bb E \langle \underline{\widehat{G}} \rangle^2+2N^{-1}\sum_i \bb E \langle \widehat{G}_{ii}^2 \rangle\zeta(t)\,,
	\end{aligned} 
\end{equation*} 
and
\begin{equation*} 
	(b)=-\frac{2\ii t}{N}{\sum_{i,j}}^*\bb E\widehat{G}_{ij}\{(\widehat{G}^2)_{ij}\} \zeta(t)+\frac{4\ii t}{N}{\sum_{i,j}}^*\sum_k\bb E\widehat{G}_{ij}\{(\widehat{G}^2)_{ik}\widehat{G}_{kj}\} H_{kk}\zeta(t)\,.
\end{equation*}
Altogether we obtain
\begin{equation}  \label{maingreen}
	\begin{aligned}
		\phi'_1(t) &= \ii\bb E \{T\langle\underline{\widehat{G}^2} \rangle\} \zeta(t) +\ii N\bb E \{T \langle \underline{\widehat{G}} \rangle^2\}\zeta(t)
		-\ii  N\bb E\{ T\langle \underline{\widehat{G}} \rangle^2\}\bb E \zeta(t) \\
		&\ \ +2\ii N^{-1}\sum_i \bb E \{T\langle \widehat{G}_{ii}^2 \rangle\}\zeta(t)
		-\frac{2 t}{N}{\sum_{i,j}}^*\bb E\{T\widehat{G}_{ij}\}\{(\widehat{G}^2)_{ij}\} \zeta(t)+O_{\prec}(N^{-1/2})\\
		&\ \ +\frac{4 t}{N}{\sum_{i,j}}^*\sum_k\bb E\{T\widehat{G}_{ij}\}\{(\widehat{G}^2)_{ik}\widehat{G}_{kj}\} H_{kk}\zeta(t)-\ii\sum_{k=2}^{\ell}\bb E\{TL_{k}\}-\ii {\sum_{i,j}}^*\bb E \{T\cal R^{(ji)}_{\ell+1}\}\,,
	\end{aligned}
\end{equation}
where 
$
	T\equiv T(z)\deq (-z-2\bb E \ul{\widehat{G}})^{-1}\,.
$
It is easy to check, from Theorem \ref{refthm1} that
\begin{equation} \label{T}
	|T(z)|=O(|\eta|^{-1/2})
\end{equation}
uniformly for $z =E+\ii \eta\in \b D$. Let us estimate the terms in \eqref{maingreen}. By Lemma \ref{prop4.4} and \eqref{T}, we see that
\[
T\langle \ul{\widehat{G}^2} \rangle  \prec \frac{1}{N|\eta|^{5/2}}\,.
\]
Since $|\zeta(t)|=|\varphi_1(t)\varphi_{2}(t)|\leq 1$, we can use Lemma \ref{lem:integration} to show that
\begin{equation} \label{4.7}
	\bb E \{T\langle\underline{\widehat{G}^2} \rangle\} \zeta(t) \prec N^{-1}\,.
\end{equation}
Similarly, Theorem \ref{refthm1}  and \eqref{T} imply $T\langle \underline{\widehat{G}} \rangle^2 \prec N^{-2}|\eta|^{-5/2}$, which leads to 
\begin{equation} \label{4.8}
	N\bb E \{T \langle \underline{\widehat{G}} \rangle^2\}\zeta(t)
	-N\bb E\{ T\langle \underline{\widehat{G}} \rangle^2\}\bb E \zeta(t)  \prec N^{-1}\,.
\end{equation}
From Lemma \ref{prop4.4} we also know that
\[
\sum_i \langle \widehat{G}_{ii}^2 \rangle =\sum_i \widehat{G}_{ii}^2 -Nm(z)^2+ \bb E\Big(Nm(z)^2-\sum_i \widehat{G}_{ii}^2 \Big) \prec |\eta|^{-1}\,,
\]
and thus
\begin{equation} \label{4.9}
	2N^{-1}\sum_i \bb E \{\langle \widehat{G}_{ii}^2 \rangle \}\zeta(t) \prec N^{-1}\,.
\end{equation}
Let us abbreviate $\widehat{G}\deq \widehat{G}(z_1)$, $T\equiv T(z_1)$ and $\widehat{F}\equiv G(z_2)$. We see that
\begin{equation}  \label{g_1}
	\begin{aligned} 
		&{\sum_{i,j}}^*\bb ET\widehat{G}_{ij}(\widehat{F}^2)_{ij} \zeta(t)=T \sum_{i,j} \bb E \widehat{G}_{ij}(\widehat{F}^2)_{ij} \zeta(t)-T \sum_{i} \bb E  \widehat{G}_{ii}(\widehat{F}^2)_{ii} \zeta(t)\\
		=&\,T \bb E\tr {\widehat{G}\widehat{F}^2}\bb E \zeta(t)-NTm(z_1)m'(z_2)\bb E \zeta(t)+O_{\prec}\Big(\frac{1}{|\eta_1\eta_2|^2}\Big)\\
		=&\,\bigg(-\frac{N}{z_1+2m(z_1)} \Big(\partial_{z_2}\frac{m(z_1)-m(z_2)}{z_1-z_2} \Big)-Nm'(z_1)m'(z_2)\bigg)\bb E \zeta(t)+O_{\prec}\Big(\frac{1}{|\eta_1\eta_2|^2}\Big)\\
		=&: Ng_1(z_1,z_2) \bb E \zeta(t)  +O_{\prec}\Big(\frac{1}{|\eta_1\eta_2|^2}\Big)\,,
	\end{aligned} 
\end{equation}
where in the second step we used Lemma \ref{prop4.4}, while in the third step we used Theorem \ref{refthm1} and the basic fact $z+2m(z)=-m(z)/m'(z)$. Hence
\begin{equation} \label{4.11}
	\begin{aligned}
		&\,-\frac{2 t}{N}{\sum_{i,j}}^*\bb E\{T\widehat{G}_{ij}\}\{(\widehat{G}^2)_{ij}\} \zeta(t)	\\
		=&\,-\frac{2t}{\pi^2} \int_{\b D^2} \Big(\frac{\partial}{\partial \bar{z}_1} \tilde{f}(z_1)\Big)\Big(\frac{\partial}{\partial \bar{z}_2} \tilde{f}(z_2)\Big)g_1(z_1,z_2)\dd^2 z_1 \dd^2 z_2 \,\bb E \zeta(t)+O_{\prec}(tN^{-1})\,.
	\end{aligned}
\end{equation}
By Lemma \ref{prop4.4}, we also have
\begin{equation*}
	\begin{aligned}
		&{\sum_{i,j}}^*\sum_k\bb ET\widehat{G}_{ij}(\widehat{F}^2)_{ik}\widehat{F}_{kj} H_{kk}\zeta(t)\\
		=&\ T\sum_k \bb E (\widehat{G}\widehat{F}^3)_{kk}H_{kk}\zeta(t)-T{\sum_{i,k}}^*\bb E\widehat{G}_{ii}(\widehat{F}^2)_{ik}\widehat{F}_{ki} H_{kk}\zeta(t)-T\sum_{k}\bb E\widehat{G}_{kk}(\widehat{F}^2)_{kk}\widehat{F}_{kk} H_{kk}\zeta(t)\,,
	\end{aligned}
\end{equation*}
which can be estimated by $O_{\prec}\Big(\frac{1}{|\eta_1\eta_2^2|}\Big)$. As a result
\begin{equation} \label{kkk}
	\frac{4 t}{N}{\sum_{i,j}}^*\sum_k\bb E\{T\widehat{G}_{ij}\}\{(\widehat{G}^2)_{ik}\widehat{G}_{kj}\} H_{kk}\zeta(t)=O_{\prec}(tN^{-1})\,.
\end{equation}
The estimate for the remainder term can be done routinely. One can follow, e.g.\,the proof of Lemma 3.4 (iii) in \cite{HKR}, and readily check that 
\begin{equation} \label{RRR}
	{\sum_{i,j}}^*\bb E\{T\cal R^{(ji)}_{\ell+1}\} \prec N^{-1}
\end{equation}
for some fixed (large) $\ell\in \bb N_+$. From now on, we shall always assume the remainder term in cumulant expansion is negligible. Inserting \eqref{4.7} -- \eqref{4.9},  \eqref{4.11} -- \eqref{RRR} into \eqref{maingreen}, we have
\begin{equation} \label{4.14}
	\begin{aligned}
		\phi'_1(t) 
		=&\,-\frac{2t}{\pi^2} \int_{\b D^2} \Big(\frac{\partial}{\partial \bar{z_1}} \tilde{f}(z_1)\Big)\Big(\frac{\partial}{\partial \bar{z_2}} \tilde{f}(z_2)\Big)g_1(z_1,z_2)\dd^2 z_1 \dd^2 z_2 \,\phi_1(t)\\
		&\,-\ii\sum_{k=2}^{\ell}\bb E\{TL_{k}\}+O_{\prec}(tN^{-1})+O_{\prec}(N^{-1/2})\,.
	\end{aligned}
\end{equation}
Therefore, what remains is the analysis of $\bb E \{TL_k\}$, $k \ge 2$. This requires very precise preliminary bounds. To this end, we introduce the notion of abstract polynomials.

\subsection{Abstract polynomials} \label{sec8}
We now define a notion of formal monomials in a set of formal variables.  Here the word \emph{formal} refers to the fact that these definitions are purely algebraic and we do not assign any values to variables or monomials. We start with the definition of the variables in our monomials.

\begin{defn} \label{def:cal_V}
	Let $\cal I_*=\{i_1,i_2,\dots\}$ be an infinite set of formal indices. For $z \in \b D$, we define
	\begin{equation*}
		\begin{aligned}
			\b G(z)\deq & \big\{T^\alpha(z)(\widehat{G}(z)H_{\mathrm d}\widehat{G}^{1+\omega}(z))_{x_0y_0}^d\big(\partial_z^{\delta}\widehat{G}_{x_1y_1}(z)\big)\widehat{G}_{x_2y_2}(z)\cdots \widehat{G}_{x_ny_n}(z):\\
			&\quad  \alpha,\omega,d,\delta\in \{0,1\}, \alpha+\omega+\delta\leq 1, n \in \bb N_+, x_0,y_0,...,x_n,y_n \in \cal I_*\big\}\,,
		\end{aligned}
	\end{equation*}
	and $\b G\deq \cup_{z\in \b D} \b G(z)$. For 
	$\cal G=T^\alpha(\widehat{G}H_{\mathrm d}\widehat{G}^{1+\omega})_{x_0y_0}^d\big(\partial_z^{\delta}\widehat{G}_{x_1y_1}\big)\widehat{G}_{x_2y_2}\cdots \widehat{G}_{x_ny_n} \in \b G
	$, we denote
	\[
	\cal I(\cal G)=\{x_0,y_0,...,x_n,y_n\} \quad \mbox{and}\quad d(\cal  G)=d\,.
	\]
	The collection of off-diagonal indices of $\cal G$ is denoted by
	\[
	\cal I_0(\cal G)=\{x_j,y_j: 0\leq j\leq n, x_j\ne y_j\}\,.
	\]	
	Accordingly, when $x_i\ne y_i$, the corresponding $(\widehat{G}(z)H_{\mathrm d}\widehat{G}^{1+\omega}(z))_{x_iy_i}$, $\partial_z^{\delta}\widehat{G}_{x_iy_i}(z)$, or $\widehat{G}_{x_iy_i}$ is called an \emph{off-diagonal} factor in $\cal G$.
\end{defn}

Now we define the monomials we use.

\begin{defn} \label{defn4.3}

	To $\sigma,\mu \in \N_+$, $a \in \bb C$, and $\theta \in \R$
	we assign a formal monomial
	\begin{equation} \label{def_T}
		\cal	P =  a t^\mu N^{-\theta}\{\cal G_1\}\cdots \{\cal G_\sigma\}\,,
	\end{equation}
	where $\cal G_1\equiv \cal G(z_1)\in \b G(z_1),...,\cal G_\sigma \equiv \cal G(z_\sigma) \in \b G(z_\sigma)$. We denote $\sigma(\cal P) = \sigma$, $\mu(\cal P)=\mu$, $\theta(\cal P) = \theta$, and
	\[
	\cal I(\cal P)\deq \cal I(\cal G_1)\cup \cdots \cup \cal I(\cal G_\sigma)\,, \quad  	\nu(\cal P)\deq \big|\cal I(\cal P)\big|\,, 
	\] 
	as well as
	\[
	d(\cal P)\deq d(\cal G_1)+\cdots +d(\cal G_\sigma)\,.
	\]
	The collection of off-diagonal indices of $\cal P$ is denoted by
	\[
	\cal I_0(\cal P)\deq \cal I_0(\cal G_1)\cup \cdots \cup \cal I_0(\cal G_\sigma) \,,
	\]
	and accordingly the definition of \emph{off-diagonal} factors is also naturally extended to $\cal P$. In addition, we define 
	\begin{equation} \label{def_TT}
		\cal	{\mathring P} =a  t^\mu N^{-\theta} \langle \{\cal G_1\}\cdots \{\cal G_\sigma\} \rangle\quad \mbox{and} \quad \widetilde{\cal P}=at^{\mu}N^{-\theta} \cal G_1\cdots \cal G_\sigma\,.
	\end{equation}
	We denote by $\b P$ the set of formal monomials $\cal P$ of the form \eqref{def_T}, and denote by ${\mathring{ \b P}}$ the set of formal monomials $\cal	{\mathring P} $ of the form \eqref{def_TT}.	
\end{defn}

The next definition concerns the \textit{evaluation} of $\cal P$.

\begin{defn} \label{def:evaluation}
	(i) For each monomial $\cal P \in  \b P$ with $\nu = \nu(\cal P)$, its \emph{evaluation} is a random variable depending on an $\nu$-tuple $(i_1,\dots,i_{\nu})\in \{1,2,\dots,N\}^{\nu}$. It is obtained by replacing, in the formal monomial $\cal P$, the formal indices $i_1,\dots,i_{\nu}$ with the integers $i_1,\dots,i_{\nu}$ and the formal variables $\widehat{G}, H_{\mathrm d}$ with the random variables defined in \eqref{def_hatG} and  \eqref{HHH}. The evaluation of $\cal	{\mathring P}  \in {\mathring{ \b P}}$ is defined accordingly.
	
	(ii) Let $\cal P \in \b P$, and set
	\[
	\cal I_2(\cal P)\deq \{i \in \cal I(\cal P ): i \mbox{ appears twice in }\cal P\}, \quad \mbox{and} \quad \nu_2(\cal P)=|\cal I_2(\cal P)|\,.
	\]
	W.O.L.G., let us assume $\cal I_2(\cal P)=\{i_1,...,i_{\nu_2}\}$ for some nonnegative integer $\nu_2\leq \nu_1$. We define the sums
	\begin{equation*} 
		\cal S_2 (\cal P) \deq \sum_{i_1,\dots,i_{\nu_2}}  \cal P\,, \quad
		\mbox{and}
		\quad 
		\cal S(\cal P)\deq {\sum_{i_{\nu_2+1},...,i_{\nu}}}^{\hspace{-0.4cm}*} \cal S_2(\cal P)\,.
	\end{equation*}
	The definitions of $\cal S(\cdot)$ and $\cal S_2(\cdot)$ can be extended to any family of random variables that is labeled with indices in $\cal I_*$, in particular for $\mathring {\cal P}$ and $\widetilde{\cal P}$.
	
\end{defn}

Observe that, after summing  over the indices $i_1,..., i_{\nu_2}\in \mathcal{I}_2(\mathcal{P})$, $\mathcal{S}_2(\mathcal{P})$ is again a monomial, where each variable might now contain several Green functions $\widehat{G}(z)$ at different $z \in \b D$; accordingly, we extend the term \emph{off-diagonal factors} to include those entries of the form $(W(H_{\mathrm d},G))_{xy}$ with $x\neq y$, where $W(H_{\mathrm d},G)$ can be any word in $H_{\mathrm d}$ and $\widehat{G}(z_1),\ldots, \widehat{G}(z_\sigma)$ or their derivatives w.r.t. $z_i$'s. Next, we define parameters that characterize smallness in our estimates of $\cal S(\cal P)$.

\begin{defn} \label{defn4.5}
	Let $\cal P \in \b P$. We use $\nu_1(\cal P)$ to denote number of off-diagonal factors in $\cal S_2(\cal P)$, and $\nu_{3,0}(\cal P)$ denotes the number of traces in $\cal S_2(\cal P)$. We also set
	\[
	\nu_3(\cal P)\deq |\cal I_2(\cal P)|-\nu_{3,0}(\cal P)\,,\nc \quad   \cal I_1(\cal P) \deq \cal I(\cal P) \backslash \cal I_2(\cal P)\,. \
	\]
	In addition, for a factor 
	$\cal G=T^\alpha(\widehat{G}H_{\mathrm d}\widehat{G}^{1+\omega})_{x_0y_0}^d\big(\partial_z^{\delta}\widehat{G}_{x_1y_1}\big)\widehat{G}_{x_2y_2}\cdots \widehat{G}_{x_ny_n} \in \b G$ of $\widetilde{\cal P}$, we use $\nu_1(\cal G)$ to denote number of off-diagonal factors in $\cal S_2(\cal P)$ that only contain terms in $\cal G$, and $\nu_{3,0} (\cal G)$ denotes the number of traces in $\cal S_2(\cal P)$ that only contain terms in $\cal G$. We further set
	\[
	\nu_3(\cal G)\deq |\cal I(\cal G)\cap \cal I_2(\cal P)|-\nu_{3,0}(\cal G)
	\]
	and
	\[
	\nu_0(\cal G)\deq {\nu}_1(\cal G)+2\nu_3(\cal G)+\alpha+2\delta+(1+2\omega)d\,.
	\]
	Finally, we set
	\[
	\nu_0(\cal P)\deq \nu_0(\cal G_1)\vee\cdots\vee\nu_0(\cal G_\sigma)\,, \quad \nu_*(\cal P)\deq (\nu_0(\cal G_1)-10)_+ +\cdots + (\nu_0(\cal G_\sigma)-10)_+\,,
	\]
	and note that $\nu_*(\cal P)=0$ for $\nu_0(\cal P)\leq 10$.
\end{defn}

\begin{example} \label{example 4.6}
	Let us consider
	\[
	\cal P\deq6 t^{10} N^{-2}\{\cal G_1\}\{\cal G_2\}\{\cal G_3\}\deq6 t^{10} N^{-2}\{T\widehat{G}_{ij}\widehat{G}_{kx}\}\{(\partial_z\widehat{G}_{ll})\widehat{G}_{xj}\widehat{G}_{ik}\}\{(\widehat{G}H_{\mathrm d}\widehat{G}^2)_{kj}\widehat{G}_{iy}\widehat{G}_{jk}^4\}\,.
	\]
	
	(i) Let us first illustrate Definitions \ref{def:cal_V} and \ref{defn4.3}. We have $\sigma(\cal P)=3$, $\mu(\cal P)=10$, $\theta(\cal P)=2$, $\cal I(\cal P)=\{i,j,k,l,x,y\}$, $\nu(\cal P)=6$, $d(\cal P)=1$. We also see that $\cal I_0(\cal P)=\{i,j,k,x,y\}$.

	(ii) Regarding Definition \ref{def:evaluation}, we have $\cal I_2(\cal P)=\{l,x\}$, $\nu_2(\cal P)=2$, and
	\begin{align*}
	&\cal S_2(\cal P)=\sum_{l,x}\cal P\\
	&=6t^{10}N^{-2}\{T(z_1)\widehat{G}_{ij}(z_1)(\widehat{G}(z_1)\widehat{G}(z_2))_{kj}(\partial_{z_2}\tr \widehat{G}(z_2))\widehat{G}_{ik}(z_2)\}_2\{(\widehat{G}H_{\mathrm d}\widehat{G}^2)_{kj}\widehat{G}_{iy}\widehat{G}_{jk}^4\} \,.
	\end{align*}

	(iii) Regarding Definition \ref{defn4.5}, we see that $\nu_1(\cal P)=3+6=9$, $\cal I_2(\cal P )=\{x,l\}$, $\nu_{3,0}(\cal P)=1$, $\nu_3(\cal P)=2-1=1$, $\cal I_1(\cal P)=\{i,j,k,y\}$. In addition, $\nu_3(\cal G_1)=|\{x\}|-0=1$, and $\nu_0(\cal G_1 )=1+2+1+0+0=4$. Similarly, $\nu_0(\cal G_2)=1+2+0+2+0=5$ and $\nu_0(\cal G_3)=6+0+0+0+3=9$. As a result, $\nu_0(\cal P)=4\vee5\vee 9=9$, $\nu_*(\cal P)=0$.
\end{example}

\begin{rem}
	Let us illustrate how the above parameters determine the size of $\cal S(P)$. First
	of all, trivially, there are $N^{\cal I_2(\cal P)}$ terms in the sum over indices in $\cal I_2(\cal P)$, but summing over these indices
	gains an $N^{-\nu_{3}(
		\cal P)}$ improvement from the non-tracial quantities produced in this summation. Then, the
	$\nu_1(\cal P)$ off-diagonal entries in the resulting sum $\cal S_2(\cal P)$ further contributes an $N^{-\nu_1(\cal P)/2}$ factor thanks to
	Lemma \ref{prop4.4}. However, we shall also monitor the power of each $\eta_i^{-1}$ which is bounded by $\nu_0(\cal G_i)/2$, for $i \in \{1, . . . ,\sigma\}$. Notice that an $\eta_i^{-5}$ can be killed in the integral $\{\cal G_i\}$ due to Lemma \ref{lem:integration}. But any higher
	power of $\eta_i^{-1}$ will contribute additional factor to $\cal S (\cal P)$. This additional contribution from high power of $\eta_i$'s is then determined by $\nu_*(\cal P)$.
\end{rem}

The following is an elementary consequence of Lemmas \ref{lem0diag} and \ref{lemGHG}. 

\begin{lem} \label{lem:intermediate}
	For any $\cal P= a t^\mu N^{-\theta}\{\cal G_1\}\cdots \{\cal G_\sigma\} \in \b P$, we have
	\[
	\cal S(\widetilde{\cal P}) \prec t^{\mu(\cal P)}N^{\nu(\cal P)-\theta(\cal P)-\nu_1(\cal P)/2-\nu_3(\cal P)-d(\cal P)/2} |\eta_1|^{-\nu_0(\cal G_1)/2} \cdots |\eta_\sigma|^{-\nu_0(\cal G_\sigma)/2} \,.
	\]
\end{lem}
\begin{proof}
	Let us first consider the estimate of $\cal S_2(\widetilde{\cal P})$. The power of $t$ in the estimate is obviously $\mu(\cal P)$. The trivial power of $N$ in the estimate of $S_2(\widetilde{\cal P})$ is $\nu_2(\cal P)$, and it can be further improved by Lemmas \ref{lem0diag} and \ref{lemGHG}. In fact, through the definitions of $\nu_1(\cal P)$, $\nu_3(\cal P)$ and $d(\cal P)$, we see that they improve the power of $N$ in the estimate by $-\nu_1(\cal P)/2-\nu_3(\cal P)-d(\cal P)/2$.  Finally, again by Lemmas \ref{lem0diag} and \ref{lemGHG}, the power of $\eta_i$ in the estimate of $\cal S_2(\widetilde{\cal P})$ is $-\nu_0(\cal G_i)/2$. Hence we have
	\[
	\cal S_2(\widetilde{\cal P}) \prec t^{\mu(\cal P)}N^{\nu_2(\cal P)-\theta(\cal P)-\nu_1(\cal P)/2-\nu_3(\cal P)-d(\cal P)/2} |\eta_1|^{-\nu_0(\cal G_1)/2} \cdots |\eta_\sigma|^{-\nu_0(\cal G_\sigma)/2}\,,
	\]
	and together with
	\[
	\cal S(\widetilde{\cal P}) \prec N^{\nu(\cal P)-\nu_2(\cal P)} \big|\cal S_2(\widetilde{\cal P}) \big|
	\] 
	we get the desired result.
\end{proof}

For $\cal P \in \b P$, let
\begin{equation*}
	\cal E_0(\cal P)\deq t^{\mu(\cal P)} N^{\nu(\cal P)-\theta(\cal P)-\nu_1(\cal P)/2-\nu_3(\cal P)-d(\cal P)/2+\nu_*(\cal P)/8}
\end{equation*}
and
\begin{equation*}
	\cal E_*(\cal P) \deq t^{\mu(\cal P)} N^{\nu(\cal P)-\theta(\cal P)-\nu_1(\cal P)/2-\nu_3(\cal P)-1}\,.
\end{equation*}

By Lemmas \ref{lem:integration} and \ref{lem:intermediate}, together with the fact that $|\eta| \geq N^{-1/4}$ for $z=E+\ii \eta \in \b D$, we have the following estimate.

\begin{lem} \label{lem4.5}
	For any $\cal P \in \b P$, we have
	\[
	\bb E \,\cal S(\cal P)=O_{\prec}\big( \cal E_0(\cal P)\big)\,.
	\]
	In addition, for any complex random variable $Y$ satisfying $|Y|\leq 1$, we have
	\begin{equation*} 
		\bb E \,\cal S(\cal P)Y=O_{\prec}\big( \cal E_0(\cal P)\big)\cdot \bb E Y +\bb E\cal S (\cal	{\mathring P} )Y\,,
	\end{equation*}
	where
	$
	\bb E\,\cal S (\cal	{\mathring P} )Y= O_{\prec}\big( \cal E_0(\cal P)\big)\,.
	$
\end{lem}

We have the following improved estimate for a special class of $\cal P$.

\begin{lem} \label{lem5.7}
	Let $t \in [0,N^{1-c}]$ for some fixed $c>0$, and let
	$\cal P =a t^\mu N^{-\theta}\{\cal G_1\}\cdots \{\cal G_\sigma\}\in \b P$ with $\nu(\cal P)=2$, $\nu_2(\cal P)=\nu_3(\cal P)=0$, $d(\cal P)\leq 1$, $\nu_0(\cal P)\leq 3+d(\cal P)$. Moreover, each $\cal G_i$, $i=1,...,\sigma$ contains at most one off-diagonal factor. For any complex random variable $Y$ satisfying $|Y|\leq 1$, we have the following estimates.

	(i) When $d(\cal P)=0$, we have
	\begin{equation*}  
		\bb E\,\cal S (\cal	{\mathring P} )Y=O_{\prec}\big(\cal E_*(\cal P) \big)\,.
	\end{equation*}
Moreover, under the additional assumption $\nu_1(\cal {P)}=1$, we have the stronger estimate
\begin{equation*}  
	\bb E\,\cal S (\cal	 P )Y=O_{\prec}\big(\cal E_*(\cal P) \big)\,.
\end{equation*}
	
	(ii) When $d(\cal P)=1$, we have
	\begin{equation*}  
		\bb E\,\cal S (\cal	 P )Y=O_{\prec}\big(\cal E_*(\cal P) \big)\,.
	\end{equation*}
\end{lem}

\begin{proof}
	See Section \ref{sec5}.
\end{proof}

Armed with the above lemmas, we are now ready to deal with $\bb E\{TL_k\}$ in \eqref{4.14}.

\subsection{The estimate of $ \bb E\{TL_2\}$} \label{sec4.3}
By \eqref{3.14} and \eqref{diff}, we have
\begin{equation} \label{4.20}
	\begin{aligned}\,
		\bb E\{TL_2\}=&\,\frac{s_3}{2N^{3/2}}\cdot\bigg\{ T {\sum_{i,j
		}}^{\hspace{-0cm}*} \bb E\bigg(\frac{\partial^2( \widehat{G}_{ji} \langle\zeta(t) \rangle)}{\partial H_{ji}^2}\bigg)\bigg\}\\
		=&\,\frac{3s_3}{N^{3/2}}{\sum_{i,j}}^* \bb E \langle \{T\widehat{G}_{ii}\widehat{G}_{jj}\widehat{G}_{ij}\} \rangle \zeta(t) +\frac{s_3}{N^{3/2}}{\sum_{i,j}}^* \bb E \langle \{T\widehat{G}^3_{ij}\}\rangle \zeta(t)\\
		& +\frac{2 s_3\ii t}{N^{3/2}}{\sum_{i,j}}^* \bb E  \{T\widehat{G}_{ii}\widehat{G}_{jj}\}\{\partial_z\widehat{G}_{ij}\}\zeta(t)-\frac{4 s_3\ii t}{N^{3/2}}{\sum_{i,j}}^* \bb E  \{T\widehat{G}_{ii}\widehat{G}_{jj}\}\{(\widehat{G}H_{\mathrm d}\widehat{G}^2)_{ij}\}\zeta(t)\\
		&+\frac{2s_3\ii t}{N^{3/2}}{\sum_{i,j}}^* \bb E  \{T\widehat{G}_{ij}^2\}\{\partial_z\widehat{G}_{ij}\}\zeta(t)-\frac{4s_3\ii t}{N^{3/2}}{\sum_{i,j}}^* \bb E  \{T\widehat{G}_{ij}^2\}\{(\widehat{G}H_{\mathrm d}\widehat{G}^2)_{ij}\}\zeta(t)\\
		&+\frac{s_3}{2N^{3/2}}\cdot\bigg\{ T {\sum_{i,j
		}}^{\hspace{-0cm}*} \bb E\widehat{G}_{ji}\bigg(\frac{\partial^2\zeta(t) }{\partial H_{ji}^2}\bigg)\bigg\}\eqd  \sum_{p=1}^7\bb EL_{2,p}\,.
	\end{aligned}
\end{equation}
The terms on RHS of \eqref{4.20} can be estimated by Lemmas \ref{lem4.5} and \ref{lem5.7}. It is easy to see that $L_{2,1}=\cal S(\cal {\mathring P})\zeta(t)$, where $\mu(\cal P)=0$, $\nu(\cal P)=2$, $\theta(\cal P)=3/2$, $\nu_1(\cal P)=1$, $d(\cal P)=0$, $\nu_2(\cal P)=\nu_3(\cal P)=0$, and $\nu_0(\cal P)=2$. Thus we can apply Lemma \ref{lem5.7} and show that
\begin{equation*}
	\bb E L_{2,1}=O_{\prec}(N^{-1})\,.
\end{equation*}
In addition, $L_{2,4}=\cal S(\cal P)\zeta(t)$, where $\mu(\cal P)=1$, $\nu(\cal P)=2$, $\theta(\cal P)=3/2$, $\nu_1(\cal P)=1$, $d(\cal P)=1$, $\nu_2(\cal P)=\nu_3(\cal P)=0$, and $\nu_0(\cal P)=4$. Thus Lemma \ref{lem5.7} implies that 
$$
\bb E L_{2,4}=O_{\prec}(tN^{-1})
$$
By a similar argument, we can show that
$
\bb E L_{2,3} =O_\prec (tN^{-1}).
$
Using the trivial estimate from Lemma \ref{lem4.5}, one easily sees that
\[
\bb E L_{2,2}=O_{\prec}(N^{-1}) \quad \mbox{and}\quad \bb E L_{2,5}+\bb E L_{2,6}= O_{\prec}(tN^{-1})\,.
\]
By applying the differentials in $\bb EL_{2,7}$ carefully and using a similar argument as above, one can show that
\[
\quad \bb E L_{2,7}=O_{\prec}(t^2N^{-1})\cdot \phi_1(t)+O_{\prec}(tN^{-1}+t^2N^{-2})\,.
\]
Combining the above estimates, we arrive at
\begin{equation} \label{L_2}
	\bb E\{TL_2\}=O_{\prec}(t^2N^{-1})\cdot \phi_1(t)+O_\prec( (t+1)N^{-1})\,.
\end{equation}

\subsection{The estimate of $ \bb E\{TL_3\}$}
By \eqref{3.14} and \eqref{diff}, we have
\begin{equation*} 
	\begin{aligned}\,
		&\ \ \bb E\{TL_3\}=\,\frac{s_4}{6N^{2}}\bigg\{ T {\sum_{i,j
		}}^{\hspace{-0cm}*} \bb E\bigg(\frac{\partial^3( \widehat{G}_{ji} \langle\zeta(t) \rangle)}{\partial H_{ji}^3}\bigg)\bigg\}\\
		&=\,\frac{s_4}{6N^{2}}{\sum_{i,j
		}}^{\hspace{-0cm}*} \bb E\bigg(\bigg\{T \frac{\partial^3 \widehat{G}_{ji} }{\partial H_{ji}^3}\bigg\}\langle \zeta(t) \rangle\bigg) -\frac{\ii s_4t}{N^{2}}  {\sum_{i,j
		}}^{\hspace{-0cm}*} \bb E\bigg(\bigg\{ T\frac{\partial^2 \widehat{G}_{ji} }{\partial H_{ji}^2}\bigg\}\big\{(\partial_z\widehat{G}_{ij})-2(\widehat{G}H_{\mathrm d}\widehat{G}^2)_{ji}\big\} \zeta(t) \bigg) \\
		& -\frac{\ii s_4t}{N^{2}}  {\sum_{i,j
		}}^{\hspace{-0cm}*} \bb E\bigg(\bigg\{ T\frac{\partial \widehat{G}_{ji} }{\partial H_{ji}}\bigg\}\bigg\{\frac{\partial (\partial_z\widehat{G}_{ij})}{\partial H_{ji}}\bigg\} \zeta(t) \bigg)+\frac{2\ii s_4t}{N^{2}} {\sum_{i,j
		}}^{\hspace{-0cm}*} \bb E\bigg(\bigg\{ T\frac{\partial \widehat{G}_{ji} }{\partial H_{ji}}\bigg\}\bigg\{\frac{\partial (\widehat{G}H_{\mathrm d}\widehat{G}^2)_{ij}}{\partial H_{ji}}\bigg\} \zeta(t) \bigg) \\
		& -\frac{ 2s_4t^2}{N^{2}} {\sum_{i,j
		}}^{\hspace{-0cm}*} \bb E\bigg(\bigg\{T \frac{\partial \widehat{G}_{ji} }{\partial H_{ji}}\bigg\} \big\{(\partial_z\widehat{G}_{ij})-2(\widehat{G}H_{\mathrm d}\widehat{G}^2)_{ij}\big\}^2 \zeta(t) \bigg)\\
		& +\frac{s_4}{6N^{2}} {\sum_{i,j
		}}^{\hspace{-0cm}*} \bb E\bigg(\big\{T \widehat{G}_{ij}\big\}\frac{\partial^3 \zeta(t) }{\partial H_{ji}^3}\bigg)\eqd  \sum_{p=1}^6\bb EL_{3,p}\,.
	\end{aligned}
\end{equation*}
In the sequel, we estimate each term on the above separately. We claim, except for $\mathbb{E} L_{3,3}$ that contains the leading contribution, all the other terms are errors. In the sequel, we do not state the estimates for each single terms generated by the derivatives. Instead, for brevity, we pick those terms which have largest bounds by using Lemmas \ref{lem4.5} and \ref{lem5.7} and show the details of their estimates only. We will informally call these terms as the {\it worst} terms. The estimates of other terms are either similar or simpler, and thus we omit the details.

\textit{Step 1.} By \eqref{diff}, we see that the worst term in $\bb EL_{3,1}$ is of the form
\[
\bb EL_{3,1,1}\deq \frac{s_4}{N^2}{\sum_{i,j}}^*\bb E\{T\widehat{G}_{ii}^2\widehat{G}_{jj}^2\}\langle \zeta(t)\rangle=\frac{s_4}{N^2}{\sum_{i,j}}^*\bb E\langle\{T\widehat{G}_{ii}^2\widehat{G}_{jj}^2\}\rangle \zeta(t)\,.
\]
We see that $L_{3,1,1}=\cal S(\cal {\mathring P})\zeta(t)$, where $\mu(\cal P)=0$, $\nu(\cal P)=2$, $\theta(\cal P)=2$, $\nu_1(\cal P)=0$, $d(\cal P)=0$, $\nu_2(\cal P)=\nu_3(\cal P)=0$, $\nu_0(\cal P)=1$. By Lemma \ref{lem5.7}, we see that
$
\bb E L_{3,1,1}=O_{\prec}(N^{-1})\,.
$
By estimating other terms in $\bb EL_{3,1}$ with Lemma \ref{lem4.5}, we can show that
\begin{equation} \label{4.21}
	\bb E L_{3,1}=O_{\prec}(N^{-1})\,.
\end{equation}

\textit{Step 2.} By \eqref{diff}, we see that the worst term in $\bb EL_{3,2}$ is of the form
\[
\bb EL_{3,2,1}\deq -\frac{s_4\ii t}{N^2}{\sum_{i,j}}^*\bb E\{T\widehat{G}_{ii}\widehat{G}_{jj}\widehat{G}_{ij}\}\{(\partial_z\widehat{G}_{ij})\} \zeta(t)\,.
\]
We see that $L_{3,2,1}=\cal S(\cal {P})\zeta(t)$, where $\mu(\cal P)=1$, $\nu(\cal P)=2$, $\theta(\cal P)=2$, $\nu_1(\cal P)=2$, $d(\cal P)=0$, $\nu_2(\cal P)=\nu_3(\cal P)=0$, and $\nu_0(\cal P)=3$, $\nu_*(\cal P)=0$. By Lemma \ref{lem4.5}, we have $\bb E L_{3,2,1}=O_\prec(tN^{-1})$. Other terms in $\bb EL_{3,2}$ can be estimated similarly, and we have
\begin{equation} \label{4.22}
	\bb E L_{3,2}=O_\prec(tN^{-1})\,.
\end{equation}

\textit{Step 3.} Now let us estimate $\bb EL_{3,3}$, which contains the leading contribution to $\mathbb{E}(TL_3)$. By \eqref{diff}, we have
\begin{equation*}
	\begin{aligned}
		\bb EL_{3,3}=&-\frac{2 s_4\ii t}{N^2}{\sum_{i,j
		}}^{*}\bb E\{T\widehat{G}_{ii}\widehat{G}_{jj}\}\{(\partial_z\widehat{G}_{ii})\widehat{G}_{jj}\}\zeta(t)-\frac{2 s_4\ii t}{N^2}{\sum_{i,j
		}}^{*}\bb E\{T\widehat{G}_{ij}^2\}\{(\partial_z\widehat{G}_{ii})\widehat{G}_{jj}\}\zeta(t)\\
		&-\frac{2 s_4\ii t}{N^2}{\sum_{i,j
		}}^{*}\bb E\{T\widehat{G}_{ii}\widehat{G}_{jj}\}\{(\partial_z\widehat{G}_{ij})\widehat{G}_{ij}\}\zeta(t)-\frac{2 s_4\ii t}{N^2}{\sum_{i,j
		}}^{*}\bb E\{T\widehat{G}_{ij}^2\}\{(\partial_z\widehat{G}_{ij})\widehat{G}_{ij}\}\zeta(t)\,,
	\end{aligned}
\end{equation*}
and let us define the last sum by $\sum_{p=1}^4\bb EL_{3,3,p}$. The leading term is
\begin{equation*}
	\begin{aligned}
		\bb E L_{3,3,1}=&-\frac{2 s_4\ii t}{N^2}N(N-1)\{Tm^2\}\{m'm\}\bb E\zeta(t)-\frac{2 s_4\ii t}{N^2}{\sum_{i,j
		}}^{*}\bb E\langle\{T\widehat{G}_{ii}\widehat{G}_{jj}\}\{(\partial_z\widehat{G}_{ii})\widehat{G}_{jj}\}\rangle\zeta(t)\\
		&-\frac{2 s_4\ii t}{N^2}\Big({\sum_{i,j}}^*\bb E\{T\widehat{G}_{ii}\widehat{G}_{jj}\}\{(\partial_z\widehat{G}_{ii})\widehat{G}_{jj}\}-N(N-1)\{Tm^2\}\{m'm\}\Big)\bb E\zeta(t)\,.
	\end{aligned}
\end{equation*}
Note that $T(z)m(z)=m'(z)+O_{\prec}(1/(N\eta^2))$, and we can use Lemmas  \ref{prop4.4}, \ref{lem4.5} and \ref{lem5.7} to estimate the second and third term on RHS of the above. Together with Lemma \ref{lem:integration}, we get
\begin{equation*}
	\bb E L_{3,3,1}=-2s_4 \ii t \{m'm\}^2 \phi_1(t)+O_{\prec}(tN^{-1})\,.
\end{equation*}
Using Lemmas \ref{lem4.5} and \ref{lem5.7}, the other three terms in $\bb EL_{3,3}$ can be estimated by $O_{\prec}(tN^{-1})$.
Hence
\begin{equation} \label{4.23}
	\bb E L_{3,3}=-2s_4 \ii t \{m'm\}^2 \phi_1(t)+O_{\prec}(tN^{-1})\,.
\end{equation}

\textit{Step 4.} By \eqref{diff}, the worst terms in $\bb E L_{3,4}$ are
\begin{multline*}
	\bb EL_{3,4,1}+\bb EL_{3,4,2}\deq\\
	\frac{4 s_4\ii t}{N^2}{\sum_{i,j
	}}^{*}\bb E\{T\widehat{G}_{ii}\widehat{G}_{jj}\}\{(\widehat{G}H_{\mathrm d}\widehat{G}^2)_{ii}\widehat{G}_{jj}\}\zeta(t)+\frac{2 s_4\ii t}{N^2}{\sum_{i,j
	}}^{*}\bb E\{T\widehat{G}_{ii}\widehat{G}_{jj}\}\{(\widehat{G}H_{\mathrm d}\widehat{G})_{ii}(\partial_z\widehat{G}_{jj})\}\zeta(t)\,.
\end{multline*}
We see that $L_{3,4,1}=\cal S(\cal {P})\zeta(t)$, where $\mu(\cal P)=1$, $\nu(\cal P)=2$, $\theta(\cal P)=2$, $d(\cal P )=1$, $\nu_1(\cal P)=\nu_2(\cal P)=\nu_3(\cal P)=0$, and $\nu_0(\cal P)=3$. By Lemma \ref{lem5.7} we have
$
\bb E L_{3,4,1}=O_{\prec} (tN^{-1})\,,
$
and similarly the same bound also holds for $\bb E L_{3,4,2}$. Estimating other terms in $\bb EL_{3,4}$ we get
\begin{equation} \label{reee}
	\bb E L_{3,4} =O_{\prec}(tN^{-1})\,.
\end{equation}

\textit{Step 5.} Again by \eqref{diff}, the worst term in $\bb EL_{3,5}$ is of the form
\[
\bb EL_{3,5,1}\deq \frac{2s_4t^2}{N^2}{\sum_{i,j}}^*\bb E\{T\widehat{G}_{ii}\widehat{G}_{jj}\}\{(\partial_z\widehat{G}_{ij})\}^2 \zeta(t)\,.
\]
We see that $L_{3,5,1}=\cal S(\cal { P})\zeta(t)$, where $\mu(\cal P)=2$, $\nu(\cal P)=2$, $\theta(\cal P)=2$, $\nu_1(\cal P)=2$, $d(\cal P)=0$, $\nu_2(\cal P)=\nu_3(\cal P)=0$, and $\nu_0(\cal P)=3$, $\nu_*(\cal P)=0$. By Lemmas \ref{lem4.5} and \ref{lem5.7}, we have
\[
\bb EL_{3,5,1}=O_{\prec}(t^2N^{-1})\cdot \phi_1(t)+O_{\prec}(t^2N^{-2})\,.
\]
Other terms in $\bb EL_{3,5}$ can be estimated by $O_{\prec}(t^2N^{-2})$ using Lemma \ref{lem4.5}. Thus
\begin{equation} \label{4.24}
	\bb EL_{3,5}=O_{\prec}(t^2N^{-1})\cdot \phi_1(t)+O_{\prec}(t^2N^{-2})\,.
\end{equation}
A similar argument applies to $\bb EL_{3,6}$, and one can show that
\begin{equation} \label{4.25}
	\bb EL_{3,6}=O_{\prec}(t^2N^{-1}+t^3N^{-2})\cdot \phi_1(t)+O_{\prec}(tN^{-1}+t^2N^{-2}+t^{3}N^{-3})\,.
\end{equation}
Combining \eqref{4.21} -- \eqref{4.25}, we have
\begin{equation} \label{L_3}
	\bb E\{TL_3\}=\big(-2s_4 \ii t \{m'm\}^2+O_{\prec}(t^2N^{-1})\big) \phi_1(t)+O_{\prec}((t+1)N^{-1})\,.
\end{equation}

\subsection{The estimate of $\bb E\{TL_k\}, k \geq 4$} \label{sec4.5}
Let us fix a $k \geq 4$. From \eqref{3.14} and \eqref{diff} we see that
	\begin{align*}
		\big|\bb E\{TL_k\} \big|\leq &\,\frac{C_k}{N^{(1+k)/2}} \cdot\sum_{p=1}^k\sum_{\substack{k_0,...,k_p\in \bb N\\ k_0+\cdots +k_p=k-p}}\bigg| {\sum_{i,j
		}}^{*} \bb E\bigg(\bigg\{T\frac{\partial^{k_0} \widehat{G}_{ji}}{\partial H_{ji}^{k_0}}\bigg\} \bigg\{\frac{\partial^{k_1} B_{ji}}{\partial H_{ji}^{k_1}}t\bigg\}\cdots \bigg\{\frac{\partial^{k_p} B_{ji}}{\partial H_{ji}^{k_p}}t\bigg\}\zeta(t)\bigg)\bigg|\\
		&+\frac{C_k}{N^{(1+k)/2}} \cdot \bigg|{\sum_{i,j
		}}^{\hspace{-0cm}*}\bb E\bigg\{ T\frac{\partial^k \widehat{G}_{ji} }{\partial H_{ji}^k}\bigg\}\langle \zeta(t) \rangle\bigg|\eqd \sum_{p=1}^kA_{k,p}+A_{k,0}\,,
	\end{align*}
where $B=(\partial_z\widehat{G})-2\widehat{G}H_{\mathrm d}\widehat{G}^2$. For $k=4$, by \eqref{diff} there is at least one factor of $\widehat{G}_{ij}$ in $A_{4,0}$. Hence,  Theorems \ref{refthm1} and \ref{lem:integration} imply $A_{4,0}=O_{\prec}(N^{-5/2}\cdot N^2\cdot N^{-1/2})=O_{\prec}(N^{-1})$. We also see that $A_{k,0}=O_{\prec}(N^{-(k+1)/2}\cdot N^2)=O_{\prec}(N^{-1})$ for $k \geq 5$. Hence
\begin{equation} \label{4.27}
	A_{k,0}=O_{\prec}(N^{-1})\,.
\end{equation}
A similar argument shows that
\begin{equation} \label{4.28}
	A_{k,1}=O_{\prec}(tN^{-1})\,.
\end{equation}
Now we fix $p \geq 2$ as well as $k_0,...,k_p\in \bb N$ satisfying $k_0+\cdots +k_p=k-p$. We consider the following representative term of $A_{k,p}$, namely
\begin{equation*}
	\begin{aligned}
		\cal A\deq  &\,\frac{1}{N^{(k+1)/2}}{\sum_{i,j
		}}^{*} \bb E\bigg(\bigg\{T\frac{\partial^{k_0} \widehat{G}_{ji}}{\partial H_{ji}^{k_0}}\bigg\} \bigg\{\frac{\partial^{k_1} B_{ji}}{\partial H_{ji}^{k_1}}t\bigg\}\cdots \bigg\{\frac{\partial^{k_p} B_{ji}}{\partial H_{ji}^{k_p}}t\bigg\}\zeta(t)\bigg)
		\\
		=&\,t^{-1}{\sum_{i,j
		}}^{*} \bb E\bigg(\bigg\{T\frac{\partial^{k_0} \widehat{G}_{ji}}{\partial H_{ji}^{k_0}}tN^{-\frac{k_0+1}{2}}\bigg\} \bigg\{\frac{\partial^{k_1} B_{ji}}{\partial H_{ji}^{k_1}}tN^{-\frac{k_1+1}{2}}\bigg\}\cdots \bigg\{\frac{\partial^{k_p} B_{ji}}{\partial H_{ji}^{k_p}}tN^{-\frac{k_p+1}{2}}\bigg\}\zeta(t)\bigg)\\
		=: &\, t^{-1}{\sum_{i,j}}^* \bb E(\cal A_{0}\cdots \cal A_{p}\zeta(t))
	\end{aligned}
\end{equation*}
Note that in order to bound $\sum_{p=1}^kA_{k,p}$, it is enough to estimate $\cal A$, as $k$ is independent of $N$. By \eqref{diff}, Theorem \ref{refthm1}, Lemmas \ref{lemGHG} and \ref{lem:integration}, we have the naive estimates
\begin{equation} \label{4.29}
	\bigg\{\frac{\partial^{n} \widehat{G}_{ji}}{\partial H_{ji}^{n}}tN^{-\frac{n+1}{2}}\bigg\}=O_{\prec}(tN^{-1})\,,\quad	\bigg\{\frac{\partial^{n} B_{ji}}{\partial H_{ji}^{n}}tN^{-\frac{n+1}{2}}\bigg\}=O_{\prec}(tN^{-1})
\end{equation}
for all fixed $n \in \bb N$. Now let us split
$
\cal A_q =\cal A_{q,1}+\cdots +\cal A_{q,n_q}
$
for all $q \in \{0,1,...,p\}$, by applying \eqref{diff} to 
$$
\frac{\partial^{k_0} \widehat{G}_{ji}}{\partial H_{ji}^{k_0}}\quad 
\mbox{or} \quad \frac{\partial^{k_q} B_{ji}}{\partial H_{ji}^{k_q}}
$$ 
such that there is only one term in each $\cal A_{q,r}$, $1 \leq r \leq n_q$. Suppose some $\cal A_{q,r}$ contains at least two factors of  $(\widehat{G}H_{\mathrm d}\widehat{G})_{ji}$, $(\widehat{G}H_{\mathrm d}\widehat{G}^2)_{ji}$ $\partial_z\widehat{G}_{ij}$ or $\widehat{G}_{ij}$. Then \eqref{diff} suggests $k_q \geq 1$, and by Theorem \ref{refthm1},  Lemmas \ref{lemGHG} and \ref{lem:integration} we have
\begin{equation*}
	\cal A_{q,r}=O_{\prec}(N^{-1} \cdot t\cdot N^{-\frac{k_q+1}{2}}) =O_{\prec}(tN^{-2}) \,.
\end{equation*}
Combining with \eqref{4.29} we have
\begin{equation*}
		t^{-1}{\sum_{i,j}}^* \bb E(\cal A_{0}\cdots \cal A_{q-1}\cal A_{q,r} \cal A_{q+1}\cdots\cal A_{p}\zeta(t))
		=O_\prec( t^{-1}\cdot N^2 \cdot (tN^{-1})^p \cdot tN^{-2})=O_{\prec}(tN^{-1})\,.
\end{equation*}
Hence, it suffices to consider the case when each $\cal A_{q,r}$ contains at most one aforementioned off-diagonal factors. Then, according to Definitions  \ref{def:cal_V} and \ref{defn4.3}, we can write 
\begin{equation} \label{4.31}
	\cal A=\sum_{r} \bb E\cal S(\cal P_r)\zeta(t)+ O_{\prec}(tN^{-1})
\end{equation}
where $\nu_0(\cal P_r)\leq3+d(\cal P_r)$ for each $\cal P_r\in \b P$, and $\sum_r$ is over finitely many (independent of $N$) terms. We also see that $\mu(\cal P_r)=p$, $\nu(\cal P_r)=2$, $\theta(\cal P_r)=(k+1)/2$, $\nu_2(\cal P)=\nu_3(\cal P)=\nu_*(\cal P)=0$.  By applying \eqref{diff}, we can write 
\[
\frac{\partial^{n}B_{ji}}{\partial H_{ji}^{n}}=:\sum_{s}\cal G_s\,, \quad \frac{\partial^{n}\widehat{G}_{ji}}{\partial H_{ji}^{n}}=:\sum_{s}\cal G'_s
\]
and note that each $\cal G_s, \cal G'_s$ contains at least $1-n$ many off-diagonal factors. Hence $\nu_1(\cal P_r)\geq p+1-(k_0+\cdots+k_p)=2p+1-k$. Thus Lemmas \ref{lem4.5} and \ref{lem5.7} shows
\begin{equation*}
	\begin{aligned}
		\bb E \cal S(\cal P_r)\zeta(t)=&\,O_{\prec}(t^p\cdot N^{2-(k+1)/2-(2p+1-k)/2})\cdot \phi_1(t)+O_{\prec}(t^p\cdot N^{2-(k+1)/2-(2p+1-k)/2-1})\\
		=&\,O_{\prec}(t^pN^{-p+1})\cdot \phi_1(t)+O_{\prec}(t^pN^{-p})=O_{\prec}(t^2N^{-1})\cdot \phi_1(t)+O_{\prec}(t^2N^{-2})\,.
	\end{aligned}
\end{equation*}
for each $\bb E \cal S(\cal P_r)\zeta(t)$ on RHS of \eqref{4.31}. As a result, $\cal A=O_{\prec}(t^2N^{-1})\cdot \phi_1(t)+O_{\prec}(t^2N^{-2})$, which implies
\[
\sum_{p=2}^kA_{k,p}=O_{\prec}(t^2N^{-1})\cdot \phi_1(t)+O_{\prec}(t^2N^{-2})\,.
\]
Combining the above with \eqref{4.27} and \eqref{4.28}, we get
\begin{equation} \label{L_k}
	\bb E\{TL_k\}=O_{\prec}(t^2N^{-1})\cdot \phi_1(t)+O_{\prec}(tN^{-1})
\end{equation}
for all $k \geq 4$.

\subsection{Conclusion} \label{sec4.6}
Inserting \eqref{L_2}, \eqref{L_3} and \eqref{L_k} into \eqref{4.14}, we have 
\begin{equation*} 
	\begin{aligned}
		\phi'_1(t) 
		=&\,-\frac{2t}{\pi^2} \int_{\b D^2} \Big(\frac{\partial}{\partial \bar{z}_1} \tilde{f}(z_1)\Big)\Big(\frac{\partial}{\partial \bar{z}_2} \tilde{f}(z_2)\Big)g_1(z_1,z_2)\dd^2 z_1 \dd^2 z_2 \cdot\phi_1(t)\\
		&\,-\frac{2s_4  t}{\pi^2} \bigg(\int_{\b D} \frac{\partial}{\partial \bar{z}} \tilde{f}(z) m'(z)m(z)\dd^2z\bigg)^2\cdot\phi_1(t)\\
		&\, +O_{\prec}(t^2N^{-1})\cdot\phi_1(t)+O_{\prec}((t+1)N^{-1}+N^{-1/2})\,,
	\end{aligned}
\end{equation*}
where $g_1(z_1,z_2)$ is defined as in \eqref{g_1}. One can then follow a standard computation, e.g.\,\cite[Section 4.3]{LS18} to evaluate the first two terms on RHS of the above, and show that
\begin{equation} \label{4.39}
	\phi'_1(t) =-\sigma^2_f  t \cdot\phi_1(t) +O_{\prec}(t^2N^{-1})\cdot\phi_1(t)+O_{\prec}((t+1)N^{-1}+N^{-1/2})\,.
\end{equation}
This finished the proof of Proposition \ref{lem. tf=0} for $\xi=1$.

\subsection{Proof of Proposition \ref{lem. tf=0} for $\xi=0$} \label{sec4.4}

In this section we compute $\phi_0(t)=\bb E \varphi_1(t)\varphi_4(t)$. We shall work under the assumption that
\[
t\in [0,N^{(1-c)/2}]\,.
\]
Thanks to the much smaller $t$, the proof is easier than the case $\xi=1$. Let us denote $\varphi_1(t)=\exp(J(t))$ and $\varphi_4(t)=\exp(K(t))$. We see that
\begin{equation} \label{J}
	\max_{k \ne l}\bigg|	\frac{\partial^n J(t)}{\partial H_{kl}^n}\bigg|=O_{\prec}\Big(\frac{t}{N^{1/2}}\cdot (N^{1/2}\b 1(n \geq 2)+1)\Big)
\end{equation}
for $n \in \bb N_+$. In addition, from (\ref{eq of phi 4}), 
\begin{equation*}
	\begin{aligned}
		K(t)&= -\frac{t^2}{2}\sum_i \mathbb{E}_d X_{i}^2+\frac{(\ii t)^3}{6}\sum_{i}\mathbb{E}_d X_{i}^3\\
		&=-\frac{a_2t^2}{2N}\sum_i \{\langle (\widehat{G}^2)_{ii}\rangle\}^2+\frac{a_3c_2^ft^2}{2N^{3/2}}\{\langle  \tr \widehat{G}^2 \rangle \}+\frac{a_3t^3 \ii}{6 N^{3/2}}\sum_i \{\langle (\widehat{G}^2)_{ii}\rangle\}^3\\
		&\ -\frac{(a_4-a_2^2)c_2^f t^3\ii }{4N^2}\sum_i \{\langle (\widehat{G}^2)_{ii}\rangle\}^2+\frac{(a_5-2a_3a_2)c_2^f t^3 \ii}{8N^{5/2}} \{\langle  \tr \widehat{G}^2 \rangle \}+F(t)\,,
	\end{aligned}
\end{equation*} 
where $F(t)$ is deterministic and satisfies $|F(t)|=O(t^2N^{-1})$. It is easy to check from Theorem \ref{refthm1} that
\begin{equation} \label{4.40}
	K(t) =O_{\prec}\Big(\frac{t^2}{N}+\frac{t^3}{N^2}\Big)=O_{\prec}\Big(\frac{t^2}{N}\Big)\,, \quad \mbox{and} \quad  K'(t) =O_{\prec}\Big(\frac{t}{N}\Big)\,.
\end{equation}
In addition, Theorem \ref{refthm1} and \eqref{diff} implies that
\begin{equation} \label{4.42}
	\max_{k\ne l}\Bigg|	\frac{\partial K(t)}{\partial H_{kl}}-\frac{2a_2t^2}{N}\sum_i \{\langle (\widehat{G}^2)_{ii}\rangle\}\{(\widehat{G}^2)_{ik}\widehat{G}_{li}+(\widehat{G}^2)_{il}\widehat{G}_{ki}\}\Bigg| =O_{\prec}\Big(\frac{t^2}{N^2}\Big)\,.
\end{equation}
and
\begin{equation} \label{4.43}
	\max_{k \ne l}\bigg|	\frac{\partial^n K(t)}{\partial H_{kl}^n}\bigg|=O_{\prec}\Big(\frac{t^2}{N^{3/2}}\cdot (N^{1/2}\b 1(n \geq 4)+1)\Big)
\end{equation}
for $n \in \bb N_+$. We have
\[
\phi_0'(t)=\frac{\ii }{\pi} \int_{\b D} \frac{\partial }{\partial \bar{z}} \tilde{f}(z)\bb E \big[\big\langle\text{Tr} \widehat{G}(z) \rangle \varphi_1(t)\varphi_4(t)\big] {\rm d}^2 z+ \bb E[ K'(t) \varphi_1(t)\varphi_4(t)]\,.
\]
By \eqref{4.40} we have  $\bb E[ K'(t) \varphi_1(t)\varphi_4(t)]=O_{\prec}(tN^{-1})$, and thus
\[
\phi_0'(t)=\ii \bb E \{\langle \tr \widehat{G} \rangle\} \varphi_1(t)\varphi_4(t)+O_{\prec}(tN^{-1})\,.
\]
Now we can compute the first term on RHS of the above by Lemma \ref{lem:cumulant_expansion}. Similar to \eqref{maingreen}, we have
\begin{equation}  \label{maingreen2}
	\begin{aligned}
		\phi'_0(t) &= \ii\bb E \{T\langle\underline{\widehat{G}^2} \rangle\} \varphi_1(t)\varphi_4(t) +\ii N\bb E \{T \langle \underline{\widehat{G}} \rangle^2\}\varphi_1(t)\varphi_4(t)
		-\ii  N\bb E\{ T\langle \underline{\widehat{G}} \rangle^2\}\phi_0(t) \\
		&\ \ +2\ii N^{-1}\sum_i \bb E \{T\langle \widehat{G}_{ii}^2 \rangle\}\varphi_1(t)\varphi_4(t)
		-\frac{2 t}{N}{\sum_{i,j}}^*\bb E\{T\widehat{G}_{ij}\}\{(\widehat{G}^2)_{ij}\} \varphi_1(t)\varphi_4(t)\\
		&\ \ +\frac{1}{N}{\sum_{i,j}}^*\bb E\{T\widehat{G}_{ij}\}\frac{\partial K(t)}{\partial H_{ij}}\varphi_1(t)\varphi_4(t)-\ii\sum_{k=2}^{\ell}\bb E\{T\widehat{L}_{k}\}+O_{\prec}(tN^{-1})\,,
	\end{aligned}
\end{equation}
where we estimate the remainder term by $O_{\prec}(tN^{-1})$ for some  fixed (large) $\ell \in \bb N_+$, and
\[
\widehat{L}_{k} \deq\frac{s_{k+1}}{k!}\frac{1}{N^{(1+k)/2}} \cdot {\sum_{i,j
}}^{\hspace{-0cm}*} \frac{\partial^k( \widehat{G}_{ji} \langle \varphi_1(t)\varphi_4(t) \rangle)}{\partial H_{ij}^k}\,.
\]
By \eqref{4.42}, Theorem \ref{refthm1} and Lemma \ref{lem:integration}, we have
\begin{equation*}
	\begin{aligned}
		&\quad\frac{1}{N}{\sum_{i,j}}^*\bb E\{T\widehat{G}_{ij}\}\frac{\partial K(t)}{\partial H_{ij}}\varphi_1(t)\varphi_4(t)\\
		&=\frac{2at^2}{N^2}{\sum_{i,j}}^*\sum_{k}\bb E\{T\widehat{G}_{ij}\}\{\langle (\widehat{G}^2)_{kk}\rangle\}\{(\widehat{G}^2)_{ki}\widehat{G}_{jk}+(\widehat{G}^2)_{kj}\widehat{G}_{ik}\} \varphi_1(t)\varphi_4(t)+O_\prec(t^2N^{-3/2})\\
		&=\frac{2at^2}{N^2}{\sum_{i,j}}\sum_{k}\bb E\{T\widehat{G}_{ij}\}\{\langle (\widehat{G}^2)_{kk}\rangle\}\{(\widehat{G}^2)_{ki}\widehat{G}_{jk}+(\widehat{G}^2)_{kj}\widehat{G}_{ik}\} \varphi_1(t)\varphi_4(t)+O_\prec(t^2N^{-3/2})\\
		&=\frac{4at^2}{N^2}\sum_{k}\bb E\{T(z_1)(\widehat{G}(z_1)\widehat{G}^3(z_2))_{kk}\}_2\{\langle (\widehat{G}^2)_{kk}\rangle\} \varphi_1(t)\varphi_4(t)+O_\prec(t^2N^{-3/2})=O_\prec(tN^{-1})
	\end{aligned}
\end{equation*}
for $t \in [0,N^{(1-c)/2}]$. The first five terms on RHS of \eqref{maingreen2} can be computed as in Section \ref{sec4.1}. Thus we have, as in \eqref{4.14} that
\begin{equation} \label{4.46}
	\phi'_0(t) 
	=-\frac{2t}{\pi^2} \int_{\b D^2} \Big(\frac{\partial}{\partial \bar{z_1}} \tilde{f}(z_1)\Big)\Big(\frac{\partial}{\partial \bar{z_2}} \tilde{f}(z_2)\Big)g_1(z_1,z_2)\dd^2 z_1 \dd^2 z_2 \,\phi_0(t)-\ii\sum_{k=2}^{\ell}\bb E\{T\widehat{L}_{k}\}+O_{\prec}(tN^{-1})\,,
\end{equation}
where $g_1(z_1,z_2)$ is defined as in \eqref{g_1}. For each $k\geq 2$, we can decompose
\begin{multline*}
	\widehat{L}_k=L_{k,0}+\sum_{m=1}^k\sum_{n=0}^m \widehat{L}_{k,m,n}\deq  \frac{s_{k+1}}{k!}\frac{1}{N^{(1+k)/2}}  {\sum_{i,j
	}}^{\hspace{-0cm}*}\frac{\partial^{k} \widehat{G}_{ji}  }{\partial H_{ij}^{k}}\langle\varphi_1(t)\varphi_4(t)\rangle
	\\
	+ \sum_{m=1}^k\sum_{n=0}^m{k \choose m}{m\choose n} \frac{s_{k+1}}{k!}\frac{1}{N^{(1+k)/2}}  {\sum_{i,j
	}}^{\hspace{-0cm}*}\frac{\partial^{k-m} \widehat{G}_{ji}  }{\partial H_{ij}^{k-m}} \frac{\partial^{m-n} \varphi_1(t) }{\partial H_{ij}^{m-n}}\frac{\partial^n \varphi_4(t)}{\partial H_{ij}^n}\,.
\end{multline*}
Note that $\bb E\{T\widehat{L}_{k,0}\}$ and $\bb E\{T\widehat{L}_{k,m,0}\}$ can be computed exactly as in Sections \ref{sec4.3} -- \ref{sec4.5}, using Lemmas \ref{lem4.5} and \ref{lem5.7} for the case $ d(\cal P)=0$. We can show that
\begin{multline} \label{4.48}
	-\ii\sum_{k=2}^{\ell}\bb E\{T\widehat{L}_{k,0}\}-\ii \sum_{k=2}^\ell \sum_{m=1}^k\bb E \{T\widehat{L}_{k,m,0}\}\\
	=	-\frac{2s_4  t}{\pi^2} \bigg(\int_{\b D} \frac{\partial}{\partial \bar{z}} \tilde{f}(z) m'(z)m(z)\dd^2z\bigg)^2\cdot\phi_0(t) +O_{\prec}(t^2N^{-1})\cdot\phi_0(t)+O_{\prec}((t+1)N^{-1})\,.
\end{multline}
Combining \eqref{4.46} and \eqref{4.48}, and compute the result as in \eqref{4.39}, we get
\begin{equation} \label{4.49}
	\phi'_0(t) =-\sigma^2_f  t \cdot\phi_0(t) +O_{\prec}(t^2N^{-1})\cdot\phi_0(t)+O_{\prec}((t+1)N^{-1})-\ii \sum_{k=2}^\ell \sum_{m=1}^k\sum_{n=1}^m\bb E \{T\widehat{L}_{k,m,n}\}\,.
\end{equation}
Hence it remains to estimate the last sum in \eqref{4.49}.

Let us first consider the case $k=2$. By \eqref{diff}, \eqref{J} and \eqref{4.43}, we see that for $t\in[0, N^{(1-c)/2}]$
\begin{equation*}
	-\ii  \sum_{m=1}^2\sum_{n=1}^m\bb E \{T\widehat{L}_{2,m,n}\}=\frac{s_3 \ii }{N^{3/2}}{\sum_{i,j
	}}^{\hspace{-0cm}*} \bb E \{T\widehat{G}_{ii}\widehat{G}_{jj}\} \frac{\partial K(t)}{\partial H_{ij}} \varphi_1(t)\varphi_4(t)+O_{\prec}(tN^{-1})\,,
\end{equation*}
and together with \eqref{4.42} we have
\begin{equation} \label{4.51}
	\begin{aligned}
		&\ -\ii  \sum_{m=1}^2\sum_{n=1}^m\bb E \{T\widehat{L}_{2,m,n}\}\\
		&=\frac{2s_3a_2 t^2\ii }{N^{5/2}}{\sum_{i,j
		}}^{\hspace{-0cm}*} \sum_k\bb E \{T\widehat{G}_{ii}\widehat{G}_{jj}\} \{\langle (\widehat{G}^2)_{kk}\rangle\}\{(\widehat{G}^2)_{ki}\widehat{G}_{jk}+(\widehat{G}^2)_{kj}\widehat{G}_{ik}\} \varphi_1(t)\varphi_4(t)+O_{\prec}(tN^{-1})\\
		&=\frac{2s_3a_2 t^2\ii }{N^{5/2}}\sum_{i,j,k}\bb E \{T\widehat{G}_{ii}\widehat{G}_{jj}\} \{\langle (\widehat{G}^2)_{kk}\rangle\}\{(\widehat{G}^2)_{ki}\widehat{G}_{jk}+(\widehat{G}^2)_{kj}\widehat{G}_{ik}\} \varphi_1(t)\varphi_4(t)+O_{\prec}(tN^{-1})\,.
	\end{aligned}
\end{equation}
Now, we continue the estimate with the isotropic law. By Theorem \ref{refthm1} we see that
\[
\sum_{i} \widehat{G}(z_1)_{ii}\widehat{G}(z_2)_{ik}=\sum_{i}(\widehat{G}(z_1)_{ii}-m(z_1))\widehat{G}(z_2)_{ik}+ \sum_{i}m(z_1)\widehat{G}(z_2)_{ik}\prec \frac{1}{|\eta_1\eta_2|^{1/2}}
\]
and $\sum_{j} \widehat{G}(z_1)_{jj}\widehat{G}(z_2)_{jk}\prec |\eta_1\eta_2|^{-1/2}$ uniformly for $z_1,z_2 \in \b D$. Plug these two estimates into \eqref{4.51}, together with Theorem \ref{refthm1} and Lemma \ref{lem:integration}, we get
\begin{equation}\label{4.52}
	-\ii  \sum_{m=1}^2\sum_{n=1}^m\bb E \{T\widehat{L}_{2,m,n}\}=O_{\prec}(tN^{-1})
\end{equation}
as desired.

Now we consider the case $k \geq 3$. Fix $m,n\geq 1$, and set $a=k-m$, $b=m-n$. We have
\begin{multline} \label{4.53}
	\bb E \{T\widehat{L}_{k,m,n}\}\leq \frac{C_k}{N^{(k+1)/2}}{\sum_{i,j
	}}^{\hspace{-0cm}*} \bb E \bigg|\bigg\{T\frac{\partial^{a} \widehat{G}_{ji}  }{\partial H_{ij}^{a}} \bigg\}\frac{\partial^{b} \varphi_1(t) }{\partial H_{ij}^{b}}\frac{\partial^n \varphi_4(t)}{\partial H_{ij}^n}\bigg|\\
	\leq \frac{C'_k}{N^{(k+1)/2}}\sum_{\substack{b_1,...,b_p\in \bb N_+\\ b_1+\cdots +b_p=b}}\sum_{\substack{n_1,...,n_q\in \bb N_+\\ n_1+\cdots +n_q=n}}{\sum_{i,j
	}}^{\hspace{-0cm}*} \bb E \bigg|\bigg\{T\frac{\partial^{a} \widehat{G}_{ji}  }{\partial H_{ij}^{a}} \bigg\}\frac{\partial^{b_1} J(t) }{\partial H_{ij}^{b_1}}\cdots \frac{\partial^{b_p} J(t) }{\partial H_{ij}^{b_p}}\frac{\partial^{n_1} K(t)}{\partial H_{ij}^{n_1}}\cdots \frac{\partial^{n_q} K(t)}{\partial H_{ij}^{n_q}}\bigg|\,.
\end{multline}
Note that $\{T \partial^a \widehat{G}_{ji}/\partial H_{ij}^a\}=O_\prec(N^{(a-1)/2})$ for $a\geq 0$. Together with \eqref{J} and \eqref{4.43}, we have
\begin{equation} \label{4.54}
	\begin{aligned}
		&\quad	\frac{C'_k}{N^{(k+1)/2}}{\sum_{i,j
		}}^{\hspace{-0cm}*} \bb E \bigg|\bigg\{T\frac{\partial^{a} \widehat{G}_{ji}  }{\partial H_{ij}^{a}} \bigg\}\frac{\partial^{b_1} J(t) }{\partial H_{ij}^{b_1}}\cdots \frac{\partial^{b_p} J(t) }{\partial H_{ij}^{b_p}}\frac{\partial^{n_1} K(t)}{\partial H_{ij}^{n_1}}\cdots \frac{\partial^{n_q} K(t)}{\partial H_{ij}^{n_q}}\bigg|\\
		&\prec {N^{-(k+1)/2}}\cdot N^2 \cdot N^{(a-1)/2}\cdot tN^{(b_1-2)/2} \cdots   tN^{(b_p-2)/2} \cdot t^2N^{(n_1-4)/2}\cdots  t^2N^{(n_p-4)/2} \\
		&= N \cdot (tN^{-1})^p \cdot (t^2N^{-2})^q \leq tN^{-1/2} \cdot N^{-p/2}\cdot N^{1-q}
	\end{aligned}
\end{equation}
where $q \geq 1$. If $(p,q)\neq (0,1)$, then we have $\eqref{4.54} \prec tN^{-1}$. If $(p,q)=(0,1)$, then 
\[
\eqref{4.54}=\frac{C'_k}{N^{(k+1)/2}}{\sum_{i,j
}}^{\hspace{-0cm}*} \bb E \bigg|\bigg\{T\frac{\partial^{k-n} \widehat{G}_{ji}  }{\partial H_{ij}^{k-n}} \bigg\}\frac{\partial^n K(t)}{\partial H^n_{ij}}\bigg|\,.
\]
One easily checks that $\{T \partial^{k-n} \widehat{G}_{ji}/\partial H_{ij}^{k-n}\}=O_\prec(N^{(k-n-3)/2})$ for $k-n \geq 2$, and $\partial^nK(t)/\partial H_{ij}^n=O(t^2N^{(n-5)/2})$ for $n\geq 2$. As $(k-n)+n=k\geq 3$, we see that in this case 
\[
\eqref{4.54} \prec N^{-(k+1)/2}\cdot N^2\cdot (N^{(k-n-3)/2}\cdot t^2N^{(n-4)/2}+N^{(k-n-1)/2}\cdot t^2N^{(n-5)/2}) \prec tN^{-1}\,.
\]
Hence we always have $\eqref{4.54} \prec tN^{-1}$, and \eqref{4.53} shows
\begin{equation} \label{4.55}
	\bb E \{T\widehat{L}_{k,m,n}\}=O_{\prec}(tN^{-1})
\end{equation}
for all $k \geq 2$ and $n \geq 1$. Combining \eqref{4.49}, \eqref{4.52} and \eqref{4.55} finishes the proof.

\section{Proof of the improved estimates for abstract polynomials} \label{sec5}

In this section we prove Lemma \ref{lem5.7}. For notational convenience, throughout the section we make the following convention
\begin{align} 
	\frac{\partial \widehat{G}_{ab}}{\partial H_{ii}}\deq -\widehat{G}_{ai}\widehat{G}_{ib}. \label{convention on diagonal dderivatie}
\end{align}
Whenever the above partial derivative notation is used, it will always be compensated by a $\widehat{G}_{ai}\widehat{G}_{ib}$ term (with positive sign) since the true derivative of $\widehat{G}$-entries w.r.t. $H_{ii}$ is apparently $0$. 
In this way, we complete \eqref{diff} to
\begin{equation} \label{diff1}
	\frac{\partial \widehat{G}_{ab}}{\partial H_{ij}}=-(\widehat{G}_{ai}\widehat{G}_{jb}+\widehat{G}_{aj}\widehat{G}_{ib})(1+\delta_{ij})^{-1}
\end{equation}
for all $i,j$. From Jensen's Inequality, we can easily deduce Lemma \ref{lem5.7} from the following result.

\begin{lem} \label{lem5.1}
	Let us adopt the assumptions in Lemma \ref{lem5.7}. We have the following estimates.
	
	(i)  Suppose $d(\cal P)=0$, we have
	\begin{equation*}  
		\bb E\,|\cal S (\cal	{\mathring P} )|^2=O_{\prec}\big(\cal E_*(\cal P)^2 \big)\,.
	\end{equation*}
Moreover, under the additional assumption $\nu_1(\cal P)=1$, we have the stronger estimate
	\[
\bb E\,|\cal S (\cal	{ P} )|^2=O_{\prec}\big(\cal E_*(\cal P)^2 \big)\,.
\]

	(ii) Suppose $d(\cal P)=1$, we have
	\[
	\bb E\,|\cal S (\cal	{ P} )|^2=O_{\prec}\big(\cal E_*(\cal P)^2 \big)\,.
	\]
\end{lem}

\subsection{The lone factor and happy trio}  \label{sec6.1}
To prove Lemma \ref{lem5.1}, we need some a priori estimate, which is given in Lemma \ref{lem5.4}, after some additional notions are introduced below.
\begin{defn}
	Let $\cal P =  a t^\mu N^{-\theta}\{\cal G_1\}\cdots \{\cal G_\sigma\}\in \b P$, where $\cal G_r \in \b G(z_r)$ for $r=1,...,\sigma$. Let $\b M\deq \{\widehat{G}(z_1), \partial_{z_1} \widehat{G}(z_1),...,\widehat{G}(z_\sigma),\partial_{z_{\sigma}} \widehat{G}(z_{\sigma})\}$. 
	
	(i) If there exists $i, j \in \cal I(\cal P)\backslash\cal I_2(\cal P)$, $i \ne j$, and $M_1\in \b M$, such that $(M_1)_{ij}$ is a factor of $\cal S_2(\cal P)$, and there is no other factors of $\cal S_2(\cal P)$ having both indices $i,j$, then we say that $(M_1)_{ij}$ is a \emph{(first) lone factor} of $\cal P$, and set $\nu_5(\cal P )=1$; otherwise $\nu_5(\cal P)=0$.  
	
	(ii)  Suppose $\nu_5(\cal P)=1$ with a lone factor $(M_1)_{ij}$. If there exists $u, v \in \cal I(\cal P)\backslash \cal I_2(\cal P )$, $u \ne v$, and $M_2,M_3\in \b M$, such that $\{u,v\}\cap \{i,j\}=\emptyset$, and $\cal S_2(\cal P)$ contains either $(M_2)_{uv}$ or $(M_2M_3)_{uv}$, and there is no other factor of $\cal S_2(\cal P)$ having both indices $u,v$, then we say that $(M_2)_{uv}$ or $(M_2M_3)_{uv}$ is a \emph {(second) lone factor} of $\cal P$, and set $\nu_6(\cal P)=1$; otherwise $\nu_6(\cal P)=0$. 
	
	(iii) Suppose there exists distinct $i,j,u,v \in \cal I(\cal P)\backslash\cal I_2(\cal P)$, such that $i$ appears exactly three times in $\cal P$, in the form $(M_1)_{ij}$, $(M_2)_{iu}$, $(M_3)_{iv}$, where $M_1,M_2,M_3 \in \b M$. We say that $(M_1)_{ij}$, $(M_2)_{iu}$, $(M_3)_{iv}$ are happy trio of $\cal P$, and set $\cal \nu_7(\cal P)=1$; otherwise $\cal\nu_7(\cal P)=0$.  Note that $\nu_7(\cal P)=1$ implies $\nu_5(\cal P)=1$.
\end{defn}

We remark here although $(M_2M_3)_{uv}$ is called a lone factor of $\cal P$ in (ii), it is indeed a factor of $\cal S_2(\cal P)$ rather than $\cal P$, by definition.

\begin{example}
	Let us take $\cal P$ and $\cal S_2(\cal P)$ as in Example \ref{example 4.6}. We easily see that $\nu_5(\cal P)=1$ and the first  lone factor can be either $\widehat{G}_{ij}(z_1)$ or $\widehat{G}_{ik}(z_2)$ or $\widehat{G}_{iy}(z_3)$. We also have $\nu_6(\cal P)=0$, as we cannot find a second lone factor of $\cal P$. In addition, $\widehat{G}_{ij}(z_1)$, $\widehat{G}_{ik}(z_2)$ and $\widehat{G}_{iy}(z_3)$ are happy trio of $\cal P$, thus $\nu_7(\cal P)=1$.
\end{example}

\begin{rem} \label{remK}
	The lone factor and happy trio are essential objects that generate additional factors of $N^{-1/2}$ in our estimates. Heuristically, the lone factors are almost mean 0 and ``weakly" correlated with other factors in $\mathcal{P}$, and thus they create additional smallness when one take expectation of $\mathcal{P}$. Nevertheless, the mechanism to exploit this smallness is more delicate. We roughly illustrate it via the following elementary examples.
	
	(i) Let us take $\cal P_{1,r} \deq \{\widehat{G}^q_{ij}\}$ for fixed $q \geq 1$. Naively, by Lemma \ref{lem4.5}, we have 
	\begin{equation} \label{6.1}
		\bb E \cal S(\cal P_{1,q})=O_{\prec}(N^{2-q/2})\,.
	\end{equation}
	For $q \geq 2$,	we can use resolvent identity 
	\begin{align}
		\widehat{G}_{ij}=\ul{\widehat{H}\widehat{G}} \widehat{G}_{ij}-(\widehat{H}\widehat{G})_{ij}\ul{\widehat{G}}+\delta_{ij}\ul{\widehat{G}} \label{resolvent eq for G}
	\end{align} 
	and Lemma \ref{lem:cumulant_expansion} to show that the estimate \eqref{6.1} cannot be improved. More precisely, we have
	\begin{align}
		\bb E \cal S(\cal P_{1,q})&=\frac{1}{N}{\sum_{i,j}}^*{\sum_{x,y}}^*\left\{ \bb E H_{xy}\widehat{G}_{yx}\widehat{G}_{ij}^q \right\}-{\sum_{i,j}}^*\sum_{x:x\ne i} \left\{ \bb E H_{ix} \widehat{G}_{xj}\widehat{G}_{ij}^{q-1}\ul{\widehat{G}}\right \}
		\notag\\
		&=\sum_{k\geq 1} \frac{s_{k+1}}{k!}\frac{1}{N^{(k+3)/2}}{\sum_{i,j}}^*{\sum_{x,y}}^*\left\{ \bb E \frac{\partial^k( \widehat{G}_{yx}\widehat{G}_{ij}^q)}{\partial H_{xy}^k} \right\}\notag\\
		&\quad-\sum_{k\geq 1} \frac{s_{k+1}}{k!}\frac{1}{N^{(k+1)/2}} {\sum_{i,j}}^*\sum_{x:x\ne i} \left\{ \bb E \frac{\partial^{k}(\widehat{G}_{xj}\widehat{G}_{ij}^{q-1}\ul{\widehat{G}})}{\partial H_{xi}^k}\right \}\notag\\
		&\eqd \sum_{k\geq 1} A_{1,k}+\sum_{k\geq 1} A_{2,k}\label{6.2}\,.	
	\end{align}
	When $q=2$, the leading term is contained in $A_{2,1}$, and we can use \eqref{diff1} to see that
	\[
	\bb E \cal S(\cal P_{1,2})= \frac{1}{N} \sum_{i,j} \{\bb E \widehat{G}_{ii}(\widehat{G}^2)_{jj}\ul{\widehat{G}}\}+\emph{error} \asymp O(N)\,.
	\]
	When $q\geq 3$, the leading term comes from $A_{2,q-1}$ and $x=j$, i.e.
	\[
	\bb E \cal S(\cal P_{1,q})=\frac{(-1)^{q+1}s_q}{N^{q/2}}\sum_{i,j} \bb E \{\widehat{G}_{jj}^q\widehat{G}_{ii}^{q-1}\ul{\widehat{G}}\}+\emph{error} \asymp O(N^{2-q/2})\,.
	\]
	When $q=1$, as $\widehat{G}^{q-1}_{ij}=1$, we do not have the above leading contributions, and we can show that $\bb E\cal S(\cal P_{1,1})=O_{\prec}(N)$, which improves Lemma \ref{lem4.5} by a factor $N^{-1/2}$. In general, the same idea applies whenever we have a first lone factor.
	
	(ii) Now we describe the idea behind the second lone factor. Let us compare $\cal P_2\deq \{\widehat{G}_{ij}\widehat{G}_{iv}\widehat{G}_{vj}\widehat{G}_{uu}\}$ with $\cal P_3\deq \{\widehat{G}_{ij}\widehat{G}_{uv}\widehat{G}_{vj}\widehat{G}_{uu}\}$. It is easy to see that $\nu_5(\cal P_2)=\nu_5(\cal P_3)=1$, while $\nu_6(\cal P_2)=0$, $\nu_6(\cal P_2)=1$. A direct application of Lemma \ref{lem4.5} leads to
	\begin{equation} \label{6.3}
		\bb E \cal S(\cal P_2)=O_{\prec}(N^{5/2})\,, \quad \mbox{and} \quad \bb E \cal S(\cal P_3)=O_{\prec}(N^{5/2})\,.
	\end{equation}
	Similar to \eqref{6.2}, we see that
	\begin{equation} \label{6.4}
		\bb E \cal S(\cal P_2)=\frac{1}{N}\sum_{i,j,u,v}\bb E \{(\widehat{G}^2)_{jv}\widehat{G}_{vj}\widehat{G}_{ii}\widehat{G}_{uu}\ul{\widehat{G}}\}+\emph{error}=\frac{1}{N}\sum_{i,u.v}\bb E \{ (\widehat{G}^3)_{vv}\widehat{G}_{ii}\widehat{G}_{uu}\ul{\widehat{G}}\}+\emph{error}\,,
	\end{equation}
	which implies $\bb E\cal S(\cal P_2)\asymp O(N^{2})$ and improves \eqref{6.3} by a factor of $N^{-1/2}$. Because $\cal P_3$ has a second lone factor $\widehat{G}_{uv}$, we have
	\[
	\bb E \cal S(\cal P_3)=\frac{1}{N}\sum_{i,j,u,v}\bb E \{(\widehat{G}^2)_{jv}\widehat{G}_{vj}\widehat{G}_{ui}\widehat{G}_{uu}\ul{\widehat{G}}\}+\emph{error}=\frac{1}{N}\sum_{i,u,v}\bb E \{ (\widehat{G}^3)_{vv}\widehat{G}_{ui}\widehat{G}_{uu}\ul{\widehat{G}}\}+\emph{error}\,,
	\]
	which implies $\bb E \cal S(\cal P_3)=O_{\prec}(N^{3/2})$ and improves \eqref{6.3} by a factor of $N^{-1}$. This suggests that the second lone factor can bring additional smallness.
	
	(iii) Lastly we remark on the happy trio. Let us denote $\cal P_4\deq \{\widehat{G}_{ij}\widehat{G}_{iv}\widehat{G}_{iu}\widehat{G}_{uu}\}$, and note that $\nu_7(\cal P_4)=1$, $\nu_7(\cal P_2)=0$. Applying lemma \ref{lem4.5} directly, we have
	\begin{equation} \label{6.5}
		\bb E \cal S(\cal P_4)=O_{\prec}(N^{5/2})\,.
	\end{equation} 
	Similar to \eqref{6.4}, we can show that
	\begin{equation} \label{6.6}
		\bb E \cal S(\cal P_4)=\frac{1}{N}\sum_{i,j,u,v}\bb E \{(\widehat{G}^2)_{jv}\widehat{G}_{iu}\widehat{G}_{ii}\widehat{G}_{uu}\ul{\widehat{G}}\}+\emph{error}\,.
	\end{equation}
	Comparing with the estimate of $\bb E\cal S(\cal P_2)$ in \eqref{6.4}, we see that the leading term on RHS of \eqref{6.6} contains two lone factors, thus further expansions shall give us better estimates. We can show that $\bb E \cal S(\cal P_4)=O_\prec(N^{3/2})$, which improves \eqref{6.5} by a factor of $N^{-1}$.
\end{rem}

In light of the above remark, we have the following estimate, which is an improvement of Lemma \ref{lem4.5}. 

\begin{lem} \label{lem5.4}
	For any $\cal P =a t^\mu N^{-\theta}\{\cal G_1\}\cdots \{\cal G_\sigma\}\in \b P$ with $\nu_0(\cal P)\leq 7-2(\nu_6(\cal P)\vee \nu_7(\cal P))$ and $d(\cal P)=0$, we have
	\[
	\bb E \cal S(\cal P)=	O_{\prec}\big(t^{\mu(\cal P)} N^{\nu(\cal P)-\theta(\cal P)-\nu_1(\cal P)/2-\nu_3(\cal P)-\nu_5(\cal P)/2-(\nu_6(\cal P)\vee \nu_7(\cal P))/2} \big)=: O_{\prec}(\cal E_1(\cal P))\,.
	\]
\end{lem}

\begin{proof}
	First, note that the case of $\nu_5(\cal P)=0$ follows from Lemma \ref{lem4.5} directly. Hence, it suffices to assume $\nu_5(\cal P)=1$ in the following. 
	
	\textit{Case 1.} Suppose $\nu_6(\cal P)=\nu_7(\cal P)=0$ and $\nu_0(\cal P)\leq 7$. Thus the lone factor of $\cal P$ is of the form $\partial_z^{\delta}\widehat{G}_{ij}\equiv \partial_{z }^{\delta}\widehat{G}_{ij}(z)$, $\delta\in \{0,1\}$, $z \in \{z_1,...,z_\sigma\}$. W.O.L.G,\,we assume $z=z_1$. Note that in the case $\nu_6(\cal P)=\nu_7(\cal P)=0$ we have $\cal E_0(\cal P)=N^{1/2}\cal E_1(\cal P)$, and thus we need to improve Lemma \ref{lem4.5} by a factor of $N^{-1/2}$. 
	
	Recall the definition of $\widetilde{\cal P}$ from \eqref{def_TT}, and we have $\bb E \cal S(\cal P)=\{\bb E \cal S(\widetilde{\cal P})\}_\sigma$. Let us abbreviate $\widetilde{\cal P}^{(A)}\deq \widetilde{P}/ \partial^{\delta}_z G_{ij}$. Using (\ref{resolvent eq for G}), 
	together with Lemma \ref{lem:cumulant_expansion} and a routine estimate of the remainder term, we have
		\begin{align}
			\bb E \widetilde{\cal P}&=\frac{1}{N}{\sum_{x,y}}^*\bb E H_{xy}\partial^{\delta}_z( \widehat{G}_{yx}\widehat{G}_{ij})\widetilde{\cal P}^{(A)}-\sum_{x:x \ne i}\bb E H_{ix}\partial_z^{\delta} (\widehat{G}_{xj}\ul{\widehat{G}}) \widetilde{\cal P}^{(A)}\notag\\
			&= \sum_{k=1}^{\ell}\frac{s_{k+1}}{k!}\frac{1}{N^{(k+3)/2}}{\sum_{x,y}}^* \bb E \frac{\partial^k (\partial_z^{\delta}( \widehat{G}_{yx}\widehat{G}_{ij}) \widetilde{\cal P}^{(A)})}{\partial H_{xy}^k}\notag\\
			&\quad -\sum_{k=1}^{\ell}\frac{s_{k+1}}{k!}\frac{1}{N^{(k+3)/2}}\sum_{x,y:x \ne i} \bb E \frac{\partial^k (\partial_z^{\delta}( \widehat{G}_{xj}\widehat{G}_{yy}) \widetilde{\cal P}^{(A)})}{\partial H_{ix}^k}+O_{\prec}(\cal E_1(\cal P)N^{-\nu(\cal P)})\notag\\
			&=:\sum_{k=1}^{\ell}\bb E L_{k}^{(1)}+\sum_{k=1}^{\ell} \bb E L_{k}^{(2)} +O_{\prec}(\cal E_1(\cal P)N^{-\nu(\cal P)}) \label{5.2}
		\end{align}
	for some fixed $\ell \in \bb N_+$. We shall first observe the cancellations between $\bb E L_{1}^{(1)}$ and $\bb E L_{1}^{(2)}$. By \eqref{diff1}, we have
		\begin{align*}
			\bb E L^{(1)}_1&=\frac{1}{N^2}\sum_{x,y}(1+\delta_{xy})\bb E\frac{\partial (\partial_z^{\delta}( \widehat{G}_{yx}\widehat{G}_{ij}) \widetilde{\cal P}^{(A)})}{\partial H_{xy}}-\frac{2}{N^2}\sum_{x}\bb E\frac{\partial (\partial_z^{\delta}( \widehat{G}_{xx}\widehat{G}_{ij}) \widetilde{\cal P}^{(A)})}{\partial H_{xx}}\\
			&=-\frac{1}{N^2}\sum_{x,y} \bb E\partial_z^{\delta}( \widehat{G}_{xx}\widehat{G}_{yy}\widehat{G}_{ij}) \widetilde{\cal P}^{(A)}-\frac{1}{N^2}\sum_{x,y} \bb E\partial_z^{\delta}( \widehat{G}_{yx}\widehat{G}_{xy}\widehat{G}_{ij}) \widetilde{\cal P}^{(A)}\\
			&\ \ -\frac{1}{N^2}\sum_{x,y} \bb E\partial_z^{\delta}( \widehat{G}_{yx}(\widehat{G}_{ix}\widehat{G}_{yj}+\widehat{G}_{iy}\widehat{G}_{xj})) \widetilde{\cal P}^{(A)}+\frac{1}{N^2}\sum_{x,y}\bb E(1+\delta_{xy})\partial_z^{\delta}( \widehat{G}_{yx}\widehat{G}_{ij})\frac{\partial  \widetilde{\cal P}^{(A)}}{\partial H_{xy}}\\
			&\ \ -\frac{2}{N^2}\sum_{x\in \cal I_1(\cal P)}\bb E\frac{\partial (\partial_z^{\delta}( \widehat{G}_{xx}\widehat{G}_{ij}) \widetilde{\cal P}^{(A)})}{\partial H_{xx}}-\frac{2}{N^2}\sum_{x: x \notin \cal I_1(\cal P)}\bb E\frac{\partial (\partial_z^{\delta}( \widehat{G}_{xx}\widehat{G}_{ij}) \widetilde{\cal P}^{(A)})}{\partial H_{xx}}\eqd \sum_{p=1}^6\bb E L^{(1)}_{1,p}
		\end{align*}
	where we applied the convention in (\ref{convention on diagonal dderivatie}). Similarly, we have
	\begin{equation*}
		\begin{aligned}
			\bb EL_{1}^{(2)}&=\frac{1}{N^2}\sum_{x,y} \bb E\partial_z^{\delta}( \widehat{G}_{xx}\widehat{G}_{ij}\widehat{G}_{yy}) \widetilde{\cal P}^{(A)}+\frac{1}{N^2}\sum_{x,y} \bb E\partial_z^{\delta}( \widehat{G}_{xj}(\widehat{G}_{xi}\widehat{G}_{yy}+\widehat{G}_{yi}\widehat{G}_{xy}+\widehat{G}_{yx}\widehat{G}_{iy})) \widetilde{\cal P}^{(A)}\\
			&\ \ -\frac{1}{N^2}\sum_{x,y}\bb E(1+\delta_{xi})\partial_z^{\delta}( \widehat{G}_{xj}\widehat{G}_{yy})\frac{\partial  \widetilde{\cal P}^{(A)}}{\partial H_{xi}}-\frac{2}{N^2}\sum_{y}\bb E\frac{\partial (\partial_z^{\delta}( \widehat{G}_{ij}\widehat{G}_{yy}) \widetilde{\cal P}^{(A)})}{\partial H_{ii}}\eqd \sum_{p=1}^4\bb E L^{(2)}_{1,p}\,.
		\end{aligned}
	\end{equation*}
	Note the cancellation between $\bb E L_{1,1}^{(1)}$ and $\bb E L_{1,1}^{(2)}$. 
	
	The estimates of $\bb EL_{1,p}^{(1)}$ and $\bb EL_{1,p}^{(2)}$ for $p \geq 2$ can be handled using Lemmas \ref{lem:intermediate} and \ref{lem4.5}. To start with, we consider
	\[
	\bb E L_{1,2}^{(1)}=-\frac{1}{N^2}\bb E \partial_z^{\delta}(\tr \widehat{G}^2)\widehat{G}_{ij}\widetilde{\cal P}^{(A)}-\frac{\delta}{N^2}\tr \widehat{G}^2 \partial_z^\delta \widehat{G}_{ij}\widetilde{\cal P}^{(A)}\,.
	\]
	By Lemma \ref{lem:intermediate} we have
	\[
	\cal S(\partial_z^\delta \widehat{G}_{ij}\widetilde{\cal P}^{(A)}) \prec t^{\mu(\cal P)}N^{\nu(\cal P)-\theta(\cal P)-\nu_1(\cal P)/2-\nu_3(\cal P)-d(\cal P)/2} |\eta_1|^{-7/2}|\eta_2|^{-7/2} \cdots |\eta_\sigma|^{-7/2}
	\]
	and together with the bound $N^{-2}\tr \partial^\delta_z\widehat{G}^2\prec N^{-1}|\eta_1|^{-1-\delta}$ we get
	\begin{align*}
		\bb E\cal S(L_{1,2}^{(1)}) & \prec t^{\mu(\cal P)}N^{\nu(\cal P)-\theta(\cal P)-\nu_1(\cal P)/2-\nu_3(\cal P)-d(\cal P)/2-1} |\eta_1|^{-9/2}|\eta_2|^{-7/2} \cdots |\eta_\sigma|^{-7/2}\\
		&=\cal E_1(\cal P)N^{-1/2}|\eta_1|^{-9/2}|\eta_2|^{-7/2} \cdots |\eta_\sigma|^{-7/2}\,.
	\end{align*}
	By Lemma \ref{lem:integration} we get $\{\bb E\cal S(L_{1,2}^{(1)}) \}_\sigma \prec \cal E_1(\cal P)N^{-1/2}$. A similar argument shows that
	\[
	\bb E \cal S(L_{1,3}^{(1)})\prec \cal E_1(\cal P)N^{-3/2}|\eta_1|^{-11/2}|\eta_2|^{-7/2} \cdots |\eta_\sigma|^{-7/2}\prec\cal E_1(\cal P)N^{-11/8}|\eta_1|^{-5}|\eta_2|^{-7/2} \cdots |\eta_\sigma|^{-7/2}\,,
	\]
	where in the second step we used $|\eta|\geq N^{-1/4}$ for $z=E+\ii \eta\in \b D$. By Lemma \ref{lem:integration} we get $\{\bb E\cal S(L_{1,3}^{(1)}) \}_\sigma \prec \cal E_1(\cal P)N^{-11/8}$. Analogously, we can also show that  $\{\bb E\cal S(L_{1,4}^{(1)}) \}_\sigma \prec \cal E_1(\cal P)N^{-11/8}$.  
	
	Next, we consider $\bb E L_{1,5}^{(1)}$. By \eqref{diff1}, one can show that
$
	\bb E L_{1,5}^{(1)}=\sum_{q=1}^n \bb E  \widetilde{\cal P}^{(1)}_{1,5,q}
$
	for some fixed $n \in \bb N_+$, where each $\cal P^{(1)}_{1,5,q}\in \b P$ satisfies $\mu\big(\cal P^{(1)}_{1,5,q}\big)=\mu(\cal P)$, $\nu\big(\cal P^{(1)}_{1,5,q}\big)=\nu(\cal P)$, $\theta\big(\cal P^{(1)}_{1,5,q}\big)=\theta(\cal P)+2$. In addition, note that by \eqref{diff1}, when we apply the differential $\partial/\partial H_{xx}$, the indices $\nu_1,\nu_3$ will not decrease. That is, $\nu_1\big(\cal P^{(1)}_{1,5,q}\big)\geq \nu_1(\cal P)$, $\nu_3(\cal P^{(1)}_{1,5,q})\geq\nu_3(\cal P)$. Moreover, we have $\nu_*(\cal P^{(1)}_{1,5,q})=0$. Thus Lemma \ref{lem4.5} implies
	\begin{equation*}
		\{\bb E \cal S(L_{1,5}^{(1)})\}_\sigma=	\sum_{q=1}^n \{\bb E \cal S(\widetilde{\cal P}^{(1)}_{1,5,q})\}_\sigma=O_{\prec}(\cal E_1(\cal P)N^{-3/2})\,.
	\end{equation*}
	As another illustration, we see that
	\[
	\bb E L_{1,6}^{(1)}=\sum_{q=1}^n \sum_{x: x \notin \cal I_1(\cal P)} \bb E  \widetilde{\cal P}^{(1)}_{1,6,q}
	\]
	for some fixed $n \in \bb N_+$, where each $\cal P^{(1)}_{1,6,q}\in \b P$ satisfies $\mu\big(\cal P^{(1)}_{1,6,q}\big)=\mu(\cal P)$, $\nu\big(\cal P^{(1)}_{1,6,q}\big)=\nu(\cal P)+1$, $\theta\big(\cal P^{(1)}_{1,6,q}\big)=\theta(\cal P)+2$, $\nu_1\big(\cal P^{(1)}_{1,6,q}\big)\geq \nu_1(\cal P)$, $\nu_3(\cal P^{(1)}_{1,6,q})\geq \nu_3(\cal P)$ and $\nu_*(\cal P_{1,6,q}^{(1)})=0$. Thus Lemma \ref{lem4.5} implies
	\begin{equation*}
		\{	\bb E \cal S(L_{1,6}^{(1)})\}_\sigma=	\sum_{q=1}^n \{\bb E \cal S(\widetilde{\cal P}^{(1)}_{1,6,q})\}_\sigma=O_{\prec}(\cal E_1(\cal P)N^{-1/2})\,.
	\end{equation*}
	One can use similar arguments to estimate the remaining terms in $\bb EL_{1}^{(1)}$ and $\bb EL_{1}^{(2)}$, and show that
	\begin{equation} \label{L_1}
		\{\bb E \cal S(L_1^{(1)})\}_\sigma+\{\bb E \cal S(L_1^{(2)})\}_\sigma =O_{\prec}(\cal E_0(\cal P)N^{-1})=O_{\prec}(\cal E_1(\cal P)N^{-1/2})\,.
	\end{equation}
	The above suggests that estimate is improved by $N^{-1}$ from the trivial estimate $\cal E_0(\cal P)$. 
	
	Now let us deal with $\bb E L_{k}^{(1)}$ and $\bb E L_{k}^{(2)}$ for $k \geq 2$.  We shall check  the latter in detail since it is more representative. We have
	\begin{equation} \label{6.14}
		\begin{aligned}
			\bb EL_{k}^{(2)}=&-\frac{s_{k+1}}{k!}\frac{1}{N^{(k+3)/2}}\sum_{x:x \notin \cal I_1(\cal P ) } \sum_{y}\bb E \frac{\partial^k (\partial_z^{\delta}( \widehat{G}_{xj}\widehat{G}_{yy})\widetilde{\cal P}^{(A)} )}{\partial H_{ix}^k}\\
			&-\frac{s_{k+1}}{k!}\frac{1}{N^{(k+3)/2}}\sum_{x:x \in \cal I_1(\cal P )\backslash\{i,j\} } \sum_{y}\bb E \frac{\partial^k (\partial_z^{\delta}( \widehat{G}_{xj}\widehat{G}_{yy}) \widetilde{\cal P}^{(A)})}{\partial H_{ix}^k}\\
			&-\frac{s_{k+1}}{k!}\frac{1}{N^{(k+3)/2}}\sum_{y}\bb E \frac{\partial^k (\partial_z^{\delta}( \widehat{G}_{jj}\widehat{G}_{yy}) \widetilde{\cal P}^{(A)})}{\partial H_{ij}^k} \eqd \sum_{p=1}^3 \bb E L_{k,p}^{(2)}\,.
		\end{aligned}
	\end{equation}
	In the sequel we shall repeatedly use the following observation: from \eqref{diff1} we see that, when $a\ne b$ and $\{a,b\}\ne \{c,d\}$, every term in 
	\begin{equation} \label{444}
		\frac{\partial^k (\sum_{i_1,...,i_n}\partial_z^{\delta_1} G^{(1)}_{ai_1}\partial_z^{\delta_2} G^{(2)}_{i_1i_2}\cdots\partial_z^{\delta_n} G^{(n)}_{i_{n-1}i_n}\partial_z^{\delta_{n+1}} G^{(n+1)}_{i_nb}) }{\partial H_{cd}^k}\,,  \quad \delta_1,...,\delta_n \in \{0,1\}
	\end{equation}
	contains at least one off-diagonal factor. As $x$ is distinct from $\cal I_1(\cal P)$ in $\bb E L^{(2)}_{k,1}$, we see that
	\[
	\bb E L^{(2)}_{k,1}=\sum_{q=1}^n \sum_{x:x \notin \cal I_1(\cal P ) } \sum_{y}\bb E  \widetilde{\cal P}_{k,1,q}^{(2)}\,,
	\]
	for some fixed $n \in \bb N_+$, where each $\cal P^{(2)}_{k,1,q}\in \b P$ satisfies $\mu\big(\cal P^{(2)}_{k,1,q}\big)=\mu(\cal P)$, $\nu\big(\cal P^{(2)}_{k,1,q}\big)=\nu(\cal P)+2$, $\theta\big(\cal P^{(2)}_{k,1,q}\big)=\theta(\cal P)+(k+3)/2$, $\nu_1\big(\cal P^{(2)}_{k,1,q}\big)\geq \nu_1(\cal P)$, $\nu_3(\cal P^{(2)}_{k,1,q})\geq\nu_3(\cal P)$. Let us fix $q$, and write $\cal P^{(2)}_{k,1,q} =a^{(2)} t^\mu N^{-\theta^{(2)}}\{\cal G_1^{(2)}\}\cdots \{\cal G_\sigma^{(2)}\}$. Our assumption $\nu_0(\cal P)\leq 7$ implies
	\begin{align}
		\nu_*(\cal P^{(2)}_{k,1,q})&\leq \sum_{i=1}^{\sigma}(\nu_0(\mathcal{G}_i^{(2)})-\nu_0(\mathcal{P})-3)_+\leq\sum_{i=1}^{\sigma} (\nu_0(\mathcal{G}_i^{(2)})-\nu_0(\mathcal{G}_i)-3)_+\notag\\		
		&= 		\sum_{i=1}^\sigma({\nu}_1(\mathcal{G}_i^{(2)})+2\nu_3(\mathcal{G}_i^{(2)})-{\nu}_1(\mathcal{G}_i)-2\nu_3(\cal G_i)-3)_+\notag\\
		&\leq \sum_{i=1}^\sigma\big(({\nu}_1(\mathcal{G}_i^{(2)})-{\nu}_1(\mathcal{G}_i)-3)_+ +2\nu_3(\cal G_i^{(2)})-2\nu_3(\cal G_i)\big)\label{33333}\,,
	\end{align}
	where in the last step we used $\nu_3(\mathcal{G}_i^{(2)})\geq \nu_3(\cal G_i)$ for $i \in \{1,...,n\}$. Note that by \eqref{6.14}, we get $\cal P_{k,1,q}^{(2)}$ through $k$ differentials. This motivates us to define for $r=0,...,k$, the term $\cal P_{k,1,q}^{(2,r)}$ such that 
	\[
	\cal P_{k,1,q}^{(2,0)}=-\frac{s_{k+1}}{k!}\frac{1}{N^{(k+3/2)}}\partial_z^{\delta}( \widehat{G}_{xj}\widehat{G}_{yy})\widetilde{\cal P}^{(A)}\,,\quad  \cal P_{k,1,q}^{(2,k)}=\cal P^{(2)}_{k,1,q}\,,
	\]
	and $ \cal P_{k,1,q}^{(2,r)}$ is a term in
	$$
	-\frac{s_{k+1}}{k!} \frac{1}{N^{(k+3)/2}}\frac{ \partial \cal P_{k,1,q}^{(2,r-1)}}{\partial H_{ix}}
	$$
	for all $r=1,...,k$. We write $\cal P^{(2,r)}_{k,1,q} =a^{(2,r)} t^\mu N^{-\theta^{(2,r)}}\{\cal G_1^{(2,r)}\}\cdots \{\cal G_\sigma^{(2,r)}\}$. By \eqref{diff1} and our observation concerning \eqref{444}, it is not hard to check that
	\begin{equation} \label{66666}
		{\nu}_1(\mathcal{G}_1^{(2,0)})={\nu}_1(\mathcal{G}_1)\,,\cdots\,, {\nu}_1(\mathcal{G}_\sigma^{(2,0)})={\nu}_1(\mathcal{G}_\sigma)\,,\quad {\nu}_1(\cal P^{(2,1)}_{k,1,q})={\nu}_1(\cal P)
	\end{equation}
	and
	\begin{equation} \label{6666}
		\sum_{i=1}^\sigma({\nu}_1(\mathcal{G}_i^{(2,r)})-{\nu}_1(\mathcal{G}^{(2,r-1)}_i))_+\leq 1+{\nu}_1(\cal P^{(2,r)}_{k,1,q})-{\nu}_1(\cal P^{(2,r-1)}_{k,1,q})\,,
	\end{equation}
	for all $r=1,...,k$. The term $+1$ on RHS of the above comes from the worst case, where a factor of the form $(\widehat{G}(z_1)\widehat{G}(z_2))_{ux}$, $u\ne x$ in $\cal P^{(2,r-1)}_{k,1,q}$ was differentiated  by $H_{ix}$, and we get $\widehat{G}(z_1)_{ui}$ and $(\widehat{G}(z_1)\widehat{G}(z_2))_{xx}$ in $\cal P_{k,q,1}^{(2,r)}$. In this case, $({\nu}_1(\mathcal{G}_1^{(2,r)})-{\nu}_1(\mathcal{G}^{(2,r-1)}_1))_++\cdots+ ({\nu}_1(\mathcal{G}_\sigma^{(2,r)})-{\nu}_1(\mathcal{G}^{(2,r-1)}_\sigma))_+=1$, while ${\nu}_1(\cal P^{(2,r)}_{k,1,q})-{\nu}_1(\cal P^{(2,r-1)}_{k,1,q})=0$. Thus \eqref{66666} and \eqref{6666} imply
	$
	({\nu}_1(\mathcal{G}_1^{(2)})-{\nu}_1(\mathcal{G}_1))_++\cdots ({\nu}_1(\mathcal{G}_\sigma^{(2)})-{\nu}_1(\mathcal{G}_\sigma))_+ \leq k+{\nu}_1(\cal P^{(2)}_{k,1,q})-{\nu}_1(\cal P),
	$
	and as a result
	\begin{align}
		\sum_{i=1}^\sigma({\nu}_1(\mathcal{G}_i^{(2)})-{\nu}_1(\mathcal{G}_i)-3)_+ \leq (k+{\nu}_1(\cal P^{(2)}_{k,1,q})-{\nu}_1(\cal P)-3)_+ \,.\label{022601}
	\end{align}
	Similarly, we can also show that
	\begin{align}
	\sum_{i=1}^\sigma\big(	\nu_3(\cal G_i^{(2)})-\nu_3(\cal G_i) \big)\leq 2(\nu_3(\cal P^{(2)}_{k,1,q})-\nu_3(\cal P))\,. \label{022710}
	\end{align}
	\nc By \eqref{33333}, \eqref{022601} and \eqref{022710}, we have
	\[
	\nu_*(\cal P^{(2)}_{k,1,q})\leq 4(\nu_3(\cal P^{(2)}_{k,1,q})-\nu_3(\cal P))+({\nu}_1(\cal P^{(2)}_{k,1,q})-{\nu}_1(\cal P)+k-3)_+\,,
	\]
	and as a result
	\begin{equation}\label{5.7}
		\nu_*(\cal P^{(2)}_{k,1,q})/8\leq \nu_3(\cal P^{(2)}_{k,1,q})-\nu_3(\cal P)+({\nu}_1(\cal P^{(2)}_{k,1,q})-{\nu}_1(\cal P)+k-3)_+/2\,.
	\end{equation}
	The above relation, together with Lemma \ref{lem4.5} imply
	\begin{equation} \label{5.6}
		\begin{aligned}
			 \{\bb E \cal S(L^{(2)}_{k,1})\}_\sigma&=\sum_{q=1}^n \{\bb E \cal S(\widetilde{\cal  P}_{k,1,q}^{(2)}) \}_\sigma\\
			&=O_{\prec}\big(t^{\mu(\cal P)} N^{\nu(\cal P)+2-(\theta(\cal P)+(k+3)/2)-\nu_1(\cal P)/2-\nu_3(\cal P)+(k-3)_+/2} \big)\\
			&=O_{\prec}(\cal E_1(\cal P)N^{-1/2})\,.
		\end{aligned}
	\end{equation}	
	From \eqref{diff1}, we also see that
	\[
	\bb E L^{(2)}_{k,2}=\sum_{q=1}^n  \sum_{y}\bb E   \widetilde{\cal P}_{k,2,q}^{(2)}\,,
	\]
	for some fixed $n \in \bb N_+$, where each $\cal P^{(2)}_{k,2,q}\in \b P$ satisfies $\mu\big(\cal P^{(2)}_{k,2,q}\big)=\mu(\cal P)$, $\nu\big(\cal P^{(2)}_{k,2,q}\big)=\nu(\cal P)+1$, $\theta\big(\cal P^{(2)}_{k,2,q}\big)=\theta(\cal P)+(k+3)/2$, $\nu_1\big(\cal P^{(2)}_{k,2,q}\big)\geq \nu_1(\cal P)-k$, $\nu_3(\cal P^{(2)}_{k,2,q})\geq \nu_3(\cal P)$. Similarly to \eqref{5.7}, we can show that
	\begin{align*}
	\nu_*(\cal P^{(2)}_{k,1,q})/8&\leq  \nu_3(\cal P^{(2)}_{k,1,q})-\nu_3(\cal P)+({\nu}_1(\cal P^{(2)}_{k,1,q})-{\nu}_1(\cal P)+k-3)_+/2\\
	&\leq\nu_3(\cal P^{(2)}_{k,1,q})-\nu_3(\cal P)+(\nu_1(\cal P^{(2)}_{k,1,q})-\nu_1(\cal P)+k)/2\,.
	\end{align*}
	Thus Lemma \ref{lem4.5} shows
	\begin{equation} 
		\begin{aligned}
			&\quad\{\bb E \cal S(L^{(2)}_{k,2})\}_\sigma=\sum_{q=1}^n \{\bb E\cal S(\widetilde{\cal  P}_{k,2,q}^{(2)})\}_\sigma\\
			& =O_{\prec}\big(t^{\mu(\cal P)} N^{\nu(\cal P)+1-(\theta(\cal P)+(k+3)/2)-\nu_1(\cal P)/2-\nu_3(\cal P)+k/2} \big)=O_{\prec}(\cal E_1(\cal P))\,.
		\end{aligned}
	\end{equation}	
	For $\bb E L_{k,3}^{(2)}$, since $\partial_z^{\delta } G_{ij}$ is a lone factor of $\cal P$, there is no factor with both indices $i,j$ in $\cal S_2(\partial_z^{\delta}( \widehat{G}_{jj}\widehat{G}_{yy}) \widetilde{\cal P}^{(A)})$. Hence we again use our observation about the differential \eqref{444}, and we see that
	\[
	\bb E L^{(2)}_{k,3}=\sum_{q=1}^n  \sum_{y}\bb E  \widetilde{\cal P}_{k,3,q}^{(2)}\,,
	\]
	for some fixed $n \in \bb N_+$, where each $\cal P^{(3)}_{k,3,q}\in \b P$ satisfies $\mu\big(\cal P^{(3)}_{k,3,q}\big)=\mu(\cal P)$, $\nu\big(\cal P^{(3)}_{k,3,q}\big)=\nu(\cal P)+1$, $\theta\big(\cal P^{(3)}_{k,3,q}\big)\geq \theta(\cal P)+(k+3)/2$, $\nu_1\big(\cal P^{(2)}_{k,3,q}\big)\geq \nu_1(\cal P)-1$, $\nu_3(\cal P^{(2)}_{k,3,q})\geq \nu_3(\cal P)$, and
	\begin{align*}
		\nu_*(\cal P^{(2)}_{k,1,q})/8&\leq \nu_3(\cal P^{(2)}_{k,1,q})-\nu_3(\cal P)+(\nu_1(\cal P^{(2)}_{k,1,q})-\nu_1(\cal P)+k-3)_+/2\\
		&\leq \nu_3(\cal P^{(2)}_{k,1,q})-\nu_3(\cal P)+(\nu_1(\cal P^{(2)}_{k,1,q})-\nu_1(\cal P)+1)/2+(k-4)_+/2
	\end{align*}
	Thus Lemma \ref{lem4.5} shows
	\begin{equation}  \label{5.9}
		\begin{aligned}
			\{	\bb E \cal S(L^{(2)}_{k,3})\}_\sigma&=\sum_{q=1}^n \{\bb E\cal S( \widetilde{\cal  P}_{k,3,q}^{(2)})\}_\sigma\\ &=O_{\prec}\big(t^{\mu(\cal P)} N^{\nu(\cal P)+1-(\theta(\cal P)+(k+3)/2)-(\nu_1(\cal P)/2+1/2)-\nu_3(\cal P)+(k-4)_+/2} \big)\\
			&=O_{\prec}(\cal E_1(\cal P))\,.
		\end{aligned}
	\end{equation}
	Note the criticality of the assumption that $\partial_z^{\delta } G_{ij}$ is a lone factor of $\cal P$. Without this condition, we can only have $\nu_1\big(\cal P^{(2)}_{k,3,q}\big)\geq \nu_1(\cal P)-1-k$ for the parameter $\nu_1$, and in this case \eqref{5.9} fails. By \eqref{5.6} -- \eqref{5.9} we have
	\begin{equation} \label{333}
		\{\bb E \cal S(L_{k}^{(2)}) \}_{\sigma}=O_{\prec}(\cal E_1(\cal P))=O_{\prec}(\cal E_0(\cal P)N^{-1/2})=O_{\prec}(\cal E_1(\cal P))
	\end{equation}
	for $k\geq 2$. A similar argument shows that
	\begin{equation} \label{4444}
		\{\bb E \cal S(L_{k}^{(1)}) \}_\sigma=O_{\prec}(\cal E_0(\cal P)N^{-1/2})=O_{\prec}(\cal E_1(\cal P))
	\end{equation}
	for $k\geq 2$. Note that in the above two relations,  we improve the trivial bound $\cal E_0(\cal P)$ by a factor of $N^{-1/2}$. Inserting \eqref{L_1}, \eqref{333} and \eqref{4444} into \eqref{5.2}, together with Lemma \ref{lem:integration} we arrive at
	\[
	\bb E \cal S(\cal P)=\{\bb E \cal S(\widetilde{\cal P})\}_\sigma=O_{\prec}(\cal E_1(\cal P))
	\]
	as desired.
	
	\textit{Case 2.} Suppose $\nu_5(\cal P)=\nu_6(\cal P)=1$, $\nu_0(\cal P)\leq 5$\,, and the lone factors of $\cal P$ are $\partial^{\delta_1}_{z_r} \widehat{G}_{ij}\equiv \partial^{\delta_1}_{z_r} \widehat{G}_{ij}(z_r)$ and $\partial^{\delta_2}_{z_s} \widehat{F}_{uv}\equiv \partial^{\delta_2}_{z_s} \widehat{G}_{uv}(z_s)$, where $z_r,z_s \in \{z_1,...,z_\sigma\}$, $\delta_{1},\delta_2 \in \{0,1\}$. Note that in the case $\nu_6(\cal P)=1$ we have  $\cal E_0(\cal P)=N\cal E_1(\cal P)$, and thus we need to improve Lemma \ref{lem4.5} by a factor of $N^{-1}$. 
	
	Let us abbreviate $\widetilde{\cal P}^{(B)}\deq \widetilde{\cal P}/( \partial^{\delta_1}_{z_r} \widehat{G}_{ij}\partial^{\delta_2}_{z_s} \widehat{F}_{uv})$. Similar to \eqref{5.2}, we have
	\begin{equation} \label{5.12}
		\begin{aligned}
			\bb E \widetilde{\cal P}&=\frac{1}{N}{\sum_{x,y}}^* \bb E H_{xy} \partial_{z_r}^{\delta_1}( \widehat{G}_{yx}\widehat{G}_{ij}) \partial^{\delta_2}_{z_s} \widehat{F}_{uv}\widetilde{\cal P}^{(B)}-\frac{1}{N}\sum_{x,y:x \ne i} \bb E \partial_{z_r}^{\delta_1}( \widehat{G}_{xj}\widehat{G}_{yy})\partial^{\delta_2}_{z_s} \widehat{F}_{uv}\widetilde{\cal P}^{(B)}\\
			&=
			\sum_{k=1}^{\ell}\frac{s_{k+1}}{k!}\frac{1}{N^{(k+3)/2}}{\sum_{x,y}}^* \bb E \frac{\partial^k (\partial_{z_r}^{\delta_1}( \widehat{G}_{yx}\widehat{G}_{ij}) \partial^{\delta_2}_{z_s} \widehat{F}_{uv}\widetilde{\cal P}^{(B)})}{\partial H_{xy}^k}
			\\ &\quad-\sum_{k=1}^{\ell}\frac{s_{k+1}}{k!}\frac{1}{N^{(k+3)/2}}\sum_{x,y:x \ne i} \bb E \frac{\partial^k (\partial_{z_r}^{\delta_1}( \widehat{G}_{xj}\widehat{G}_{yy})\partial^{\delta_2}_{z_s} \widehat{F}_{uv}\widetilde{\cal P}^{(B)})}{\partial H_{ix}^k}+O_{\prec}(\cal E_1(\cal P)N^{-\nu(\cal P)})\\
			&=:\sum_{k=1}^{\ell}\bb E L_{k}^{(3)}+\sum_{k=1}^{\ell} \bb E L_{k}^{(4)} +O_{\prec}(\cal E_1(\cal P)N^{-\nu(\cal P)})
		\end{aligned}
	\end{equation}
	for some fixed $\ell \in \bb N_+$. By \eqref{diff1}, we have
		\begin{align*}
			\bb E L^{(3)}_1
			&=-\frac{1}{N^2}\sum_{x,y} \bb E\partial_{z_r}^{\delta_1}( \widehat{G}_{xx}\widehat{G}_{yy}\widehat{G}_{ij})\partial^{\delta_2}_{z_s} \widehat{F}_{uv} \widetilde{\cal P}^{(B)}\\
			&\ \  -\frac{1}{N^2}\sum_{x,y} \bb E\partial_{z_r}^{\delta}( \widehat{G}_{yx}(\widehat{G}_{xy}\widehat{G}_{ij}+\widehat{G}_{ix}\widehat{G}_{yj}+\widehat{G}_{iy}\widehat{G}_{xj})) \partial^{\delta_2}_{z_s} \widehat{F}_{uv}\widetilde{\cal P}^{(B)}\\
			&\ \ +\frac{1}{N^2}\sum_{x,y}\bb E\partial_{z_r}^{\delta_1}( \widehat{G}_{yx}\widehat{G}_{ij})\frac{\partial  (\partial^{\delta_2}_{z_s} \widehat{F}_{uv} \widetilde{\cal P}^{(B)})}{\partial H_{xy}}(1+\delta_{xy})\\
			&\ \ -\frac{2}{N^2}\sum_{x\in \cal I_1(\cal P)}\bb E\frac{\partial (\partial_{z_r}^{\delta_1}( \widehat{G}_{xx}\widehat{G}_{ij})\partial^{\delta_2}_{z_s} \widehat{F}_{uv} \widetilde{\cal P}^{(B)})}{\partial H_{xx}}\\
			&\ \ -\frac{2}{N^2}\sum_{x: x \notin \cal I_1(\cal P)}\bb E\frac{\partial (\partial_{z_r}^{\delta}( \widehat{G}_{xx}\widehat{G}_{ij})\partial^{\delta_2}_{z_s} \widehat{F}_{uv} \widetilde{\cal P}^{(B)})}{\partial H_{xx}}\eqd \sum_{p=1}^5\bb E L^{(3)}_{1,p}
		\end{align*}
	and similarly
		\begin{align*}
			\bb EL_{1}^{(4)}&=\frac{1}{N^2}\sum_{x,y} \bb E\partial_{z_r}^{\delta_1}( \widehat{G}_{xx}\widehat{G}_{ij}\widehat{G}_{yy})\partial^{\delta_2}_{z_s} \widehat{F}_{uv} \widetilde{\cal P}^{(B)}\\
			&+\frac{1}{N^2}\sum_{x,y} \bb E\partial_{z_r}^{\delta_1}( \widehat{G}_{xj}(\widehat{G}_{xi}\widehat{G}_{yy}+\widehat{G}_{yi}\widehat{G}_{xy}+\widehat{G}_{yx}\widehat{G}_{iy})) \partial^{\delta_2}_{z_s} \widehat{F}_{uv}\widetilde{\cal P}^{(B)}\\
			& -\frac{1}{N^2}\sum_{x,y}\bb E\partial_{z_r}^{\delta_1}( \widehat{G}_{xj}\widehat{G}_{yy})\frac{\partial (\partial^{\delta_2}_{z_s} \widehat{F}_{uv} \widetilde{\cal P}^{(B)})}{\partial H_{xi}}(1+\delta_{xi})\\
			&-\frac{2}{N^2}\sum_{y}\bb E\frac{\partial (\partial_{z_r}^{\delta_1}( \widehat{G}_{ij}\widehat{G}_{yy})\partial^{\delta_2}_{z_s} \widehat{F}_{uv} \widetilde{\cal P}^{(B)})}{\partial H_{ii}} \eqd \sum_{p=1}^4\bb E L^{(4)}_{1,p}\,.
		\end{align*}
	We again see that there is a cancellation between $\bb EL_{1,1}^{(3)}$ and $\bb EL_{1,1}^{(4)}$. Similar to \eqref{L_1}, we can use Lemmas \ref{lem:intermediate} and \ref{lem4.5} to show that the estimate for other terms in $\{\bb E\cal S( L_{1}^{(3)})\}_\sigma$ and $\{\bb E \cal S (L_{1}^{(4)})\}_\sigma$ are improved by $N^{-1}$ from the trivial estimate $\cal E_0(\cal P)$, i.e.
	\begin{equation} \label{5.15}
		\{\bb E \cal S(L_1^{(3)})\}_\sigma+\{\bb E \cal S(L_1^{(4)})\}_\sigma=O_{\prec}(\cal E_0(\cal P)N^{-1}) =O_{\prec}(\cal E_1(\cal P)N^{-1/2})
		\,.
	\end{equation}
	
	Next let us deal with $\bb EL_{k}^{(3)}$ and $\bb E L_{k}^{(4)}$ for $k \geq 2$. We shall give a careful treatment of the latter. We have
	\begin{equation*}
		\begin{aligned}
			\bb EL_{k}^{(4)}=&-\frac{s_{k+1}}{k!}\frac{1}{N^{(k+3)/2}}\sum_{x:x \notin \cal I_1(\cal P ) } \sum_{y}\bb E\frac{\partial^{k} (\partial_{z_r}^{\delta_1}( \widehat{G}_{xj}\widehat{G}_{yy})\partial^{\delta_2}_{z_s} \widehat{F}_{uv}\widetilde{\cal P}^{(B)})}{\partial H_{ix}^{k}}\\
			&-\frac{s_{k+1}}{k!}\frac{1}{N^{(k+3)/2}}\sum_{x:x \in \cal I_1(\cal P )\backslash\{i,j,u,v\} } \sum_{y}\bb E \frac{\partial^{k} (\partial_{z_r}^{\delta_1}( \widehat{G}_{xj}\widehat{G}_{yy})\partial^{\delta_2}_{z_s} \widehat{F}_{uv}\widetilde{\cal P}^{(B)})}{\partial H_{ix}^{k}}\\
			&-\frac{s_{k+1}}{k!}\frac{1}{N^{(k+3)/2}}\sum_{y}\bb  E\frac{\partial^{k} (\partial_{z_r}^{\delta_1}( \widehat{G}_{jj}\widehat{G}_{yy})\partial^{\delta_2}_{z_s} \widehat{F}_{uv}\widetilde{\cal P}^{(B)})}{\partial H_{ij}^{k}}\\
			&-\frac{s_{k+1}}{k!}\frac{1}{N^{(k+3)/2}}\sum_{x \in \{u,v\}}\sum_{y}\bb  E\frac{\partial^{k} (\partial_{z_r}^{\delta_1}( \widehat{G}_{xj}\widehat{G}_{yy})\partial^{\delta_2}_{z_s} \widehat{F}_{uv}\widetilde{\cal P}^{(B)})}{\partial H_{ix}^{k}}\eqd \sum_{p=1}^{4}\bb E L_{k,p}^{(4)}\,.
		\end{aligned}
	\end{equation*}
	We can further split
	\begin{equation*} 
		\begin{aligned}
			\bb E L_{k,1}^{(4)}=&-\frac{s_{k+1}}{k!}\frac{1}{N^{(k+3)/2}}\sum_{n=1}^k{k \choose n}\sum_{x:x \notin \cal I_1(\cal P ) } \sum_{y}\bb E\frac{\partial^{k-n} (\partial_{z_r}^{\delta_1}( \widehat{G}_{xj}\widehat{G}_{yy})\widetilde{\cal P}^{(B)})}{\partial H_{ix}^{k-n}}
			\frac{\partial^{n} (\partial^{\delta_2}_{z_s} \widehat{F}_{uv})}{\partial H_{ix}^{n}}\\
			&-\frac{s_{k+1}}{k!}\frac{1}{N^{(k+3)/2}}\sum_{x:x \notin \cal I_1(\cal P ) } \sum_{y}\bb E\frac{\partial^{k} (\partial_{z_r}^{\delta_1}( \widehat{G}_{xj}\widehat{G}_{yy})\widetilde{\cal P}^{(B)})}{\partial H_{ix}^{k}}
			(\partial^{\delta_2}_{z_s} \widehat{F}_{uv})\eqd \bb E L_{k,1,1}^{(4)}+\bb E L_{k,1,2}^{(4)}\,.
		\end{aligned}
	\end{equation*}
	Note that when $n \geq 1$, there are at least two off-diagonal entries in each term of $\partial^{n} (\partial^{\delta_2}_{z_s} \widehat{F}_{uv})/\partial H_{ix}^{n}$. Thus we see that
	\[
	\bb E L^{(4)}_{k,1,1}=\sum_{q=1}^r \sum_{x:x \notin \cal I_1(\cal P ) } \sum_{y}\bb E  \widetilde{\cal P}_{k,1,1,q}^{(4)}\,,
	\]
	for some fixed $r \in \bb N_+$, where each $\cal P^{(4)}_{k,1,1,q}\in \b P$ satisfies $\mu\big(\cal P^{(4)}_{k,1,1,q}\big)=\mu(\cal P)$, $\nu\big(\cal P^{(4)}_{k,1,1,q}\big)=\nu(\cal P)+2$, $\theta\big(\cal P^{(4)}_{k,1,1,q}\big)=\theta(\cal P)+(k+3)/2$, $\nu_1\big(\cal P^{(4)}_{k,1,1,q}\big)\geq \nu_1(\cal P)+1$, $\nu_3(\cal P^{(4)}_{k,1,1,q})\geq \nu_3(\cal P)$. In addition, similar to \eqref{5.7}, we have
	\begin{equation*}
		\nu_*(\cal P^{(4)}_{k,1,1,q})/8\leq \nu_3(\cal P^{(4)}_{k,1,1,q})-\nu_3(\cal P)+(\nu_1(\cal P^{(4)}_{k,1,1,q})-\nu_1(\cal P)-1)/2+(k-2)_+/2\,.
	\end{equation*}
	Thus Lemma \ref{lem4.5} implies
	\begin{equation} \label{5.18}
		\begin{aligned}
		&\quad \{\bb E \cal S(L^{(4)}_{k,1,1})\}_\sigma\\
		&=O_{\prec}\big(t^{\mu(\cal P)} N^{\nu(\cal P)+2-(\theta(\cal P)+(k+3)/2)-(\nu_1(\cal P)+1)/2-\nu_3(\cal P)+(k-2)_+/2} \big)=O_\prec(\cal E_1(\cal P))
		\end{aligned}
	\end{equation}
	for $k \geq 2$. Similarly, we have
	\begin{equation} \label{5.19}
		\{\bb E \cal S(L^{(4)}_{k,1,2})\}_\sigma=O_\prec(\cal E_1(\cal P))
	\end{equation}
	for $k\geq 3$. To handle $\bb E L^{(4)}_{2,1,2}$, note that
	\[
	\bb E L^{(4)}_{2,1,2}=\sum_{q=1}^r \sum_{x:x \notin \cal I_1(\cal P ) } \sum_{y}\bb E  \widetilde{\cal P}_{2,1,2,q}^{(4)}
	\]
	for some fixed $r \in \bb N_+$, where each $\cal P^{(4)}_{2,1,2,q}\in \b P$ satisfies $\mu\big(\cal P^{(4)}_{2,1,2,q}\big)=\mu(\cal P)$, $\nu\big(\cal P^{(4)}_{2,1,2,q}\big)=\nu(\cal P)+2$, $\theta\big(\cal P^{(4)}_{2,1,2,q}\big)=\theta(\cal P)+5/2$, $\nu_1\big(\cal P^{(4)}_{2,1,2,q}\big)\geq \nu_1(\cal P)$, $\nu_3(\cal P^{(4)}_{2,1,2,q})\geq \nu_3(\cal P)$. Let us further split the sum over $q$ such that $\nu_1(\cal P^{(4)}_{2,1,2,q})=\nu_1(\cal P)$, $\nu_3(\cal P^{(4)}_{2,1,2,q})=\nu_3(\cal P)$ for $1\leq q\leq m\leq r$, and otherwise $q \geq m+1$. By Lemma \ref{lem4.5}, one has
	\begin{equation*} 
		\sum_{q=m+1}^r\{\bb E \cal S\big(\widetilde{\cal P}_{2,1,2,q}^{(4)}\big)\}_\sigma=O_\prec(\cal E_1(\cal P))\,.
	\end{equation*}
	For $q \leq m$, directly applying Lemma \ref{lem4.5} only gives a bound $O_{\prec}(\cal E_1(\cal P)N^{1/2})$. In this case, we use the fact that  $\partial_{z_s}^{\delta_2}\widehat{F}_{uv}$ is a lone factor of $\cal P^{(4)}_{2,1,2,q}$, i.e.\,$\nu_5(\cal P^{(4)}_{2,1,2,q})=1$. Let us write $\cal P^{(4)}_{2,1,2,q} =a^{(4)} t^\mu N^{-\theta}\{\cal G_1^{(4)}\}\cdots \{\cal G_\sigma^{(4)}\}$. Similar to \eqref{5.7}, the conditions $\nu_1(\cal P^{(4)}_{2,1,2,q})=\nu_1(\cal P)$, $\nu_3(\cal P^{(4)}_{2,1,2,q})=\nu_3(\cal P)$ ensure that
	$\nu_0(\cal P^{(4)}_{2,1,2,q})\leq\nu_0(\cal P)+2\leq 7$. Thus we can use the the argument in Case 1 to get an additional improvement of $N^{-1/2}$. Hence,  we have finished estimating all terms in $\bb E L_{2,1,2}^{(4)}$, and we conclude
	\[
	\{\bb E \cal S(L^{(4)}_{2,1,2})\}_\sigma=O_\prec(\cal E_1(\cal P))\,.
	\]
	Together with \eqref{5.18} and \eqref{5.19} we have
	\begin{equation} \label{5.21}
		\{	\bb E \cal S(L_{k,1}^{(4)})\}_\sigma=O_\prec(\cal E_1(\cal P))\,.
	\end{equation}
	A similar argument applies when we show that
	\begin{equation} \label{5.22}
		\{	\bb E \cal S(L_{k,2}^{(4)})\}_\sigma+\{	\bb E \cal S(L_{k,3}^{(4)})\}_\sigma=O_\prec(\cal E_1(\cal P))\,,
	\end{equation} 
	as in both cases we get a lone factor when the derivatives $\partial^k/\partial H^k_{ix}$ do not hit $\partial_z^{\delta_2}\widehat{F}_{uv}$. 
	
	However, we need to be more careful when we deal with $\bb EL_{k,4}^{(4)}$, since now $x \in \{u,v\}$, and the derivatives $\partial^k/\partial H^k_{ix}$ might destroy the lone factor $\partial_z^{\delta_2}\widehat{F}_{uv}$. We apply \eqref{diff1} and see that
	\[
	\bb E L_{k,4}^{(4)}=\sum_{q=1}^r \sum_{y} \bb E \widetilde{\cal P}_{{k,4,q}}^{(4)}
	\]
	for some fixed $r \in \bb N_+$, where each $\cal P^{(4)}_{k,4,q}\in \b P$ satisfies $\mu\big(\cal P^{(4)}_{k,4,q}\big)=\mu(\cal P)$, $\nu\big(\cal P^{(4)}_{k,4,q}\big)=\nu(\cal P)+1$, $\theta\big(\cal P^{(4)}_{k,4,q}\big)=\theta(\cal P)+(k+3)/2$, $\nu_1\big(\cal P^{(4)}_{k,4,q}\big)\geq \nu_1(\cal P)-k$, $\nu_3(\cal P^{(4)}_{k,4,q})\geq \nu_3(\cal P)$. Again, we split the sum over $q$ such that $\nu_1(\cal P^{(4)}_{k,4,q})=\nu_1(\cal P)-k$, $\nu_3(\cal P^{(4)}_{k,4,q})=\nu_3(\cal P)$ for $1\leq q\leq m\leq n$, and otherwise $q \geq m+1$. One can use Lemma \ref{lem4.5} to show that
	\[
	\sum_{q=m+1}^4\{\bb E \cal S\big(\widetilde{\cal P}_{2,4,q}^{(4)}\big)\}_\sigma=O_\prec(\cal E_1(\cal P))\,.
	\]
	For $q\leq m$, directly applying Lemma \ref{lem4.5} only gives a bound $O_{\prec}(\cal E_1(\cal P)N^{1/2})$. However, the condition $\nu_1(\cal P^{(4)}_{k,4,q})=\nu_1(\cal P)-k$ implies that all the $k$ derivatives $\partial^k/\partial H^k_{ix}$ are applied on factors with indices $i,x$. As $i \notin \{u,v\}$, we see that the derivatives $\partial^k/\partial H^k_{ix}$ cannot create a factor with indices $u,v$, i.e.\, $\partial_{z_s}^{\delta_2}\widehat{F}_{uv}$ is still a lone factor for all terms in $\cal P^{(4)}_{k,4,q}$. Thus $\nu_5(\cal P^{(4)}_{k,4,q})=1$. In addition, $\nu_1(\cal P^{(4)}_{k,4,q})=\nu_1(\cal P)-k$, $\nu_3(\cal P^{(4)}_{k,4,q})=\nu_3(\cal P)$ ensures $\nu_0(\cal P^{(4)}_{k,4,q})\leq \nu_0(\cal P)\leq 5$. Thus we can use the argument in Case 1 to get an additional factor of $N^{-1/2}$. In conclusion, we get 
	\[
	\{	\bb E \cal S(L_{k,1}^{(4)})\}_\sigma=O_\prec(\cal E_1(\cal P))\,.
	\]
	Combining the above with \eqref{5.21} and \eqref{5.22} we get
	\begin{equation} \label{5.23}
		\{	\bb E \cal S(L_{k}^{(4)})\}_\sigma=O_\prec(\cal E_1(\cal P))
	\end{equation}
	for all fixed $k \geq 2$. Similar steps can be used to show that
	\begin{equation} \label{5.24}
		\{	\bb E \cal S(L_{k}^{(3)})\}_\sigma=O_\prec(\cal E_1(\cal P))
	\end{equation}
	for all fixed $k \geq 2$. Inserting \eqref{5.15}, \eqref{5.23} and \eqref{5.24} into \eqref{5.12}, we get $\bb E \cal S(\cal P)=\{\bb E \cal S(\widetilde{\cal P})\}_\sigma=O_{\prec}(\cal E_1(\cal P))$ as desired.
	
	\textit{Case 3.} Suppose $\nu_5(
	\cal P)=\nu_6(\cal P)=1$, $\nu_0(\cal P)\leq 5$\,, and the lone factors of $\cal P$ are $(\partial^{\delta_1}_{z_r} \widehat{G}\partial_{z_t}^{\delta_3}\widehat{U})_{ij}\equiv (\partial^{\delta_1}_{z_r} \widehat{G}(z_r)\partial^{\delta_3}_{z_t} \widehat{G}(z_t))_{ij}$ and $\partial^{\delta_2}_{z_s} \widehat{F}_{uv}\equiv \partial^{\delta_2}_{z_s} \widehat{G}_{uv}(z_s)$, where $z_r,z_s,z_t \in \{z_1,...,z_\sigma\}$, $\delta_{1},\delta_2,\delta_3 \in \{0,1\}$. In this case, again we need to improve Lemma \ref{lem4.5} by a factor of $N^{-1}$.
	Using resolvent identity \eqref{resolvent eq for G} we have
	\begin{equation} \label{5.25}
		(\partial^{\delta_1}_{z_r} \widehat{G}\partial_{z_t}^{\delta_3}\widehat{U})_{ij}
		=\partial^{\delta_1}_{z_r}\partial_{z_t}^{\delta_3}(\widehat{G}\widehat{U})_{ij}
		=\partial^{\delta_1}_{z_s}\partial_{z_t}^{\delta_3}(\ul{\widehat{H}\widehat{G}}(\widehat{G}\widehat{U})_{uv})-\partial^{\delta_1}_{z_s}\partial_{z_t}^{\delta_3}((\widehat{H}\widehat{G}\widehat{U})_{ij}\ul{\widehat{G}})
		+\partial_{z_t}^{\delta_1}\ul{\widehat{G}}\partial^{\delta_3}_{z_s} \widehat{U}_{ij}\,.
	\end{equation}
	Let us write $ (\partial^{\delta_1}_{z_r} \widehat{G}\partial_{z_t}^{\delta_3}\widehat{U})_{ij}=\sum_w(\partial^{\delta_1}_{z_r} \widehat{G})_{iw}(\partial_{z_t}^{\delta_3}\widehat{U})_{wj}$, and note that $(\partial^{\delta_1}_{z_r} \widehat{G}\partial_{z_t}^{\delta_3}\widehat{U})_{ij} $ is not a factor of $\widetilde{\cal P}$, but a factor of $\sum\limits_{x}\widetilde{\cal P}$. We denote $\widetilde{\cal P}^{(C)}\deq (\sum\limits_{w}\widetilde{P})/(\partial^{\delta_1}_{z_r} \widehat{G}\partial_{z_t}^{\delta_3}U)_{ij}\partial^{\delta_2}_{z_s} \widehat{F}_{uv}$. By \eqref{5.25} we have
	\begin{equation} \label{5.26}
	\begin{aligned} 
		\sum_{w}\bb E \widetilde{P}=&\,\frac{1}{N}{\sum_{x,y}}^* \bb E H_{xy}\partial^{\delta_1}_{z_r}\partial_{z_t}^{\delta_3}(\widehat{G}_{xy}(\widehat{G}\widehat{U})_{ij})\partial^{\delta_2}_{z_s} \widehat{F}_{uv} \widetilde{\cal P}^{(C)}\\
		&-\frac{1}{N}\sum_{x,y:x \ne i} \bb E H_{ix}  \partial^{\delta_1}_{z_r}\partial_{z_t}^{\delta_3}((\widehat{G}\widehat{U})_{xj}\widehat{G}_{yy})\partial^{\delta_2}_{z_s} \widehat{F}_{uv} \widetilde{\cal P}^{(C)}\\
		&+\bb E \partial_{z_t}^{\delta_1}\ul{\widehat{G}}\partial^{\delta_3}_{z_s} \widehat{U}_{ij} \partial^{\delta_2}_{z_s} \widehat{F}_{uv} \widetilde{\cal P}^{(C)} \eqd \bb E A_1+\bb E A_2+\bb EA_3\,.
	\end{aligned}
	\end{equation}
	We see that the RHS on the first line of \eqref{5.26} is very similar to that of \eqref{5.12}, the main difference is that two matrices $\widehat{G}$ are replaced by $\widehat{G}\widehat{U}$ in \eqref{5.26}. By \eqref{diff1},  we see that after expanding $\bb E A_1$
	and $\bb EA_2$ using Lemma \ref{lem:cumulant_expansion}, there will be a cancellation between their second-cumulant terms, and other terms can be estimated in a similar fashion as described in Case 2. This allows one to show that
	\[
	\{\bb E \cal S(A_1)\}_\sigma+\{\bb E \cal S(A_2)\}_\sigma=O_{\prec}(\cal E_1(\cal P))\,.
	\]
	In addition, note that the term $\bb E A_3$ can be treated as in Case 2. Hence we have $\bb E \cal S(\cal P)=\{\bb E \cal S(\widetilde{\cal P})\}_\sigma=O_{\prec}(\cal E_1(\cal P))$ as desired. 
	
	\textit{Case 4.} Suppose $\nu_5(\cal P)=\nu_7(\cal P)=1$ and $\nu_0(\cal P)\leq 5$, and the happy trio are $\partial_{z_r}^{\delta_1}\widehat{G}_{ij}\equiv \partial_{z_r}^{\delta_1}\widehat{G}_{ij}(z_r)$, $\partial_{z_s}^{\delta_2}\widehat{F}_{iu}\equiv \partial_{z_s}^{\delta_2}\widehat{G}_{iu}(z_s)$, and $\partial_{z_r}^{\delta_3}\widehat{U}_{iv}\equiv \partial_{z_r}^{\delta_1}\widehat{G}_{iv}(z_t)$, where $z_r,z_s,z_t \in \{z_1,...,z_\sigma\}$, $\delta_{1},\delta_2,\delta_3 \in \{0,1\}$. In this case, again we need to improve Lemma \ref{lem4.5} by a factor of $N^{-1}$.
	
	We shall proceed in a very similar way as in Case 2: we perform one cumulant expansion, and the resulting terms either gain a factor $N^{-1}$ or $N^{-1/2}$. In the latter the terms contain one lone factor, and we can use our estimate in Case 1 to gain another factor $N^{-1/2}$. Let us denote $\widetilde{\cal P}^{(D)}\deq \widetilde{P}/ \partial_{z_r}^{\delta_1}\widehat{G}_{ij}\partial_{z_s}^{\delta_2}\widehat{F}_{iu}\partial_{z_r}^{\delta_3}\widehat{U}_{iv}$. By \eqref{resolvent eq for G}, Lemma \ref{lem:cumulant_expansion}, and a routine estimate of the remainder term, we have
	\begin{equation} \label{5.277}
		\begin{aligned}
			\bb E \widetilde{\cal P}&=\frac{1}{N}{\sum_{x,y}}^*\bb E H_{xy}\partial^{\delta_1}_{z_r}( \widehat{G}_{yx}\widehat{G}_{ij})\partial_{z_s}^{\delta_2}\widehat{F}_{iu}\partial_{z_r}^{\delta_3}\widehat{U}_{iv}\widetilde{\cal P}^{(D)}-\sum_{x:x \ne i}\bb E H_{ix}\partial_{z_r}^{\delta_1} (\widehat{G}_{xj}\ul{G})\partial_{z_s}^{\delta_2}\widehat{F}_{iu}\partial_{z_r}^{\delta_3}\widehat{U}_{iv} \widetilde{\cal P}^{(D)}\\
			&= \sum_{k=1}^{\ell}\frac{s_{k+1}}{k!}\frac{1}{N^{(k+3)/2}}{\sum_{x,y}}^* \bb E \frac{\partial^k (\partial^{\delta_1}_{z_r}( \widehat{G}_{yx}\widehat{G}_{ij})\partial_{z_s}^{\delta_2}\widehat{F}_{iu}\partial_{z_r}^{\delta_3}\widehat{U}_{iv}\widetilde{\cal P}^{(D)})}{\partial H_{xy}^k}\\
			&\quad -\sum_{k=1}^{\ell}\frac{s_{k+1}}{k!}\frac{1}{N^{(k+3)/2}}\sum_{x,y:x \ne i} \bb E \frac{\partial^k (\partial_{z_r}^{\delta_1} (\widehat{G}_{xj}\widehat{G}_{yy})\partial_{z_s}^{\delta_2}\widehat{F}_{iu}\partial_{z_r}^{\delta_3}\widehat{U}_{iv} \widetilde{\cal P}^{(D)})}{\partial H_{ix}^k}+O_{\prec}(\cal E_1(\cal P)N^{-\nu(\cal P)})\\
			&=:\sum_{k=1}^{\ell}\bb E L_{k}^{(5)}+\sum_{k=1}^{\ell} \bb E L_{k}^{(6)} +O_{\prec}(\cal E_1(\cal P)N^{-\nu(\cal P)})
		\end{aligned}
	\end{equation}
	for some fixed $\ell \in \bb N_+$. Similar to \eqref{L_1} and \eqref{5.15} we have a cancellation between $\bb EL_{1}^{(5)}$ and $\bb EL_{1}^{(6)}$. We can show that
	\begin{equation*}
		\begin {aligned}
		\bb EL_{1}^{(5)}+\bb EL_{1}^{(6)}=& \frac{1}{N^2}\sum_{x,y} \bb E \partial_{z_r}^{\delta_1}(\widehat{G}_{yy}\widehat{G}_{xj})\partial_{z_s}^{\delta_2}(\widehat{F}_{xu}\widehat{F}_{ii}) \partial_{z_t}^{\delta_3}\widehat{U}_{iv} \widetilde{\cal P}^{(D)}\\
		&+\frac{1}{N^2}\sum_{x,y} \bb E \partial_{z_r}^{\delta_1}(\widehat{G}_{yy}\widehat{G}_{xj})\partial_{z_t}^{\delta_3}(\widehat{U}_{xv}\widehat{U}_{ii}) \partial_{z_s}^{\delta_2}\widehat{F}_{iu} \widetilde{\cal P}^{(D)}+\cal E^{(1)}\\
		\eqd& \sum_{q=1}^n \sum_{x,y}\bb E \widetilde {\cal P}^{(5)}_{1,q}+\cal E^{(1)}
	\end{aligned}
\end{equation*}
for some fixed $n \in \bb N_+$, where it can be checked using Lemma \ref{lem4.5} that
\[
\{\cal S(\cal E^{(1)})\}_\sigma=O_{\prec}(\cal E_0(\cal P)N^{-1})=O_{\prec}(\cal E_1(\cal P))\,.
\]
In addition, each $\widetilde {\cal P}^{(5)}_{1,q}$ satisfies $\mu\big({\cal P}^{(5)}_{1,q}\big)=\mu(\cal P)$, $\nu\big({\cal P}^{(5)}_{1,q}\big)=\nu(\cal P)+2$, $\theta\big({\cal P}^{(5)}_{1,q}\big)=\theta(\cal P)+2$, $\nu_1\big({\cal P}^{(5)}_{1,q}\big)= \nu_1(\cal P)-1$, $\nu_3({\cal P}^{(5)}_{1,q})= \nu_3(\cal P)+1$. Moreover, we see that $\nu_0({\cal P}^{(5)}_{1,q})=\nu_0(\cal P)+1\leq 6$ and $\nu_5({\cal P}^{(5)}_{1,q})=1$. Thus we can apply the result in Case 1 and show that
\[
\sum_{q=1}^n \{\bb E \cal S(\widetilde {\cal P}^{(5)}_{1,q})\}_\sigma= O_{\prec}\big(t^{\mu(\cal P)} N^{\nu(\cal P)+2-(\theta(\cal P)+2)-(\nu_1(\cal P)-1)/2-(\nu_3(\cal P)+1)-1/2}\big)=O_{\prec}(\cal E_1(\cal P))\,.
\]
As a result,
\begin{equation}\label{5.29}
	\{	\bb E\cal S(L_{1}^{(5)})+\bb E\cal S(L_{1}^{(6)})\}_\sigma=O_{\prec}(\cal E_1(\cal P))
\end{equation}
Similarly, we can use \eqref{diff1} and Lemma \ref{lem4.5} to show that for fixed $\ell\geq 2$,
\[
\sum_{k=2}^{\ell}\{\bb E\cal S(L_{k}^{(5)})\}_\sigma=O_{\prec}(\cal E_1(\cal P))\,,
\]
and
\begin{align*}
\sum_{k=2}^{\ell}\bb E L_k^{(6)}&=-\frac{s_3}{N^{5/2}}\sum_{x:x \notin \cal I_1(\cal P)}\sum_{y} \bb E \partial_{z_r}^{\delta_1} (\widehat{G}_{xj}\widehat{G}_{yy})\partial_{z_s}^{\delta_2}(\widehat{F}_{xu}\widehat{F}_{ii})\partial_{z_r}^{\delta_3}(\widehat{U}_{xv}\widehat{U}_{ii})\widetilde{\cal P}^{(D)}+\cal E^{(2)}\\
&\eqd \sum_{q=1}^n \sum_{x,y}\bb E \widetilde {\cal P}^{(6)}_{2,q}+\cal E^{(2)}
\end{align*}
for some fixed $n \in \bb N_+$, where $\{\cal S(\cal E^{(2)})\}_\sigma=O_{\prec}(\cal E_1(\cal P))$. In addition, each $\widetilde {\cal P}^{(6)}_{2,q}$ satisfies $\mu\big({\cal P}^{(6)}_{2,q}\big)=\mu(\cal P)$, $\nu\big({\cal P}^{(6)}_{2,q}\big)=\nu(\cal P)+2$, $\theta\big({\cal P}^{(6)}_{2,q}\big)=\theta(\cal P)+5/2$, $\nu_1\big({\cal P}^{(6)}_{2,q}\big)= \nu_1(\cal P)$, $\nu_3({\cal P}^{(6)}_{2,q})= \nu_3(\cal P)$. Moreover, we see that $\nu_0({\cal P}^{(6)}_{k,q})=\nu_0(\cal P)\leq 5$ and $\nu_5({\cal P}^{(6)}_{k,q})=1$. Thus we can apply the result in Case 1 and show that
\[
\sum_{q=1}^n \{\bb E \cal S(\widetilde {\cal P}^{(6)}_{k,q})\}_\sigma= O_{\prec}\big(t^{\mu(\cal P)} N^{\nu(\cal P)+2-(\theta(\cal P)+5/2)-\nu_1(\cal P)/2-\nu_3(\cal P)-1/2}\big)=O_{\prec}(\cal E_1(\cal P))\,.
\]
As a result, we have
\begin{equation} \label{5.31}
	\sum_{k=2}^\ell	\{	\bb E\cal S(L_{k}^{(5)})+\bb E\cal S(L_{k}^{(6)})\}_\sigma=O_{\prec}(\cal E_1(\cal P))\,.
\end{equation}
Inserting \eqref{5.29} and \eqref{5.31} into \eqref{5.277}, we have $\bb E \cal S(\cal P)=\{\bb E \cal S(\widetilde{\cal P})\}_\sigma=O_{\prec}(\cal E_1(\cal P))$ as desired. This finishes the proof.
\end{proof}

Now we are ready to prove Lemma \ref{lem5.1}. Let $\cal P =  a t^\mu N^{-\theta}\{\cal G_1\}\cdots \{\cal G_\sigma\}\in \b P$ as in Lemma \ref{lem5.1}, where $\cal G_r \in \b G(z_r)$ for $r=1,...,\sigma$. It suffices to assume $a=1$ and $\mu=\theta=0$.

\subsection{Proof of Lemma \ref{lem5.1} (i)} \label{sec5.2}
In this section we prove Lemma \ref{lem5.1} for $d(\cal P)=0$. We see that $\cal E_0(\cal P)=\cal E_*(\cal P)N$, i.e.\,we need to gain an improvement of factor $N^{-1}$.  Let us recall the assumption from Lemma \ref{lem5.7} that $\nu(\cal P)=2$, $\nu_2(\cal P)=\nu_3(\cal P)=0$, and the assumption  $\mu=\theta=0$ stated above. In case $\sigma=1$, some toy examples are $\cal P=\{\mathcal{G}_1\}$ with $\cal G_1=\widehat{G}_{ii}^2\widehat{G}^2_{jj}$,  $\widehat{G}_{ij}\widehat{G}_{ii}\widehat{G}_{jj}$ or $\widehat{G}_{ij}^2\widehat{G}_{ii}\widehat{G}_{jj}$ (say), which correspond to the cases $\nu_1(\cal P)=0$, $1$ or $2$, respectively. In the sequel, we separate the discussion for general $\mathcal{P}$ into the cases $\nu_1(\cal P)=0$, $\nu_1(\cal P)=1$ and $\nu_1(\cal P)\geq 2$. 

\textit{Case 1.} Let us first consider the case $\nu_1(\cal P)=0$, i.e. there is no off-diagonal entries in $\cal P$. In this case $\cal E_*(\cal P)=t^{\mu(\cal P)}N^{2-\theta(\cal P)-1}=N$. By Lemma \ref{prop4.4}, we see that
\[
\langle\cal S(\widetilde{\cal P})\rangle \prec N^{2-1}\frac{1}{|\eta_1|^{\nu_0(\cal G_1)/2+1}}\cdots \frac{1}{|\eta_\sigma|^{\nu_0(\cal G_\sigma)/2+1}} \leq  N\frac{1}{|\eta_1|^{5/2}}\cdots \frac{1}{|\eta_\sigma|^{5/2}} \,,
\]
where in the last step we used $\nu_0(\cal G_r)\leq \nu_0(\cal P)\leq 3$ for all $r=1,...,\sigma$. Then Lemma \ref{lem:integration} implies
\[
\cal S (\cal{\mathring  P}) = \{\langle\cal S(\widetilde{P})\rangle \}_\sigma =O_{\prec}(\cal E_*(\cal P))\,,
\]
and together with Lemma \ref{prop_prec} we complete the proof.

\textit{Case 2.} Now we assume $\nu_1(\cal P)=1$, i.e.\,there is exactly one off-diagonal factor in $\cal P$. Then $\cal E_*(\cal P)=N^{1/2} $. W.O.L.G.\,we denote it by $\partial_{z_1}^{\delta}\widehat{G}_{ij}\equiv \partial_{z_1}^{\delta}\widehat{G}_{ij}(z_1)$, $\delta \in \{0,1\}$. Let $\widetilde{\cal P}^{(E)}\deq \widetilde{\cal P}/ \partial_{z_1}^{\delta}\widehat{G}_{ij}$, and note that $\widetilde{\cal P}^{(E)}$ only contains diagonal entries. By the definition of $\b D$ in \eqref{ddd}, $\cal P$ is real, and $\bb E \cal |S(\cal P)|^2=\bb E \cal S(\cal P)^2=\bb E\{\cal S(\widetilde{P})\}_\sigma^2$. Let us write $\widetilde{\cal P}'=\{\cal G_{1'}\}\cdots \{\cal G_{\sigma'}\}\in \b P$, where each $\cal G_{r'}$ is obtained from $\cal G_r$ by changing $z_r$ into $z_{r'}$. We have
\begin{equation*} 
\bb E \cal |S(\cal P)|^2=\bb E\{\cal S(\widetilde{P})\}_\sigma^2=\{\bb E \cal S(\widetilde{\cal P})S(\widetilde{\cal P}')\}_{2\sigma}\eqd\Big\{{\sum_{i,j}}^{*}{\sum_{u,v}}^{*} \bb E\widetilde{\cal P}_{ij} \widetilde{\cal P}'_{uv}\Big\}_{2\sigma}\,.
\end{equation*}
Similar to \eqref{5.2}, we have for $i \ne j$ and $u \ne v$ that
\begin{equation}  \label{5.27}
\begin{aligned}
	\bb E \widetilde{\cal P}_{ij} \widetilde{\cal P}'_{uv}&=\frac{1}{N}{\sum_{x,y}}^*\bb E H_{xy}\partial^{\delta}_{z_1}( \widehat{G}_{yx}\widehat{G}_{ij})\widetilde{\cal P}^{(E)}_{ij}\widetilde{\cal P}'_{uv}-\sum_{x:x \ne i}\bb E H_{ix}\partial_{z_1}^{\delta} (\widehat{G}_{xj}\ul{G}) \widetilde{\cal P}^{(E)}_{ij}\widetilde{\cal P}'_{uv}\\
	&= \sum_{k=1}^{\ell}\frac{s_{k+1}}{k!}\frac{1}{N^{(k+3)/2}}{\sum_{x,y}}^* \bb E \frac{\partial^k (\partial_{z_1}^{\delta}( \widehat{G}_{yx}\widehat{G}_{ij})\widetilde{\cal P}^{(E)}_{ij}\widetilde{\cal P}'_{uv})}{\partial H_{xy}^k}\\
	&\quad -\sum_{k=1}^{\ell}\frac{s_{k+1}}{k!}\frac{1}{N^{(k+3)/2}}\sum_{x,y:x \ne i} \bb E \frac{\partial^k (\partial_{z_1}^{\delta}( \widehat{G}_{xj}\widehat{G}_{yy}) \widetilde{\cal P}^{(E)}_{ij}\widetilde{\cal P}'_{uv})}{\partial H_{ix}^k}+O_{\prec}(\cal E_*(\cal P)^2N^{-4})\\
	&=\sum_{k=1}^{\ell}\sum_{n=1}^{k}{k \choose n}\frac{s_{k+1}}{k!}\frac{1}{N^{(k+3)/2}}{\sum_{x,y}}^* \bb E \frac{\partial^{k-n} (\partial_{z_1}^{\delta}( \widehat{G}_{yx}\widehat{G}_{ij})\widetilde{\cal P}^{(E)}_{ij})}{\partial H_{xy}^{k-n}} \frac{\partial^n \widetilde{\cal P}'_{uv}}{\partial H_{xy}^n}\\
	&\quad+\sum_{k=1}^{\ell}\frac{s_{k+1}}{k!}\frac{1}{N^{(k+3)/2}}{\sum_{x,y}}^* \bb E \frac{\partial^k (\partial_{z_1}^{\delta}( \widehat{G}_{yx}\widehat{G}_{ij})\widetilde{\cal P}^{(E)}_{ij})}{\partial H_{xy}^k} \widetilde{\cal P}'_{uv}\\
	&\quad -\sum_{k=1}^{\ell}\sum_{n=1}^{k}{k \choose n}\frac{s_{k+1}}{k!}\frac{1}{N^{(k+3)/2}}\sum_{x,y:x \ne i} \bb E \frac{\partial^{k-n} (\partial_{z_1}^{\delta}( \widehat{G}_{xj}\widehat{G}_{yy}) \widetilde{\cal P}^{(E)}_{ij})}{\partial H_{ix}^{k-n}}\frac{\partial^n \widetilde{\cal P}'_{uv}}{\partial H_{ix}^n}\\
	&\quad- \sum_{k=1}^{\ell}\frac{s_{k+1}}{k!}\frac{1}{N^{(k+3)/2}}\sum_{x,y:x \ne i} \bb E \frac{\partial^k (\partial_{z_1}^{\delta}( \widehat{G}_{xj}\widehat{G}_{yy}) \widetilde{\cal P}^{(E)}_{ij})}{\partial H_{ix}^k}\widetilde{\cal P}'_{uv}+O_{\prec}(\cal E_*(\cal P)^2N^{-4})\\
	&\eqd \sum_{k=1}^{\ell}\sum_{n=1}^k \bb EL_{k,n}^{(5)}+\sum_{k=1}^{\ell} \bb EL_{k}^{(6)}+\sum_{k=1}^{\ell}\sum_{n=1}^k \bb EL_{k,n}^{(7)}+\sum_{k=1}^{\ell} \bb EL_{k}^{(8)}+O_{\prec}(\cal E_*(\cal P)^2N^{-4})\,.
\end{aligned}
\end{equation}
We emphasis that all the above $L$-terms is dependent of $i,j,u,v$, and for convenience we omit the indices. Note that in the above, $\bb EL^{(5)}_{k,n}$ and $\bb EL_{k,n}^{(7)}$ contains the information for the covariance of $\widetilde{\cal P}_{ij}$ and $ \widetilde{\cal P}'_{uv}$ terms, and in their treatment we shall use the estimates  proved in Lemma \ref{lem5.4}.  On the other hand,  in $\bb EL^{(6)}_k$ and $\bb EL_{k}^{(8)}$, we see that $ \widetilde{\cal P}'_{uv}$ is intact and thus one can focus on the estimate of the derivative of $\widetilde{\cal P}_{ij}$,  which will be similar to the counterpart in the estimate of $\bb E \widetilde{\cal P}$ in \eqref{5.12}. 

We first deal with $\bb E L_{1,1}^{(7)}$. Note that
\begin{align}
{\sum_{i,j}}^*{\sum_{u,v}}^*
&={\sum_{i,j,u,v}}^*+{\sum_{i,j,u}}^* (\b 1(i=v)+\b 1(j=v))+{\sum_{i,j,v}}^*(\b 1(i=u)\notag\\
&+\b 1(j=u))+{\sum_{i,j}}^* \b 1(\{i,j\}=\{u,v\})\nonumber\\
&\eqd {\sum}_{1}+\cdots+{\sum}_{4} \label{5.28}\,,
\end{align}
and $\bb E L_{1,1}^{(7)}$ behaves differently when we apply different ${\sum}_w$, $w=1,...,4$. We have
\begin{align}
{\sum_{i,j}}^{*}{\sum_{u,v}}^{*}\bb EL_{1,1}^{(7)}&=-\frac{1}{N^2}{\sum_{i,j}}^{*}{\sum_{u,v}}^{*}\sum_{x,y}\bb E \partial_{z_1}^{\delta}( \widehat{G}_{xj}\widehat{G}_{yy}) \widetilde{\cal P}^{(E)}_{ij}\frac{\partial \widetilde{\cal P}'_{uv}}{\partial H_{ix}}(1+\delta_{ix})\notag\\
&+\frac{2}{N}{\sum_{i,j}}^{*}{\sum_{u,v}}^{*}\sum_{y}\bb E \partial_{z_1}^{\delta}( \widehat{G}_{ij}\widehat{G}_{yy}) \widetilde{\cal P}^{(E)}_{ij}\frac{\partial \widetilde{\cal P}'_{uv}}{\partial H_{ii}}\nonumber \\
&\eqd \sum_{q=1}^{n_1}{\sum_{i,j}}^{*}{\sum_{u,v}}^{*}\sum_{x,y} \bb E \widetilde{\cal P}^{(7)}_{1,1,1,q}+\sum_{q=1}^{n_2}{\sum_{i,j}}^{*}{\sum_{u,v}}^{*}\sum_{y} \bb E \widetilde{\cal P}^{(7)}_{1,1,2,q}\nonumber \\
&\eqd \sum_{q=1}^{n_1}\sum_{w=1}^4{\sum}_{w}\sum_{x,y} \bb E \widetilde{\cal P}^{(7)}_{1,1,1,q,w}+\sum_{q=1}^{n_2}\sum_{w=1}^4{\sum}_w\sum_{y} \bb E \widetilde{\cal P}^{(7)}_{1,1,2,q,w} \label{5.30}
\end{align} 
for some fixed $n_1,n_2\in \bb N$, where in the last step we used \eqref{5.28}. By \eqref{diff1}, we see that for each $q$, we have $\cal P^{(7)}_{1,1,1,q,1}\in \b P$, with $\nu(\cal P^{(7)}_{1,1,1,q,1})=6$, $\theta(\cal P^{(7)}_{1,1,1,q,1})=2$, $\nu_3(\cal P^{(7)}_{1,1,1,q,1})=1$. Moreover, we see that when an off-diagonal factor in $\widetilde{P}'_{uv}$ is differentiated, each $\cal S_2(\cal P^{(7)}_{1,1,1,q,1})$ contains two lone factors, which are either in the form $\partial_{z_1}^{\delta_1}\partial_{z_r}^{\delta_2}(\widehat{G}\widehat{F})_{ju}$ and $\partial_{z_r}^{\delta_3}\widehat{F}_{iv}$, or $\partial_{z_1}^{\delta_1}\partial_{z_r}^{\delta_2}(\widehat{G}\widehat{F})_{jv}$ and $\partial_{z_r}^{\delta_3}\widehat{F}_{iu}$, where $\delta_1,...,\delta_3\in \{0,1\}$, and $\widehat{F}\equiv \widehat{G}(z_{r'})$ for some $r'\in \{1,2,...,\sigma\}$. In this case $\nu_1(\cal P^{(7)}_{1,1,1,q,1})\geq 2$, $\nu_5(\cal P^{(7)}_{1,1,1,q,1})=\nu_6(\cal P^{(7)}_{1,1,1,q,1})=1$, and $\nu_0(\cal P^{(7)}_{1,1,1,q,1})\leq 5= 7-2(\nu_6(\cal P^{(7)}_{1,1,1,q,1}))$. When a diagonal factor in $\widetilde{P}_{uv}'$, say $\widehat{F}_{uu}$, was differentiated,  each $\cal S_2(\cal P^{(7)}_{1,1,1,q,1})$ contains only one lone factor $\widehat{F}_{iu}$ or $\widehat{F}_{xu}$. In this case $\nu_1(\cal P^{(7)}_{1,1,1,q,1})\geq 3$, $\nu_5(\cal P^{(7)}_{1,1,1,q,1})=1$, $\nu_6(\cal P^{(7)}_{1,1,1,q,1})=0$, and $\nu_0(\cal P^{(7)}_{1,1,1,q,1})\leq 6\leq 7-2(\nu_6(\cal P^{(7)}_{1,1,1,q,1}))$. We can then apply Lemma \ref{lem5.4} and show that
\begin{align*}
\{\bb E\cal S(\widetilde{\cal P}^{(7)}_{1,1,1,q,1})\}_{2\sigma}&=\Big\{{\sum}_1\sum_{x,y} \bb E \widetilde{\cal P}^{(7)}_{1,1,1,q,1} \Big \}_{2\sigma}\prec N^{6-2-2/2-1-1/2-1/2}+N^{6-2-3/2-1-1/2-0/2}\\
&=\cal E_*(\cal P)^2
\end{align*}
as desired. The estimate concerning other terms on RHS of \eqref{5.30} are easier, as for them $\nu\leq 5$, i.e.\,there are fewer summations. We can use Lemma \ref{lem5.4} and show that
\begin{equation} \label{5.34}
\Big\{	{\sum_{i,j}}^{*}{\sum_{u,v}}^{*}\bb EL_{1,1}^{(7)}\Big\}_{2\sigma} \prec \cal E_*(\cal P)^2\,.
\end{equation}

Now, we consider $\bb EL_{k,n}^{(7)}$ for $k \geq 2$. A generalization of \eqref{5.28} shows that
\begin{align*}
	{\sum_{i,j}}^*{\sum_{u,v}}^*\sum_{x:x \ne i}&={\sum_{i,j,u,v,x}}^*+{\sum_{i,j,u,v}}^*(\b 1(x=j)+\b 1(x=u)+\b 1(x=v))\\
	&\quad +{\sum_{i,j,u}}^* (\b 1(i=v)+\b 1(j=v))(\b 1(x=j)+\b 1(x=u)+\b 1(x=v))\\
	&\quad +{\sum_{i,j,v}}^*(\b 1(i=u)+\b 1(j=u))(\b 1(x=j)+\b 1(x=u)+\b 1(x=v))\\
	&\quad +{\sum_{i,j}}^* \b 1(\{i,j\}=\{u,v\})(\b 1(x=j)+\b 1(x=u)+\b 1(x=v))\eqd \sum_{w=1}^5{\sum}^{(1)}_w\,.
\end{align*}
Similar to \eqref{5.30}, we have
\begin{equation*}
\begin{aligned}
	{\sum_{i,j}}^*{\sum_{u,v}}^*\bb EL_{k,n}^{(7)}&=-{k \choose n}\frac{s_{k+1}}{k!}\frac{1}{N^{(k+3)/2}}\sum_{w=1}^5{\sum}^{(1)}_w\sum_{y} \bb E \frac{\partial^{k-n} (\partial_{z_1}^{\delta}( \widehat{G}_{xj}\widehat{G}_{yy}) \widetilde{\cal P}^{(E)}_{ij})}{\partial H_{ix}^{k-n}}\frac{\partial^n \widetilde{\cal P}'_{uv}}{\partial H_{ix}^n}\\
	&\eqd  \sum_{q=1}^r\sum_{w=1}^5{\sum}^{(1)}_w\sum_{y} \bb E \widetilde{\cal P}^{(7)}_{k,n,q,w}
\end{aligned}
\end{equation*} 
for some fixed $r\in \bb N$. 

We first check the case $w=1$. By \eqref{diff1}, we see that for each $q$, we have $\cal P^{(7)}_{k,n,q,1}\in \b P$, with $\nu(\cal P^{(7)}_{k,n,q,1})=6$, $\theta(\cal P^{(7)}_{k,n,q,1})=(k+3)/2$, $\nu_1(\cal P^{(7)}_{k,n,q,1})\geq 3$, $\nu_3(\cal P^{(7)}_{k,n,q,1})\in \{0,1\} $, $\nu_0(\cal P^{(7)}_{k,n,q,1})\leq 2\nc+\nu_1(\cal P^{(7)}_{k,n,q,1})$, and it is not hard to show that
\[
\nu_*(\cal P^{(7)}_{k,n,q,1}) \leq 2(2\nu_3(\cal P^{(7)}_{k,n,q,1})+\nu_0(\cal P^{(7)}_{k,n,q,1})-10)_+\leq  2(\nu_1(\cal P^{(7)}_{k,n,q,1})-6)_+\,.
\]
When $k \geq 3$ and $\nu_1(\cal P^{(7)}_{k,n,q,1})\geq 4$, it is easy to see from Lemma \ref{lem4.5} that $\{\bb E \cal S(\widetilde{\cal P}^{(7)}_{k,n,q,1})\}_{2\sigma}=O_{\prec}(\cal E_*(\cal P)^2)$. When $k \geq 3$ and $\nu_1(\cal P^{(7)}_{k,n,q,1})=3$, we have $\nu_5(\cal P^{(7)}_{k,n,q,1})=1$, as $\partial^n \widetilde{P}'_{uv}/\partial H_{ix}^n$ contains at least one lone factor. In addition, $\nu_0(\cal P^{(7)}_{k,n,q,1})\leq 2+3=5$, thus we can use Lemma \ref{lem5.4} and show that $\{\bb E \cal S(\widetilde{\cal P}^{(7)}_{k,n,q,1})\}_{2\sigma}=O_{\prec}(\cal E_*(\cal P)^2)$. When $k=2$ and $\nu_0({\cal P}^{(7)}_{k,n,q,1})\geq 5$, we can use Lemma \ref{lem4.5} to show that $\{\bb E \cal S(\widetilde{\cal P}^{(7)}_{k,n,q,1})\}_{2\sigma}=O_{\prec}(\cal E_*(\cal P)^2)$. When $k=2$ and $\nu_1({\cal P}^{(7)}_{k,n,q,1})=4$, we have $\nu_0(\cal P^{(7)}_{k,n,q,1})\leq 6$ and $\nu_5({\cal P}^{(7)}_{k,n,q,1})=1$, thus we can use Lemma \ref{lem5.4} to show that $\{\bb E \cal S(\widetilde{\cal P}^{(7)}_{k,n,q,1})\}_{2\sigma}=O_{\prec}(\cal E_*(\cal P)^2)$. When $k=2$ and $\nu_1({\cal P}^{(7)}_{k,n,q,1})=3$, we have $\nu_0(\cal P^{(7)}_{k,n,q,1})\leq 5$. In addition, we either have two lone factors or a happy trio, where the happy trio is of the form $\widehat{G}_{xj}$, $\widehat{F}_{xu}$, $\widehat{F}_{xv}$. This means we either have $\nu_5({\cal P}^{(7)}_{k,n,q,1})=\nu_6({\cal P}^{(7)}_{k,n,q,1})=1$, or $\nu_5({\cal P}^{(7)}_{k,n,q,1})=\nu_7({\cal P}^{(7)}_{k,n,q,1})=1$. Thus we can use Lemma \ref{lem5.4} to show that $\{\bb E \cal S(\widetilde{\cal P}^{(7)}_{k,n,q,1})\}_{2\sigma}=O_{\prec}(\cal E_*(\cal P)^2)$. This finishes the estimate for $w=1$. For $w=2,...,5$, the estimate are easier, as there are fewer summations. We omit the details. Hence,  we have showed that
\[
\Big\{{\sum_{i,j}}^*{\sum_{u,v}}^*\bb EL_{k,n}^{(7)}\Big\}_{2\sigma} \prec \cal E_*(\cal P)^2
\]
for $k\geq 2$. Together with \eqref{5.34} we have
\begin{equation} \label{5.37}
\sum_{k=1}^\ell\sum_{n=1}^k\Big\{{\sum_{i,j}}^*{\sum_{u,v}}^*\bb EL_{k,n}^{(7)}\Big\}_{2\sigma} \prec \cal E_*(\cal P)^2\,.
\end{equation}
Similarly, we can also show that
\begin{equation} \label{5.38}
\sum_{k=1}^\ell\sum_{n=1}^k\Big\{{\sum_{i,j}}^*{\sum_{u,v}}^*\bb EL_{k,n}^{(5)}\Big\}_{2\sigma} \prec \cal E_*(\cal P)^2\,.
\end{equation}

Now let us estimate $L_k^{(6)}$ and $L_k^{(8)}$. We see that they are essentially about the expectation of $\widetilde{\cal P}$, as the factor $\widetilde{\cal P}'$ is not differentiated. Similar to \eqref{L_1} and \eqref{5.15}, we also have a cancellation between the leading terms of $\bb E L_k^{(6)}$ and $\bb EL_k^{(8)}$. In addition, note that since $\nu_1(\cal P)=1$, $\widetilde{\cal P}^{(E)}$ only contains diagonal entries. Together with Lemma \ref{lem4.5}, we can show that
\begin{equation*}
\bb E L_1^{(6)}+\bb EL_1^{(8)}=\bb E \cal E^{(3)}_{ij} \widetilde{P}'_{uv}\,,\quad \sum_{k=2}^{\ell}\bb EL_{k}^{(6)}=\bb E  \cal E^{(4)}_{ij} \widetilde{P}'_{uv}\,,\quad  \sum_{k=3}^{\ell}\bb EL_{k}^{(8)}=\bb E  \cal E^{(5)}_{ij} \widetilde{P}'_{uv}
\end{equation*}
and 
\begin{equation*}
\begin{aligned}
	\bb EL_{2}^{(8)}&=-\frac{s_3}{N^{5/2}}\sum_{x,y}\bb E \partial^\delta_{z_1}(\widehat{G}_{xx}\widehat{G}_{ii}\widehat{G}_{xj}\widehat{G}_{yy})\widetilde{\cal P}^{(E)}_{ij}\widetilde{\cal P}'_{uv}\\
	&\quad -\frac{s_3}{2N^{5/2}}\sum_{x,y} \mathbb{E}\partial^{\delta}_{z_1}(\widehat{G}_{xj}\widehat{G}_{yy})\widetilde{\cal P}'_{ij}\frac{\partial^2 \widetilde{\cal P}^{(E)}_{ij}}{\partial H_{xi}^2}+\bb E  \cal E^{(6)}_{ij} \widetilde{P}'_{uv}\\
	&\eqd \bb E\cal E_{ij}^{(7,1)}\widetilde{\cal P}'_{ij}
	+\bb E\cal E_{ij}^{(7,2)}\widetilde{\cal P}'_{ij}+	\bb E  \cal E^{(6)}_{ij} \widetilde{P}'_{uv}	\,.
\end{aligned}
\end{equation*}
where 
$
\{\cal S(\cal E^{(n)})\}_\sigma \prec  \cal E_*(\cal P)
$
for $n=3,...,6$. Directly estimating $\cal E_{ij}^{(7,1)}$ by Lemma \ref{lem:intermediate} would not be enough for us. By the isotropic semicircle law Theorem \ref{refthm1}, we have
\begin{equation} \label{swd}
\sum_{x} \widehat{G}_{xx}\widehat{G}_{xj}=\sum_{x} m(z_1)\widehat{G}_{xj}+\sum_{x}(\widehat{G}_{xx}-m(z_1))\widehat{G}_{xj} \prec \frac{1}{|\eta_1|}\,.
\end{equation}
Thus
\[
\cal E_{ij}^{(7,1)}=\frac{1}{N^{5/2}}\sum_{x,y} \partial^\delta_{z_1}(\widehat{G}_{xx}\widehat{G}_{ii}\widehat{G}_{xj}\widehat{G}_{yy})\widetilde{\cal P}^{(E)}_{ij} \prec  N^{1-5/2}  \frac{1}{|\eta_1|^{2}}\cdot \frac{1}{|\eta_{2}|^{3/2}}\cdots \frac{1}{|\eta_{\sigma}|^{3/2}}\,,
\]
and together with Lemma \ref{lem:integration} we get $\{\cal S(\cal E^{(7,1)})\}_\sigma \prec  \cal E_*(\cal P)$.  The steps for $\cal E_{ij}^{(7,2)}$ are similar, since $\widetilde{P}^{(E)}$ only contains diagonal entries, we see that the worst terms in $\cal E_{ij}^{(7,2)}$  will contain factors of the form $\widehat{G}_{xj}(z_1)\widehat{G}_{xx}(z_r)$ for some $r \in \{1,...,\sigma\}$. We can use an estimate similar  to \eqref{swd} and show that $\{\cal S(\cal E^{(7,2)})\}_\sigma \prec  \cal E_*(\cal P)$. Thus we have proved that
\begin{equation} \label{5.40}
\sum_{k=2}^\ell\Big\{{\sum_{i,j}}^*{\sum_{u,v}}^*(\bb E L_k^{(6)}+\bb EL_k^{(8)})\Big\}_{2\sigma} =O_{\prec}(\cal E_*(\cal P)) \cdot \bb E |\cal S(\cal P)| \leq  O_{\prec}(\cal E_*(\cal P)) \cdot (\bb E |\cal S(\cal P)|^2)^{1/2}\,.
\end{equation}
Inserting \eqref{5.37}, \eqref{5.38} and \eqref{5.40} into \eqref{5.27}, we obtain
\[
\bb E \cal |S(\cal P)|^2 =O_{\prec}(\cal E_*(\cal P)^2)+O_{\prec}(\cal E_*(\cal P)) \cdot (\bb E |\cal S(\cal P)|^2)^{1/2}\,,
\]
which implies $\bb E \cal |S(\cal P)|^2=O_{\prec}(\cal E_*(\cal P)^2)$ as desired.

\textit{Case 3}. We consider the case $\nu_1(\cal P )\geq 2$, and as a result $\cal E_*(\cal P)=N^{1-\nu_1(\cal P)/2}$.  The discussion is very similar to Case 2. Let $\partial^{\delta_1}_{z_1}\widehat{G}^{(1)}_{ij}$,...,$\partial^{\delta_{\nu_1}}_{z_{\nu_1}}\widehat{G}^{({\nu_1})}_{ij}$ be the off-diagonal factors in $\cal P$, where $\delta_1,...,\delta_{\nu_1} \in \{0,1\}$. Let $\widetilde{\cal P}^{(F)} \deq\widetilde{\cal P}/ (\partial^{\delta_1}_{z_1}\widehat{G}^{(1)}_{ij}\cdots \partial^{\delta_{\nu_1}}_{z_{\nu_1}}\widehat{G}^{({\nu_1})}_{ij} )$. Similar to \eqref{5.27}, we have for $i\ne j$ and $u \ne v$ that    
\begin{align}
&\quad \bb E \langle\widetilde{\cal P}_{ij} \rangle \langle \widetilde{\cal P}'_{uv} \rangle\notag \\
&=\sum_{k=1}^{\ell}\sum_{n=1}^{k}{k \choose n}\frac{s_{k+1}}{k!}\frac{1}{N^{(k+3)/2}}{\sum_{x,y}}^* \bb E \frac{\partial^{k-n}(\partial^{\delta_1}_{z_1}(\widehat{G}^{(1)}_{xy}\widehat{G}^{(1)}_{ij})\partial^{\delta_{2}}_{z_{2}}\widehat{G}^{({2})}_{ij}\cdots \partial^{\delta_{\nu_1}}_{z_{\nu_1}}\widehat{G}^{({\nu_1})}_{ij} \widetilde{\cal P}^{(F)}_{ij})}{\partial H_{xy}^{k-n}} \frac{\partial^n \widetilde{\cal P}'_{uv}}{\partial H_{xy}^n}\nonumber\\
&\quad+\sum_{k=1}^{\ell}\frac{s_{k+1}}{k!}\frac{1}{N^{(k+3)/2}}{\sum_{x,y}}^* \bb E \frac{\partial^k (\partial^{\delta_1}_{z_1}(\widehat{G}^{(1)}_{xy}\widehat{G}^{(1)}_{ij})\partial^{\delta_{2}}_{z_{2}}\widehat{G}^{({2})}_{ij}\cdots \partial^{\delta_{\nu_1}}_{z_{\nu_1}}\widehat{G}^{({\nu_1})}_{ij} \widetilde{\cal P}^{(F)}_{ij})}{\partial H_{xy}^k} \langle \widetilde{\cal P}'_{uv}\rangle\nonumber\\
&\quad -\sum_{k=1}^{\ell}\sum_{n=1}^{k}{k \choose n}\frac{s_{k+1}}{k!}\frac{1}{N^{(k+3)/2}}\sum_{x,y:x \ne i} \bb E \frac{\partial^{k-n} (\partial^{\delta_1}_{z_1}(\widehat{G}^{(1)}_{xj}\widehat{G}^{(1)}_{yy})\partial^{\delta_{2}}_{z_{2}}\widehat{G}^{({2})}_{ij}\cdots \partial^{\delta_{\nu_1}}_{z_{\nu_1}}\widehat{G}^{({\nu_1})}_{ij} \widetilde{\cal P}^{(F)}_{ij})}{\partial H_{ix}^{k-n}}\frac{\partial^n \widetilde{\cal P}'_{uv}}{\partial H_{ix}^n}\nonumber\\
&\quad- \sum_{k=1}^{\ell}\frac{s_{k+1}}{k!}\frac{1}{N^{(k+3)/2}}\sum_{x,y:x \ne i} \bb E \frac{\partial^k (\partial^{\delta_1}_{z_1}(\widehat{G}^{(1)}_{xj}\widehat{G}^{(1)}_{yy})\partial^{\delta_{2}}_{z_{2}}\widehat{G}^{({2})}_{ij}\cdots \partial^{\delta_{\nu_1}}_{z_{\nu_1}}\widehat{G}^{({\nu_1})}_{ij} \widetilde{\cal P}^{(F)}_{ij})}{\partial H_{ix}^k}\langle\widetilde{\cal P}'_{uv}\rangle+O_{\prec}(\cal E_*(\cal P)^2N^{-4})\nonumber\\
&\eqd \sum_{k=1}^{\ell}\sum_{n=1}^k \bb EL_{k,n}^{(9)}+\sum_{k=1}^{\ell} \bb EL_{k}^{(10)}+\sum_{k=1}^{\ell}\sum_{n=1}^k \bb EL_{k,n}^{(11)}+\sum_{k=1}^{\ell} \bb EL_{k}^{(12)}+O_{\prec}(\cal E_*(\cal P)^2N^{-4}) \notag\,.
\end{align}
The treatment of $L_{k,n}^{(9)}$ and $L_{k,n}^{(11)}$ is similar to that of $L_{k,n}^{(7)}$ in Case 2: we use Lemma \ref{lem5.4} to estimate the worst terms, and the other terms can be estimated by the trivial bound in Lemma \ref{lem4.5}. As a result, we have
\begin{equation}  \label{5.43}
\sum_{k=1}^\ell\sum_{n=1}^k\Big\{{\sum_{i,j}}^*{\sum_{u,v}}^*\bb EL_{k,n}^{(9)}\Big\}_{2\sigma} \prec \cal E_*(\cal P)^2\quad \mbox{and}\quad 
\sum_{k=1}^\ell\sum_{n=1}^k\Big\{{\sum_{i,j}}^*{\sum_{u,v}}^*\bb EL_{k,n}^{(11)}\Big\}_{2\sigma} \prec \cal E_*(\cal P)^2\,.
\end{equation}
By Lemma \ref{lem4.5} we can also show that
\begin{equation} \label{5.44}
\sum_{k=2}^\ell\Big\{{\sum_{i,j}}^*{\sum_{u,v}}^*\bb E L_k^{(10)}\Big\}_{2\sigma} = O_{\prec}(\cal E_*(\cal P)) \cdot (\bb E |\cal S(\cal{ \mathring{P}})|^2)^{1/2}\,.
\end{equation}	
The estimates concerning $L_{k}^{(12)}$ is more complicated than what we saw for $L_{k}^{(8)}$ in Case 2, and the reason is that we have more than 1 off-diagonal entries in $\cal P$. When the derivative $\partial^k/\partial H_{ix}^k$ hits the off-diagonal entries $\partial^{\delta_{2}}_{z_{2}}\widehat{G}^{({2})}_{ij}\cdots \partial^{\delta_{\nu_1}}_{z_{\nu_1}}\widehat{G}^{({\nu_1})}_{ij} $, direct estimate by Lemma \ref{lem4.5} is not small enough for all terms. But for these terms the value of $\nu_1$ will decrease. Using the cancellation between the leading terms of $\bb EL_{1}^{(10)}$ and $\bb EL_{1}^{(12)}$, we have
\begin{align}
\bb EL_{1}^{(10)}+EL_{1}^{(12)}&=\frac{1}{N^2(\nu_1-2)!}\sum_{s\in \sym(\{2,...,\nu_1\})}\sum_{x,y}  \bb E \Big( \partial^{\delta_1}_{z_1}(\widehat{G}^{(1)}_{xj}\widehat{G}^{(1)}_{yy})\partial^{\delta_{s(2)}}_{z_{s(2)}}(\widehat{G}^{({s(2)})}_{xj}\widehat{G}^{(s(2))}_{ii}) \nonumber\\
&\quad    \times \partial^{\delta_{s(3)}}_{z_{s(3)}}\widehat{G}^{({s(3)})}_{ij}\cdots \partial^{\delta_{s(\nu_1)}}_{z_{s(\nu_1)}}\widehat{G}^{({s(\nu_1)})}_{ij}\cdot\widetilde{\cal P}^{(F)}_{ij}\langle\widetilde{\cal P}'_{uv}\rangle\Big) +\bb E \cal E^{(8)}_{ij} \langle\widetilde{\cal P}'_{uv}\rangle\nonumber\\
&\quad\eqd \sum_{s\in \sym(\{2,...,\nu_1\})}\sum_{x,y}  \bb E \widetilde{\cal P}^{(12,s)}_{1} \langle\widetilde{\cal P}'_{uv}\rangle+\bb E \cal E^{(8)}_{ij} \langle\widetilde{\cal P}'_{uv}\rangle\,, \label{5.55}
\end{align}
where $\sym(M)$ denotes the symmetric group of set $M$, and we can use Lemma \ref{lem4.5} to show that
$
\{\cal S(\cal E^{(8)}) \}_\sigma\prec \cal E_*(\cal P ).
$
In addition, we have
\begin{align}  
	\bb EL_{2}^{(12)}&=\frac{s_3}{N^{5/2}(\nu_1-3)!}\sum_{s\in \sym(\{2,...,\nu_1\})}\sum_{x:x \ne \{i,j\}}\sum_{y}  \bb E \Big( \partial^{\delta_1}_{z_1}(\widehat{G}^{(1)}_{xj}\widehat{G}^{(1)}_{yy})\partial^{\delta_{s(2)}}_{z_{s(2)}}(\widehat{G}^{({s(2)})}_{xj}\widehat{G}^{(s(2))}_{ii})\label{5.56}\\
	&\quad \times\partial^{\delta_{s(3)}}_{z_{s(3)}}(\widehat{G}^{({s(3)})}_{xj}\widehat{G}^{(s(3))}_{ii}) \partial^{\delta_{s(4)}}_{z_{s(4)}}\widehat{G}^{({s(4)})}_{ij}\cdots \partial^{\delta_{s(\nu_1)}}_{z_{s(\nu_1)}}\widehat{G}^{({s(\nu_1)})}_{ij} \widetilde{\cal P}^{(F)}_{ij}\langle\widetilde{\cal P}'_{uv}\rangle\Big)\notag\\
	&\quad  +\frac{s_3}{N^{5/2}(\nu_1-3)!}\sum_{s\in \sym(\{2,...,\nu_1\})}\sum_{y}  \bb E \Big( \partial^{\delta_1}_{z_1}(\widehat{G}^{(1)}_{jj}\widehat{G}^{(1)}_{yy})\partial^{\delta_{s(2)}}_{z_{s(2)}}(\widehat{G}^{({s(2)})}_{jj}\widehat{G}^{(s(2))}_{ii})\notag\\
	&\quad \times\partial^{\delta_{s(3)}}_{z_{s(3)}}(\widehat{G}^{({s(3)})}_{jj}\widehat{G}^{(s(3))}_{ii}) \partial^{\delta_{s(4)}}_{z_{s(4)}}\widehat{G}^{({s(4)})}_{ij}\cdots \partial^{\delta_{s(\nu_1)}}_{z_{s(\nu_1)}}\widehat{G}^{({s(\nu_1)})}_{ij} \widetilde{\cal P}^{(F)}_{ij}\langle\widetilde{\cal P}'_{uv}\rangle\Big) +\bb E \cal E^{(9)}_{ij} \langle\widetilde{\cal P}'_{uv}\rangle\notag\\
	&\eqd \sum_{s\in \sym(\{2,...,\nu_1\})}\sum_{x: x \ne\{i,j\}}\sum_{y}  \bb E \widetilde{\cal P}^{(12,s)}_{2,1} \langle\widetilde{\cal P}'_{uv}\rangle\notag\\
	&\quad  +\sum_{s\in \sym(\{2,...,\nu_1\})}\sum_{y}  \bb E \widetilde{\cal P}^{(12,s)}_{2,2} \langle\widetilde{\cal P}'_{uv}\rangle+\bb E \cal E^{(9)}_{ij} \langle\widetilde{\cal P}'_{uv}\rangle \notag
\end{align}
and 
\begin{equation} \label{5.45}
\begin{aligned}
	\bb EL_{k}^{(12)}&=
	\frac{s_{k+1}}{N^{(k+3)/2}(\nu_1-k-1)!}\\&
	\times\sum_{s\in \sym(\{2,...,\nu_1\})}\sum_{y}  \bb E \Big( \partial^{\delta_1}_{z_1}(\widehat{G}^{(1)}_{jj}\widehat{G}^{(1)}_{yy})\cdots\partial^{\delta_{s(K+1)}}_{z_{s(k+1)}}(\widehat{G}^{({s(K+1)})}_{jj}\widehat{G}^{(s(k+1))}_{ii})\\
	&\quad \times \partial^{\delta_{s(k+2)}}_{z_{s(k+2)}}\widehat{G}^{({s(k+2)})}_{ij}\cdots \partial^{\delta_{s(\nu_1)}}_{z_{s(\nu_1)}}\widehat{G}^{({s(\nu_1)})}_{ij} \widetilde{\cal P}^{(F)}_{ij}\langle\widetilde{\cal P}'_{uv}\rangle\Big) +\bb E \cal E^{(10,k)}_{ij} \langle\widetilde{\cal P}'_{uv}\rangle\\
	&\eqd\sum_{s\in \sym(\{2,...,\nu_1\})}\sum_{y}  \bb E \widetilde{\cal P}^{(12,s)}_{k} \langle\widetilde{\cal P}'_{uv}\rangle+\bb E \cal E^{(10,k)}_{ij} \langle\widetilde{\cal P}'_{uv}\rangle
\end{aligned}
\end{equation}
for $3\leq k \leq \nu_1-1$, where we can use Lemma \ref{lem4.5} and Theorem \ref{refthm1} to show that
$$\{\cal S(\cal E^{(9)}) \}_\sigma+\sum_{k=3}^{(\nu_1-1)\wedge \ell}\{\cal S(\cal E^{(10,k)}) \}_\sigma\prec \cal E_*(\cal P )\,.
$$
We see that for each $s \in \sym(\{2,...,\sigma\})$, we have $\cal E_0(\cal P^{(12,s)}_{2,1})=\cal E_0(\cal P)N^{-1/2}$, which is not enough. But we can again use 
$
\widehat{G}_{jx}=\ul{\widehat{H}\widehat{G}} \widehat{G}_{jx}-(\widehat{H}\widehat{G})_{jx}\ul{\widehat{G}}+\delta_{jx}\ul{\widehat{G}}
$ 
and Lemma \ref{lem:cumulant_expansion} to do another expansion. Estimating the result using Lemma \ref{lem4.5} and \eqref{swd}, we get
\begin{align}
&\sum_{x: x \ne\{i,j\}}\sum_{y}  \bb E \widetilde{\cal P}^{(12,s)}_{2,1} \langle\widetilde{\cal P}'_{uv}\rangle \notag\\
 =&\frac{s_3^2}{N^4(\nu_1-3)!} \sum_{x: x \notin \{i,j\}}\sum_{y}\bb E \Big( \partial^{\delta_1}_{z_1}(\widehat{G}^{(1)}_{jj}\widehat{G}^{(1)}_{yy})\partial^{\delta_{s(2)}}_{z_{s(2)}}(\widehat{G}^{({s(2)})}_{xx}\widehat{G}^{({s(2)})}_{jj}\widehat{G}^{(s(2))}_{ii})\label{5.47}\\
&\times	\partial^{\delta_{s(3)}}_{z_{s(3)}}(\widehat{G}^{({s(3)})}_{xx}\widehat{G}^{({s(3)})}_{jj}\widehat{G}^{(s(3))}_{ii}) \partial^{\delta_{s(4)}}_{z_{s(4)}}\widehat{G}^{({s(4)})}_{ij}\cdots \partial^{\delta_{s(\nu_1)}}_{z_{s(\nu_1)}}\widehat{G}^{({s(\nu_1)})}_{ij} \widetilde{\cal P}^{(F)}_{ij}\langle\widetilde{\cal P}'_{uv}\rangle\Big)\notag\\
&+\bb E \cal E_{ij}^{(11,s)}\langle \widetilde{\cal P}'_{uv}\rangle\nonumber\\
&\eqd  \sum_{x: x \ne\{i,j\}}\sum_{y}  \bb E  \widetilde{\cal P}^{(12,s)}_{2,1,1} \langle\widetilde{\cal P}'_{uv}\rangle +\bb E \cal E_{ij}^{(11,s)}\langle \widetilde{\cal P}'_{uv}\rangle\,, \notag
\end{align}
where $\{\cal S(\cal E_{ij}^{(11,s)})\}_\sigma\prec \cal E_*(\cal P)$. In addition, we also have
\begin{equation} \label{5.445}
\Big\{{\sum_{i,j}}^*{\sum_{x,y}}^*\bb E L_{k}^{(12)}\Big\}_{2\sigma} =O_{\prec}(\cal E_*(\cal P))\cdot  (\bb E |\cal S(\cal{ \mathring{P}})|^2)^{1/2}
\end{equation}
for $k \geq \nu_1$. By \eqref{5.43} -- \eqref{5.445}, together with the identity $\bb E X\langle Y \rangle =\bb E \langle X \rangle\langle Y \rangle$, we have
\begin{equation} \label{5.49}
\begin{aligned}
	\bb E |\cal S(\cal{ \mathring{P}})|^2 &=\sum_{s \in \sym(\{2,...,\nu_1\})}\bigg(\bb E \cal S({\cal { \mathring{P}}}_{1}^{(12,s)})\cal S(\cal{ \mathring{P}})+\bb E \cal S({\cal { \mathring{P}}}_{2,1,1}^{(12,s)})\cal S(\cal{ \mathring{P}})+\bb E \cal S({\cal { \mathring{P}}}_{2,2}^{(12,s)})\cal S(\cal{ \mathring{P}})\\
	&\quad +\sum_{k=3}^{(\nu_1-1)\wedge \ell}\bb E \cal S({\cal { \mathring{P}}}_{k}^{(12,s)})\cal S(\cal{ \mathring{P}})\bigg)
	+O_{\prec}(\cal E_*(\cal P)^2)+O_{\prec}(\cal E_*(\cal P)) \cdot (\bb E |\cal S(\cal{ \mathring{P}})|^2)^{1/2}\,.
\end{aligned}
\end{equation}
We see that $\cal E_0(\cal P^{(12,s)}_1)=\cal E_0(\cal P^{(12,s)}_{2,2})=\cal E_0(\cal P^{(12,s)}_{k})=\cal E_0(\cal P)$, $\cal E_0({\cal { P}}_{2,1,1}^{(12,s)})=\cal E_0(\cal P)N^{-1/2}$, and $\nu_1(\cal P^{(12,s)}_1)$, $\nu_1(\cal P_{2,1,1}^{(12,s)})$, $\nu_1(\cal P^{(12,s)}_{2,2})$, $\nu_1(\cal P^{(12,s)}_{k})\leq  \nu_1(\cal P)-2$ for all $s \in \sym\{2,...,\sigma\}$ and $k\leq \nu_1-1$. Rewriting \eqref{5.49} in an abstract form, we get 
\begin{equation} \label{5.50}
\bb E |\cal S(\cal{ \mathring{P}})|^2=\sum_{q=1}^n\bb E \cal S(\cal{ \mathring{P}}^{(q)})\cal S(\cal{ \mathring{P}})+O_{\prec}(\cal E_*(\cal P)^2)+O_{\prec}(\cal E_*(\cal P)) \cdot (\bb E |\cal S(\cal{ \mathring{P}})|^2)^{1/2}
\end{equation}
for some fixed $n$. Each $\cal P^{(q)}$ satisfies $\cal E_0(\cal P^{(q)})\leq \cal E_0(\cal P)$, $\nu_1(\cal P^{(q)})\leq \nu_1(\cal P)-2$, and each $\bb E \cal S(\cal{ \mathring{P}}^{(q)})\cal S(\cal{ \mathring{P}})$ can be expanded again using \eqref{resolvent eq for G} and Lemma \ref{lem:cumulant_expansion}. Similar to Cases 1 and 2, one can show that $\cal S(\cal P)\prec  \cal E_0(\cal P)N^{-1}=\cal E_*(\cal P)$ when $\nu_1(\cal P)=0$. Hence repeating \eqref{5.50} finitely many steps we get
\[
\bb E |\cal S(\cal{ \mathring{P}})|^2=O_{\prec}(\cal E_*(\cal P)^2)+O_{\prec}(\cal E_*(\cal P)) \cdot (\bb E |\cal S(\cal{ \mathring{P}})|^2)^{1/2}\,.
\]
This finishes the proof.

\subsection{Proof of Lemma \ref{lem5.1} (ii)} \label{sec5.3}

In this section we prove Lemma \ref{lem5.1} for $d(\cal P)=1$. In this case we have $\cal E_0(\cal P)=\cal E_*(\cal P)N^{1/2}$, thus we need to gain an improvement of $N^{-1/2}$. Comparing with the need of improving $N^{-1}$ in Section \ref{sec5.2}, the estimate in this section is easier.

\textit{Case 1.} Let us first consider the case $\nu_1(\cal P)=0$. Note that we have for $\widehat{G}\equiv \widehat{G}(z)$ that
\begin{equation*}
	\sum_i (\widehat{G}H_{\mathrm d}\widehat{G}^{1+\omega})_{ii}=\sum_{k} \widehat{G}^{2+\omega}_{kk} H_{kk} \prec \frac{1}{|\eta|^{1+\omega}}\,,
\end{equation*}
uniformly for $z=E+\ii \eta \in \b D$. Together with Theorem \ref{refthm1} we have
\begin{align}
	\cal S(\widetilde{\cal P}) &\prec  t^{\mu(\cal P)}N^{\nu(\cal P)-\theta(\cal P)-\nu_1(\cal P)/2-\nu_3(\cal P)-d(\cal P)/2} |\eta_1|^{-\nu_0(\cal G_1)/2} \cdots |\eta_\sigma|^{-\nu_0(\cal G_\sigma)/2}\notag\\
	&\quad \cdot (|N\eta_1|^{-1/2}+\cdots+|N\eta_\sigma|^{-1/2})\nonumber\\
	&\prec N |\eta_1|^{-5/2}\cdots N |\eta_\sigma|^{-5/2}\,, \label{5.51}
\end{align}
where in the second step we used $\nu_0(\cal P)\leq 4$. By Lemma \ref{lem:integration} we have
\[
\cal S(\cal P)=O_{\prec}(N)=O_{\prec}(\cal E_*(\cal P))\,,
\]
and together with Lemma \ref{prop_prec} we complete the proof.

\textit{Case 2.} Now let us consider the case when $\nu_1(\cal P)=1$. In this situation, the additional factor of $N^{-1/2}$ comes from the isotropic law. If the off-diagonal entry is in the form $\partial^{\delta}_z G_{ij}$, and we also have $(\widehat{G}H_{\mathrm d}\widehat{G}^{1+\omega})_{ii}$ in $\cal P$, then \eqref{5.51} is obviously true from Lemma \ref{prop4.4} and
\[
\sum_{j} \widehat{G}_{ij} \prec \frac{1}{|\eta|^{1/2}}\,.
\]
If the off-diagonal entry is in the form $(\widehat{G}H_{\mathrm d}\widehat{G}^{1+\omega})_{ij}$, then we use the estimate
\[
\frac{1}{\sqrt{N}} \sum_{j} (\widehat{G}H_{\mathrm d}\widehat{G}^{1+\omega})_{ij}=\sum_{k} \widehat{G}_{ik} \widehat{G}_{k\b v} H_{kk} \prec \frac{1}{|N\eta^{1+\omega}|}\,,
\]
where $\b v=(1,...,1)$. Together with Lemma \ref{prop4.4} we again obtain \eqref{5.51}. Following the steps in Case 1 we complete the proof.

\textit{Case 3.} We consider the case $\nu_1(\cal P )\geq 2$. We work under the additional assumption that $\cal P$ contains the off-diagonal factor $(\widehat{F}H_{\mathrm d}\widehat{F}^{1+\omega})_{ij}$, where $\widehat{F}\deq \widehat{G}(z_{\nu_1})$, $\omega \in \{0,1\}$; the case where $\cal P$ contains  $(\widehat{F}H_{\mathrm d}\widehat{F}^{1+\omega})_{ii}$ or $(\widehat{F}H_{\mathrm d}\widehat{F}^{1+\omega})_{jj}$ can be proved in a similar way. Let $\partial^{\delta_1}_{z_1}\widehat{G}^{(1)}_{ij}$,...,$\partial^{\delta_{\nu_1-1}}_{z_{\nu_1-1}}\widehat{G}^{({\nu_1-1})}_{ij}$ be the other off-diagonal factors in $\cal P$, where $\delta_1,...,\delta_{\nu_1-1} \in \{0,1\}$. Let 
$$
\widetilde{\cal P}^{(G)} \deq\widetilde{\cal P}/ (\partial^{\delta_1}_{z_1}\widehat{G}^{(1)}_{ij}\cdots \partial^{\delta_{\nu_1-1}}_{z_{\nu_1-1}}\widehat{G}^{({\nu_1-1})}_{ij} (\widehat{F}H_{\mathrm d}\widehat{F}^{1+\omega})_{ij})\,.
$$ 
Note that the condition $\nu_1(\cal P)\geq 2$ grantees the existence of $\partial^{\delta_1}_{z_1}\widehat{G}^{(1)}_{ij}$ in $\cal P$. Similar to \eqref{5.27}, we have for $i\ne j$ and $u \ne v$ that    
\begin{align}
	&\quad \ 	\bb E \widetilde{\cal P}_{ij}  \widetilde{\cal P}'_{uv} \nonumber\\
	&=\sum_{k=1}^{\ell}\sum_{n=1}^{k}{k \choose n}\frac{s_{k+1}}{k!}\frac{1}{N^{(k+3)/2}}{\sum_{x,y}}^* \bb E \frac{\partial^{k-n}(\partial^{\delta_1}_{z_1}(\widehat{G}^{(1)}_{xy}\widehat{G}^{(1)}_{ij})\partial^{\delta_{2}}_{z_{2}}\widehat{G}^{({2})}_{ij}\cdots \partial^{\delta_{\nu_1-1}}_{z_{\nu_1-1}}\widehat{G}^{({\nu_1-1})}_{ij} (\widehat{F}H_{\mathrm d}\widehat{F}^{1+\omega})_{ij}\widetilde{\cal P}^{(G)}_{ij})}{\partial H_{xy}^{k-n}} \frac{\partial^n \widetilde{\cal P}'_{uv}}{\partial H_{xy}^n}\nonumber\\
	&\quad+\sum_{k=1}^{\ell}\frac{s_{k+1}}{k!}\frac{1}{N^{(k+3)/2}}{\sum_{x,y}}^* \bb E \frac{\partial^k (\partial^{\delta_1}_{z_1}(\widehat{G}^{(1)}_{xy}\widehat{G}^{(1)}_{ij})\partial^{\delta_{2}}_{z_{2}}\widehat{G}^{({2})}_{ij}\cdots \partial^{\delta_{\nu_1-1}}_{z_{\nu_1-1}}\widehat{G}^{({\nu_1-1})}_{ij} (\widehat{F}H_{\mathrm d}\widehat{F}^{1+\omega})_{ij}\widetilde{\cal P}^{(G)}_{ij})}{\partial H_{xy}^k} \widetilde{\cal P}'_{uv}\nonumber\\
	&\quad -\sum_{k=1}^{\ell}\sum_{n=1}^{k}{k \choose n}\frac{s_{k+1}}{k!}\frac{1}{N^{(k+3)/2}}\sum_{x,y:x \ne i} \bb E \frac{\partial^{k-n} (\partial^{\delta_1}_{z_1}(\widehat{G}^{(1)}_{xj}\widehat{G}^{(1)}_{yy})\partial^{\delta_{2}}_{z_{2}}\widehat{G}^{({2})}_{ij}\cdots \partial^{\delta_{\nu_1-1}}_{z_{\nu_1-1}}\widehat{G}^{({\nu_1-1})}_{ij} (\widehat{F}H_{\mathrm d}\widehat{F}^{1+\omega})_{ij} \widetilde{\cal P}^{(G)}_{ij})}{\partial H_{ix}^{k-n}}\frac{\partial^n \widetilde{\cal P}'_{uv}}{\partial H_{ix}^n}\nonumber\\
	&\quad- \sum_{k=1}^{\ell}\frac{s_{k+1}}{k!}\frac{1}{N^{(k+3)/2}}\sum_{x,y:x \ne i} \bb E \frac{\partial^k (\partial^{\delta_1}_{z_1}(\widehat{G}^{(1)}_{xj}\widehat{G}^{(1)}_{yy})\partial^{\delta_{2}}_{z_{2}}\widehat{G}^{({2})}_{ij}\cdots \partial^{\delta_{\nu_1-1}}_{z_{\nu_1-1}}\widehat{G}^{({\nu_1-1})}_{ij} (\widehat{F}H_{\mathrm d}\widehat{F}^{1+\omega})_{ij}\widetilde{\cal P}^{(G)}_{ij})}{\partial H_{ix}^k}\widetilde{\cal P}'_{uv}+O_{\prec}(\cal E_*(\cal P)^2N^{-4})\nonumber\\
	&\eqd \sum_{k=1}^{\ell}\sum_{n=1}^k \bb EL_{k,n}^{(13)}+\sum_{k=1}^{\ell} \bb EL_{k}^{(14)}+\sum_{k=1}^{\ell}\sum_{n=1}^k \bb EL_{k,n}^{(15)}+\sum_{k=1}^{\ell} \bb EL_{k}^{(16)}+O_{\prec}(\cal E_*(\cal P)^2N^{-4}) \label{5.53}\,.
\end{align}
Note that the entries of $\widehat{F}H_{\mathrm d}\widehat{F}^{1+\omega}$ behaves very similar to those of $\partial_z^{\delta}\widehat{G}$, in the way that
\begin{equation} \label{ww}
	\begin{aligned}
		\frac{\partial{(\widehat{F}H_{\mathrm d}\widehat{F}^{1+\omega})_{ij}}}{\partial H_{kl}}&=-\widehat{F}_{ik}(\widehat{F}H_{\mathrm d}\widehat{F}^{1+\omega})_{lj}-\widehat{F}_{il}(\widehat{F}H_{\mathrm d}\widehat{F}^{1+\omega})_{kj}-(\widehat{F}H_{\mathrm d}\widehat{F}^{1+\omega})_{ik}\widehat{F}_{lj}-(\widehat{F}H_{\mathrm d}\widehat{F}^{1+\omega})_{il}\widehat{F}_{kj}\\
		&\quad -\omega(\widehat{F}H_{\mathrm d}\widehat{F})_{ik}\partial^{\omega}_{z_{\nu_1}}\widehat{F}_{lj}-\omega(\widehat{F}H_{\mathrm d}\widehat{F})_{il}\partial^{\omega}_{z_{\nu_1}}\widehat{F}_{kj}\,.
	\end{aligned}
\end{equation}
Unlike in Cases 2 and 3 of Section \ref{sec5.2}, we cannot use Lemma \ref{lem5.4} in the estimates of the mixed terms $L_{k,n}^{(13)}$ and $L_{k,n}^{(15)}$. On the other hand, the mixed terms in \eqref{5.53} are easier to estimate, as each of them contains two entries of $\widehat{F}H_{\mathrm d}\widehat{F}^{1+\omega}$, which gives an additional factor of $N^{-1}$. By applying the differentials carefully using \eqref{diff1}, \eqref{ww}, and estimating the result directly by Lemma \ref{lem4.5}, we can show that
\begin{equation*}  
	\sum_{k=1}^\ell\sum_{n=1}^k\Big\{{\sum_{i,j}}^*{\sum_{u,v}}^*\bb EL_{k,n}^{(13)}\Big\}_{2\sigma} \prec \cal E_*(\cal P)^2\quad \mbox{and}\quad 
	\sum_{k=1}^\ell\sum_{n=1}^k\Big\{{\sum_{i,j}}^*{\sum_{u,v}}^*\bb EL_{k,n}^{(15)}\Big\}_{2\sigma} \prec \cal E_*(\cal P)^2\,.
\end{equation*}
In addition, \eqref{diff1}, \eqref{ww}, and Lemma \ref{lem4.5} can also be used to show that
\begin{equation*}
	\sum_{k=2}^\ell\Big\{{\sum_{i,j}}^*{\sum_{u,v}}^*\bb E L_k^{(14)}\Big\}_{2\sigma} = O_{\prec}(\cal E_*(\cal P)) \cdot (\bb E |\cal S(\cal{ {P}})|^2)^{1/2}\,.
\end{equation*}
Similar to \eqref{5.55} -- \eqref{5.47}, we can show that
$L_{k}^{(16)}$ contains several terms where the parameter $\cal E_0$ is unchanged, but $\nu_1$ is reduced by at least 2. Hence similar to \eqref{5.50}, we have
\begin{equation} \label{5.500}
	\bb E |\cal S(\cal{ {P}})|^2=\sum_{q=1}^n\bb E \cal S(\cal{ {P}}^{(q)})\cal S(\cal{ {P}})+O_{\prec}(\cal E_*(\cal P)^2)+O_{\prec}(\cal E_*(\cal P)) \cdot (\bb E |\cal S(\cal{ \mathring{P}})|^2)^{1/2}
\end{equation}
for some fixed $n$. Each $\cal P^{(q)}$ satisfies $\cal E_0(\cal P^{(q)})\leq \cal E_0(\cal P)$, $\nu_1(\cal P^{(q)})\leq \nu_1(\cal P)-2$, and each $\bb E \cal S(\cal{ {P}}^{(q)})\cal S(\cal{ {P}})$ can be expanded again using \eqref{resolvent eq for G} and Lemma \ref{lem:cumulant_expansion}.  Note from Cases 1 and 2 that $\cal S(\cal P)\prec  \cal E_0(\cal P)N^{-1}=\cal E_*(\cal P)$ when $\nu_0(\cal P)\in \{0,1\}$. Hence repeating \eqref{5.500} finitely many steps we get
\[
\bb E |\cal S(\cal{ {P}})|^2=O_{\prec}(\cal E_*(\cal P)^2)+O_{\prec}(\cal E_*(\cal P)) \cdot (\bb E |\cal S(\cal{ {P}})|^2)^{1/2}\,.
\]
This concludes the proof.

\section{The lower bound --  Proof of Theorem \ref{thmlowerbound}} \label{s. lowbound}

In this section, we establish a lower bound for the convergence rate.
The key technical  step for the proof of  Theorem \ref{thmlowerbound} is the following proposition on the estimate of the three point function of the Green function.   Fix small $\delta>0$ such that $f$ is analytic in $[-2-10\delta,2+10\delta]$.  We define the domains for $a=1,2,3$,
\begin{align*}
	 \mathcal{S}_a\equiv \cal S_a(\delta)= \{z\in\mathbb{C}: \text{dist}(z, [-2,2])= a\delta\}\,,\quad \mathcal{S}_a^{\geq}\equiv\mathcal{S}_a^{\geq}(\delta)= \{z=E+\ii\eta\in \mathcal{S}_a: \eta\geq N^{-5}\}. 
\end{align*}
For brevity, we also set
\[
b_n\deq \cal C_n(\sqrt{N}H_{11})
\]
for all fixed $n \in \bb N_+$. 

Again, for brevity, in the sequel, we will focus on case $\beta=1$.

\begin{pro} \label{proGFK} Suppose that the assumptions in Theorem \ref{thmlowerbound} hold and $\beta=1$. Let $z_a\in  \mathcal{S}_a^\geq, a=1,2,3$, we have 
	\begin{align*}
		&\mathbb{E} \big\la\ul{G}(z_1)\big\ra \big\la\ul{G}(z_2)\big\ra \big\la\ul{G}(z_3)\big\ra \notag\\
		&=N^{-\frac72}b_3m'(z_1)m'(z_2)m'(z_3)+ 8N^{-4}\big(h(z_1,z_2,z_3)+h(z_3,z_2,z_1)+h(z_2,z_1,z_3)\big)\notag\\
		&\quad+N^{-4}b_4m'(z_1)m'(z_2)m'(z_3)\big(m(z_1)+m(z_2)+m(z_3)\big)+O_\prec\big(N^{-\frac92}\big), 
	\end{align*}
	where
	\begin{align*}
		h(u,v,w):=\frac{1}{(u^2-4)^{\frac32}(v^2-4)^{\frac12}(w^2-4)^{\frac12}(u-v)(w-u)}, \quad u,v,w\in \mathbb{C}\setminus [-2,2].
	\end{align*}
\end{pro}

\begin{rem}
	Note that the parameter $z\in \mathcal{S}_a$ is on global scale since it is away from $[-2,2]$ by a distance of constant order. It suffices to carry out all the estimates below for Green functions on $\mathcal{S}_a$ thanks to the analyticity of the test function $f$. Most of the error terms in the estimates concerning Green functions depend on $(|E^2-4|+\eta)^{-1}$ in the sequel. But we omit this dependence for simplicity since  $|E^2-4|+\eta\sim 1$ for $z\in \mathcal{S}_a$ anyway. 
\end{rem}

In the sequel, we first prove Theorem \ref{thmlowerbound}  based on Proposition \ref{proGFK}.
\begin{proof}[Proof of Theorem \ref{thmlowerbound}]
  By Cauchy integral formula
	\begin{equation} \label{7.111}
  (\text{Var}(\text{Tr} f_{\gamma}(H)))^{\frac12} 	\mathring{\mathcal{Z}}_{f,\gamma}\mathbf{1}\big(\|H\|\leq 2+N^{-\frac12}\big)=	\frac{\ii}{2\pi} \oint_{\mathcal{S}_a} f_\gamma(z_a) \big\la \ul{G}(z_a)\big\ra {\rm d}z\mathbf{1}\big(\|H\|\leq 2+N^{-\frac12}\big)\,,
\end{equation}
   we can write  
	\begin{align}
		&\quad (\text{Var}(\text{Tr} f_{\gamma}(H)))^{\frac32}\mathbb{E} \mathring{\mathcal{Z}}_{f,\gamma}^3\mathbf{1}\big(\|H\|\leq 2+N^{-\frac12}\big)=\mathbb{E}\big\la \tr f_\gamma(H) \big\ra^3\mathbf{1}\big(\|H\|\leq 2+N^{-\frac12}\big) \notag\\
		&=N^3 \mathbb{E}\prod_{a=1}^3 \bigg[ \frac{\ii}{2\pi} \oint_{\mathcal{S}_a} f_\gamma(z_a) \big\la \ul{G}(z_a)\big\ra {\rm d}z_a\bigg]\mathbf{1}\big(\|H\|\leq 2+N^{-\frac12}\big)\notag\\
		&= N^3 \frac{\ii^3}{(2\pi)^3} \oint_{\mathcal{S}_1^{\geq }\times \mathcal{S}_2^{\geq }\times \mathcal{S}_3^\geq} f_\gamma(z_1) f_{\gamma}(z_2)f_\gamma(z_3) \mathbb{E}\big\la \ul{G}(z_1)\big\ra \big\la \ul{G}(z_2)\big\ra\big\la \ul{G}(z_3)\big\ra {\rm d}z_1{\rm d}z_2 {\rm d}z_3+O(N^{-2})
		\label{111910}
	\end{align}	
	Applying Proposition \ref{proGFK}, we have 
	\begin{align}
		&\quad N^3 \frac{\ii^3}{(2\pi)^3} \oint_{\mathcal{S}_1^{\geq }\times \mathcal{S}_2^{\geq }\times \mathcal{S}_3^\geq} f_\gamma(z_1) f_{\gamma}(z_2)f_\gamma(z_3) \mathbb{E}\big\la \ul{G}(z_1)\big\ra \big\la \ul{G}(z_2)\big\ra\big\la \ul{G}(z_3)\big\ra {\rm d}z_1{\rm d}z_2 {\rm d}z_3 \notag\\
		&= N^{-\frac12} b_3\mathcal{L}_1+N^{-1}\mathcal{L}_2+N^{-1}b_4\mathcal{L}_3+O_\prec(N^{-\frac32}),  \label{121201}
	\end{align}
	where 
	\begin{align*}
		\mathcal{L}_1= &\frac{\ii^3}{(2\pi)^3}  \oint_{ \mathcal{S}_{1}\times  \mathcal{S}_{2}\times \mathcal{S}_{3} } f_\gamma(z_1) f_\gamma(z_2)f_\gamma(z_3) m'(z_1)m'(z_2)m'(z_3)  {\rm d} z_1 {\rm d} z_2{\rm d} z_3  \notag\\
		\mathcal{L}_2= &\frac{\ii^3}{\pi^3}  \oint_{\mathcal{S}_{1}\times  \mathcal{S}_{2}\times \mathcal{S}_{3}}  f_\gamma(z_1) f_\gamma(z_2)f_\gamma(z_3) \big(h(z_1,z_2,z_3)+h(z_3,z_2,z_1)+h(z_2,z_1,z_3)\big) {\rm d} z_1 {\rm d} z_2 {\rm d} z_3,\notag\\
		\mathcal{L}_3= &\frac{\ii^3}{(2\pi)^3}  \oint_{\mathcal{S}_{1}\times  \mathcal{S}_{2}\times \mathcal{S}_{3}} f_\gamma(z_1) f_\gamma(z_2)f_\gamma(z_3) m'(z_1)m'(z_2)m'(z_3)\big(m(z_1)+m(z_2)+m(z_3)\big){\rm d} z_1 {\rm d} z_2 {\rm d} z_3.
	\end{align*}
	
	In the sequel, we estimate $\mathcal{L}_1$, $\mathcal{L}_2$ and $\mathcal{L}_3$. To this end, we denote by $m_k\equiv m(z_k)$ for simplicity and apply the identities
	\begin{align}
		z_k=-(m_k+m_k^{-1}), \qquad \sqrt{z_k^2-4}= \frac{m_k^2-1}{m_k}=-\frac{m_k}{m'_k}. \label{112002}
	\end{align}
	We further write 
	\begin{align}
		m_k=\rho_k \mathrm{e}^{\ii\theta_k}, \qquad 1>\rho_1>\rho_2>\rho_3>0.  \label{112001}
	\end{align}	
	Notice that when $z$ goes counterclockwise, $m(z)$  goes clockwise. With the above parameterization, we have 
	\begin{align}
		\mathcal{L}_1= & -\frac{\ii^3}{(2\pi)^3} \oint_{|m_1|=\rho_1}  \oint_{|m_2|=\rho_2} \oint_{|m_3|=\rho_3}\prod_{a=1}^3f_\gamma(-m_a-m_a^{-1})   {\rm d} m_1 {\rm d} m_2{\rm d} m_3\notag\\
		&=-\lim_{\rho_3\to 1}\lim_{\rho_2\to 1}\lim_{\rho_1\to 1}\frac{\ii^6}{(2\pi)^3} \prod_{k=1}^3\int_{0}^{2\pi}  \rho_k\mathrm{e}^{\ii\theta_k}f_\gamma(-\rho_k\mathrm{e}^{\ii \theta_k}-\rho_k^{-1}\mathrm{e}^{-\ii\theta_k}) {\rm d} \theta_k\notag\\
		&= -\frac{\ii^6}{(2\pi)^3} \Big(\int_{0}^{2\pi} \mathrm{e}^{\ii\theta}f_\gamma(-2\cos\theta) {\rm d} \theta\Big)^3=\frac{\ii^6}{(2\pi)^3} \Big(\int_{-\pi}^\pi  f_\gamma(2\cos\psi)\cos\psi {\rm d}\psi\Big)^3\notag\\
		&=-\frac{1}{8}(c_1^{f_\gamma})^3=-\frac{1}{8}(1-\gamma)^3(c_1^f)^3, \label{121202}
	\end{align}
	where in the first two steps we used Green's formula and the fact that $f_\gamma$ is analytic. Similarly, for $\mathcal{L}_3$, we have 
	\begin{align}
		\mathcal{L}_3= & -\frac{3\ii^6}{(2\pi)^3}\int_{-\pi}^{\pi} f_\gamma(2\cos\psi_1)\cos(2\psi_1){\rm d}\psi_1 \Big(\int_{-\pi}^{\pi}  f_\gamma(2\cos\psi_2)\cos(\psi_2) {\rm d }\psi_2\Big)^2= \frac{3}{8}c_2^{f_\gamma}(c_1^{f_\gamma})^2. \label{121203}
	\end{align}

	Next, we turn to the estimate of $\mathcal{L}_2$.  We further do the decomposition 
	\begin{align*}
		\mathcal{L}_2=\mathcal{L}_{2}^{123}+\mathcal{L}_{2}^{321}+\mathcal{L}_{2}^{213}, 
	\end{align*}
	where 
	\begin{align*}
		\mathcal{L}_{2}^{abc}= \frac{\ii^3}{\pi^3}  \oint_{\mathcal{S}_1\times \mathcal{S}_2\times \mathcal{S}_3 } f_\gamma(z_1)f_\gamma(z_2)f_\gamma(z_3)h(z_a,z_b,z_c) {\rm d} z_1 {\rm d} z_2 {\rm d} z_3, \quad \{a,b,c\}=\{1,2,3\}. 
	\end{align*}
	We emphasize that the integral with three $h(z_a,z_b,z_c)$'s are not symmetric due to the assumptions in $\rho_k$'s in (\ref{112001}). However, the estimates of $\mathcal{L}_{2}^{abc}$'s are similar. We only state the details for $\mathcal{L}_{2}^{123}$ in the sequel. Applying (\ref{112002}) and Green's formula, we have 
	\begin{align*}
		\mathcal{L}_{2}^{123}&= \frac{\ii^3}{\pi^3}  \oint_{\mathcal{S}_1\times \mathcal{S}_2\times \mathcal{S}_3 }  \frac{f_\gamma(z_1)f_\gamma(z_2)f_\gamma(z_3)}{(z_1^2-4)^{\frac{3}{2}}(z_2^2-4)^{\frac12}(z_3^2-4)^{\frac12}(z_1-z_2)(z_3-z_1)} {\rm d} z_1 {\rm d}z_2 {\rm d}z_3\notag\\
		&=- \lim_{\rho_3\to 1}\lim_{\rho_2\to 1}\lim_{\rho_1\to 1}\frac{\ii^6}{\pi^3}\int_{[0,2\pi)^3} \rho_1^2\rho_2\rho_3    \mathrm{e}^{\ii (2\theta_1+\theta_2+\theta_3)} \notag\\
		&\qquad\times \frac{f_\gamma(-\rho_1\mathrm{e}^{\ii\theta_1}-\rho_1^{-1}\mathrm{e}^{-\ii\theta_1})f_\gamma(-\rho_2\mathrm{e}^{\ii\theta_2}-\rho_2^{-1}\mathrm{e}^{-\ii\theta_2})f_\gamma(-\rho_3\mathrm{e}^{\ii\theta_3}-\rho_3^{-1}\mathrm{e}^{-\ii\theta_3}) {\rm d} \theta_1{\rm d}\theta_2 {\rm d}\theta_3}{(1-\rho_1^2\mathrm{e}^{2\ii\theta_1})^2 (1-\rho_1^{-1}\rho_2\mathrm{e}^{\ii(\theta_2-\theta_1)})(1-\rho_1\rho_2\mathrm{e}^{\ii(\theta_1+\theta_2)})(1-\rho_1^{-1}\rho_3\mathrm{e}^{\ii(\theta_3-\theta_1)})(1-\rho_1\rho_3\mathrm{e}^{\ii(\theta_1+\theta_3)})}\notag\\
		&=  -\lim_{\rho_3\to 1}\lim_{\rho_2\to 1}\lim_{\rho_1\to 1}\frac{\ii^6}{\pi^3}\sum_{\alpha, \tau, \gamma,\sigma, \psi} (\psi+1) \notag\\
		&\qquad \times \int_{0}^{2\pi} \rho_1^{\tau-\alpha+\sigma-\gamma+2\psi+2}\mathrm{e}^{\ii(\tau-\alpha+\sigma-\gamma+2\psi+2)\theta_1} f_\gamma(-\rho_1\mathrm{e}^{\ii\theta_1}-\rho_1^{-1}\mathrm{e}^{-\ii\theta_1}) {\rm d}\theta_1\notag\\
		&\qquad \times  \int_{0}^{2\pi} \rho_2^{\alpha+\tau+1}\mathrm{e}^{\ii (\alpha+\tau+1)\theta_2} f_\gamma(-\rho_1\mathrm{e}^{\ii\theta_2}-\rho_1^{-1}\mathrm{e}^{-\ii\theta_2}) {\rm d} \theta_2\notag\\
		&\qquad \times \int_{0}^{2\pi} \rho_3^{\gamma+\sigma+1} \mathrm{e}^{\ii (\gamma+\sigma+1)\theta_3} f_\gamma(-\rho_3\mathrm{e}^{\ii\theta_3}-\rho_3^{-1}\mathrm{e}^{-\ii\theta_3}) {\rm d} \theta_3. 
	\end{align*}
	Hence,  we have
	\begin{align*}
		\mathcal{L}_{2}^{123}
		&=-\frac{\ii^6}{\pi^3} \sum_{\alpha, \tau, \gamma,\sigma, \psi=0}^{\infty}(\psi+1)\int_{-\pi}^{\pi} \cos\big((\tau-\alpha+\sigma-\gamma+2\psi+2)\theta_1\big) f_\gamma(2\cos\theta_1) {\rm d}\theta_1\notag\\
		&\qquad \times  \int_{-\pi}^{\pi} \cos \big((\alpha+\tau+1)\theta_2\big) f_\gamma(2\cos\theta_2) {\rm d} \theta_2 \int_{-\pi}^{\pi}  \cos \big((\gamma+\sigma+1)\theta_3\big) f_\gamma(2\cos\theta_3) {\rm d} \theta_3\notag\\
		&= \sum_{\alpha, \tau, \gamma,\sigma, \psi=0}^{\infty}(\psi+1) c_{\tau-\alpha+\sigma-\gamma+2\psi+2}^{f_\gamma} c_{\alpha+\tau+1}^{f_\gamma}c_{\gamma+\sigma+1}^{f_\gamma},
	\end{align*}
	Similarly , we can derive
	\begin{align*}
		&\mathcal{L}_{2}^{321}
		=   \sum_{\alpha,\tau,\gamma,\sigma,\psi=0}^{\infty}(\psi+1) c_{\sigma-\gamma}^{f_\gamma} c_{\tau-\alpha}^{f_\gamma} c_{\alpha+\tau+\gamma+\sigma+2\psi+4}^{f_\gamma}\,,\notag\\
		&\mathcal{L}_{2}^{213}
		=-\sum_{\alpha,\tau,\gamma,\sigma,\psi}(\psi+1) c_{\tau-\alpha}^{f_\gamma}c_{\gamma+\sigma+1}^{f_\gamma} c_{\alpha+\tau-\gamma+\sigma+2\psi+3}^{f_\gamma}. 
	\end{align*}
	In summary, we have 
	\begin{align}
		\mathcal{L}_2=&\sum_{\alpha, \tau, \gamma,\sigma, \psi=0}^{\infty}(\psi+1)\Big( c_{\tau-\alpha+\sigma-\gamma+2\psi+2}^{f_\gamma} c_{\alpha+\tau+1}^{f_\gamma}c_{\gamma+\sigma+1}^{f_\gamma}+c_{\sigma-\gamma}^{f_\gamma} c_{\tau-\alpha}^{f_\gamma} c_{\alpha+\tau+\gamma+\sigma+2\psi+4}^{f_\gamma}\notag\\
		&\qquad\qquad\qquad\qquad-c_{\tau-\alpha}^{f_\gamma}c_{\gamma+\sigma+1}^{f_\gamma} c_{\alpha+\tau-\gamma+\sigma+2\psi+3}^{f_\gamma}\Big). \label{121204}
	\end{align}
	Inserting \eqref{121201}, (\ref{121202}) -- (\ref{121204}) into \eqref{111910}, we have 
	\begin{align}
		\mathbb{E} \mathring{\mathcal{Z}}_{f,\gamma}^3\mathbf{1}\big(\|H\|\leq 2+N^{-\frac12}\big)= (\text{Var}(\text{Tr} f_{\gamma}(H)))^{-\frac32}\big(-r_1^{f_\gamma}N^{-\frac12}+r_2^{f_\gamma} N^{-1}\big)+O(N^{-\frac32})\,,
		\label{121215}
	\end{align}
and by another use of Corollary \ref{cor:spectral norm} we obtain \eqref{www} as desired. Next, we prove (\ref{121101}) by contradiction. 
	Denote by $F_{N,\gamma}(x)$ and $\Phi(x)$ the distribution function of $\mathring{\mathcal{Z}}_{f,\gamma}$ and standard normal, respectively. Further denote by $\widehat{F}_{N,\gamma}(x)$ the distribution function of $\mathring{\mathcal{Z}}_{f,\gamma}\mathbf{1}\big(\|H\|\leq 2+N^{-\frac12}\big)$. According to Corollary \ref{cor:spectral norm},  we have
	\[
\sup_{x \in \bb R}|	F_{N,\gamma}(x)-\widehat{F}_{N,\gamma}(x)|=O(N^{-2})\,,
	\] 
	thus it suffices to provide a lower bound for the distance between $\widehat{F}_{N,\gamma}$ and $\Phi$. 
	First, we write
	\begin{align}
		& \mathbb{E} \mathring{\mathcal{Z}}_{f,\gamma}^3\mathbf{1}\big(\|H\|\leq 2+N^{-\frac12}\big)= \int x^3 {\rm d} \big(\widehat{F}_{N,\gamma}(x)-\Phi(x)\big) \notag\\
		&\qquad\qquad= \int_{-N^{\frac{\kappa}{4}}}^{N^{\frac{\kappa}{4}}} x^3 {\rm d} \big(\widehat{F}_{N,\gamma}(x)-\Phi(x)\big)+ \mathbb{E} \mathring{\mathcal{Z}}_{f,\gamma}^3\mathbf{1}\big(\|H\|\leq 2+N^{-\frac12}, |\mathring{\mathcal{Z}}_{f,\gamma}|> N^{\frac{\kappa}{4}}\big). \label{121210}
	\end{align}
	For the second term in the RHS, it is easy to see from \eqref{7.111} and Theorem \ref{refthm1} that
	 $$
	\mathring{\mathcal{Z}}_{f,\gamma}\mathbf{1}\big(\|H\|\leq 2+N^{-\frac12}\big) \prec 1\,.
	$$ 
Hence
	\begin{align*}
		|\mathbb{E} \mathring{\mathcal{Z}}_{f,\gamma}^3\mathbf{1}\big(\|H\|\leq 2+N^{-\frac12}, |\mathring{\mathcal{Z}}_{f,\gamma}|> N^{\frac{\kappa}{4}}\big)|  \prec \bb P (\|H\|\leq 2+N^{-\frac12}, |\mathring{\mathcal{Z}}_{f,\gamma}|> N^{\frac{\kappa}{4}}\big) \prec N^{-2}\,.
	\end{align*}
For the first term in the RHS of (\ref{121210}), using integration by parts and rigidity, we have
	\begin{align*}
		\int_{-N^{\frac{\kappa}{4}}}^{N^\frac{\kappa}{4}} x^3 {\rm d} \big(\widehat{F}_{N,\gamma}(x)-\Phi(x)\big)=3\int_{-N^{\frac{\kappa}{4}}}^{N^\frac{\kappa}{4}}  x^2\big(\widehat{F}_{N,\gamma}(x)-\Phi(x)\big){\rm d}x+O(N^{-2})\,.
	\end{align*}
 Then, apparently, if (\ref{121101}) does not hold, (\ref{121215}) would not be true. Hence, we conclude (\ref{121101}) by contradiction. 
\end{proof}

In the sequel, we prove Proposition \ref{proGFK}. 
To facilitate our analysis, we first state an estimate for the two point functions. With certain abuse of notation, we set 
\begin{align*}
	T_i:=(-z_i-2\mathbb{E}\ul{G}(z_i))^{-1}. 
\end{align*}
Further, for brevity, we will use the following shorthand notations in the sequel
\begin{align*}
G\equiv G(z_1), \quad F\equiv G(z_2), \quad K\equiv G(z_3). 
\end{align*}
The following result is a trivial adaption of \cite[(4.11),(4.23)]{HK2} to our settings, whose proof is omitted.
\begin{lem} \label{lem.2-point-corr} Under the assumption of Proposition \ref{proGFK}, we have for $z_a\in \mathcal{S}_a, a=1,2$, we have
	\begin{align*} 
		\mathbb{E}\la \ul{G}\ra\la\ul{F}\ra= \frac{2}{N^2} T_2\mathbb{E} \ul{FG^2}+O_\prec \Big(N^{-\frac52}\Big)\,,
	\end{align*}
and
\begin{align*}
		\mathbb{E} \la \ul{G^2}\ra \la F\ra=\frac{4}{N^2} T_2\mathbb{E}\ul{FG^3}+O_\prec \Big(N^{-\frac52}\Big).    
	\end{align*}
\end{lem}

In the following discussion, we fix $z_a=E_a+\ii \eta_a\in \mathcal{S}_a, a=1,2,3$. We remark here that all the error terms in the following discussion  may be proportional to certain fixed power of $(|E_a^2-4|+\eta_a)^{-1}$ and $|z_a-z_b|^{-1}$ with $a\neq b$. But we omit this dependence for simplicity of the presentation since they are all order 1 quantities, thanks to our definition of the domain $\mathcal{S}_{a}$'s. Applying the identity $zG=HG-I$ and Lemma \ref{lem:cumulant_expansion}, we have 
\begin{align}
	z_1\mathbb{E}\la \ul{G}\ra \la \ul{F}\ra \la \ul{K}\ra =\mathbb{E} \ul{HG} \big\la \la \ul{F}\ra \la \ul{K}\ra\big\ra=\frac{1}{N}\sum_{ij} \mathbb{E} H_{ji} G_{ij} \big\la \la \ul{F}\ra \la \ul{K}\ra\big\ra= W_1+W_2+W_3+\frac{1}{N}\sum_{ij} R_{ij},  \label{100903}
\end{align}
where 
\begin{align*}
	W_k:= \frac{1}{N}\sum_{ij} \frac{1}{k!} \mathcal{C}_{k+1}(H_{ji}) \mathbb{E} \frac{\partial^k}{\partial H_{ji}^k}\Big( G_{ij}\big\la \la \ul{F}\ra \la \ul{K}\ra\big\ra\Big), 
\end{align*}
 Analogously to (\ref{diff}), we shall use
\begin{align}
	\frac{\partial}{\partial H_{k\ell}} G_{ij}= -(1+\delta_{k\ell})^{-1}\big(G_{ik}G_{\ell j}+G_{i\ell}G_{kj}\big)\,. \label{101301}
\end{align} 
A routine verification of the remainder term shows
$
N^{-1}\sum_{ij} R_{ij}=O_\prec(N^{-\frac92}).
$
In the squeal we deal with $W_1,W_2,W_3$.

\subsection{The term $W_1$} Recall the setting in Assumption \ref{122910} (i). 
Using (\ref{101301}) repeatedly, it is  straightforward to derive 
\begin{align}
	W_1= -\mathbb{E}\la(\ul{G})^2 \ra \la \ul{F}\ra \la \ul{K}\ra-\frac{1}{N}\mathbb{E} \la \ul{G^2} \ra \la \ul{F}\ra \la \ul{K}\ra - \frac{2}{N^2}\mathbb{E}\ul{GF^2}\la \ul{K}\ra -\frac{2}{N^2} \mathbb{E} \ul{GK^2}\la \ul{F}\ra.  \label{100901}
\end{align}
Notice that 
\begin{align}
	\mathbb{E}\la(\ul{G})^2 \ra \la \ul{F}\ra \la \ul{K}\ra= \mathbb{E}\la\ul{G}\ra^2  \la \ul{F}\ra \la \ul{K}\ra-\mathbb{E} \la\ul{G}\ra^2 \mathbb{E}\la\ul{F}\ra\la \ul{K}\ra+2\mathbb{E}\ul{G} \mathbb{E}\la \ul{G}\ra\la \ul{F}\ra \la \ul{K}\ra. \label{100902}
\end{align}
Plugging  (\ref{100901}) and (\ref{100902}) into (\ref{100903}), we have 
\begin{align}
	T_1^{-1} \mathbb{E}\la \ul{G}\ra \la \ul{F}\ra \la \ul{K}\ra&= \mathbb{E}\la\ul{G}\ra^2  \la \ul{F}\ra \la \ul{K}\ra-\mathbb{E} \la\ul{G}\ra^2 \mathbb{E}\la\ul{F}\ra\la \ul{K}\ra+\frac{1}{N}\mathbb{E} \la \ul{G^2} \ra \la \ul{F}\ra \la \ul{K}\ra\notag\\
	&\quad  +\frac{2}{N^2} \mathbb{E} \ul{GK^2}\la \ul{F}\ra+ \frac{2}{N^2}\mathbb{E}\ul{GF^2}\la \ul{K}\ra-W_2-W_3+O_\prec(N^{-\frac92})\notag\\
	&=: A_1+\cdots+A_5-W_2-W_3+O_\prec(N^{-\frac92}). \label{100910}
\end{align}
Then, we can proceed with a further estimate of the terms $A_i$'s, by similar calculations. For instance
\begin{align*}
	z_1A_1 &=\mathbb{E}\la\ul{HG}\ra \la\ul{G}\ra  \la \ul{F}\ra \la \ul{K}\ra=\frac{1}{N}\sum_{ij} \mathbb{E} H_{ji} G_{ij} \big\la \la\ul{G}\ra  \la \ul{F}\ra \la \ul{K}\ra\big\ra= W_{1,1}+W_{1,2}+W_{1,3}+O_\prec(N^{-\frac{11}{2}}),
\end{align*}
where 
\begin{align*}
	W_{1,k}= \frac{1}{k!N} \sum_{ij} \mathcal{C}_{k+1}(H_{ji}) \mathbb{E} \frac{\partial^k}{\partial H_{ji}^k} \Big( G_{ij}\big\la \la\ul{G}\ra\la \ul{F}\ra \la \ul{K}\ra\big\ra\Big) \,.
\end{align*}
Then similarly to (\ref{100910}), we can derive 
\begin{align*}
	T_1^{-1}A_1 &=
	\frac{2}{N^2}\mathbb{E} \ul{G^3} \la \ul{F}\ra\la \ul{K}\ra+\frac{2}{N^2} \mathbb{E} \ul{GF^2} \la\ul{G}\ra\la \ul{K}\ra +\frac{2}{N^2} \mathbb{E} \ul{GK^2}\la \ul{G}\ra\la \ul{F}\ra-W_{1,2}-W_{1,3}+O_\prec(N^{-5})\notag\\
	&\quad +\frac{1}{N}\mathbb{E}\la \ul{G^2}\ra \la \ul{G}\ra \la \ul{F}\ra \la \ul{K}\ra+\mathbb{E}\la \ul{G}\ra^3\la \ul{F}\ra\la \ul{K}\ra-\mathbb{E}\la\ul{G}\ra^2\mathbb{E}\la\ul{G}\ra\la \ul{F}\ra\la\ul{K}\ra\notag\\
	&=A_{1,1}+\cdots+A_{1,3}-W_{1,2}-W_{1,3}+O_\prec(N^{-5}).
\end{align*}
Applying Lemmas \ref{prop4.4}, we can easily get 
\begin{align*}
	A_{1,1} =\frac{2}{N^2}\mathbb{E} \ul{G^3}  \mathbb{E}\la \ul{F}\ra\la \ul{K}\ra+O_\prec(N^{-5})\,,\quad
	A_{1,2} =\frac{2}{N^2} \mathbb{E} \ul{GF^2} \mathbb{E}\la\ul{G}\ra\la \ul{K}\ra+O_\prec(N^{-5})\,, 
\end{align*}
and
\[
	A_{1,3}= \frac{2}{N^2} \mathbb{E} \ul{GK^2} \mathbb{E}\la \ul{G}\ra\la \ul{F}\ra+O_\prec(N^{-5})\,.
\]
Furthermore, we can derive  
\begin{align*}
	W_{1,2}
	=&\,\frac{b_3}{N^{5/2}} \sum_{i}  \mathbb{E}\bigg[  (G_{ii})^3\big\la \la\ul{G}\ra\la \ul{F}\ra \la \ul{K}\ra\big\ra+\frac{2}{N}(G_{ii})^2(G^2)_{ii}\la \ul{F}\ra\la \ul{K}\ra+ \frac{1}{N} (G_{ii})^2 (F^2)_{ii}\la \ul{G}\ra\la \ul{K}\ra\notag\\
	&+\frac{1}{N}(G_{ii})^2(K^2)_{ii}\la \ul{G}\ra\la\ul{F}\ra+\frac{1}{N^2} G_{ii} (G^2)_{ii} (F^2)_{ii}\la \ul{K}\ra+\frac{1}{N^2} G_{ii} (G^2)_{ii} (K^2)_{ii}\la \ul{F}\ra  \notag\\
	&+ \frac{1}{N^2} G_{ii} (F^2)_{ii} (K^2)_{ii}\la \ul{G}\ra + \frac{1}{N}G_{ii}F_{ii}(F^2)_{ii} \la \ul{G}\ra \la \ul{K}\ra+\frac{1}{N}G_{ii}K_{ii}(K^2)_{ii}\la \ul{G}\ra \la \ul{F}\ra\bigg]=O_\prec(N^{-9/2})
\end{align*}
and  
\begin{align*}
	W_{1,3}
	=&-\frac{b_4}{N^3} \sum_{i} \mathbb{E} \bigg[ (G_{ii})^4 \big\la \la \ul{G}\ra \la \ul{F}\ra \la \ul{K}\ra\big\ra+\frac{3}{N}(G_{ii})^3 (G^2)_{ii} \la \ul{F}\ra \la \ul{K}\ra+\frac{1}{N}G_{ii} (F_{ii})^2(F^2)_{ii} \la \ul{G}\ra \la \ul{K}\ra\notag\\
	&+\frac{1}{N}G_{ii}  (K_{ii})^2(K^2)_{ii}\la \ul{G}\ra \la \ul{F}\ra+\frac{1}{N}(G_{ii})^3(F^2)_{ii} \la\ul{G}\ra \la\ul{K}\ra+\frac{1}{N}(G_{ii})^3(K^2)_{ii} \la\ul{G}\ra\la \ul{F}\ra\notag\\
	&+\frac{1}{N}(G_{ii})^2F_{ii}(F^2)_{ii}\la\ul{G}\ra\la\ul{K}\ra+\frac{1}{N}(G_{ii})^2K_{ii}(K^2)_{ii} \la\ul{G}\ra\la\ul{F}\ra+\frac{2}{N^2}(G_{ii})^2(G^2)_{ii}(F^2)_{ii}\la\ul{K}\ra\notag\\
	&+\frac{2}{N^2}(G_{ii})^2(G^2)_{ii} (K^2)_{ii}\la\ul{F}\ra+\frac{1}{N^2}G_{ii}(G^2)_{ii}F_{ii}(F^2)_{ii} \la\ul{K}\ra+\frac{1}{N^2}G_{ii} (G^2)_{ii}  K_{ii}(K^2)_{ii}\la\ul{F}\ra\notag\\
	&+\frac{1}{N^2}G_{ii}(F^2)_{ii}K_{ii}(K^2)_{ii}\la\ul{G}\ra+\frac{1}{N^2}G_{ii} F_{ii}(F^2)_{ii}(K^2)_{ii}\la\ul{G}\ra+\frac{1}{N^2}(G_{ii})^2 (F^2)_{ii}(K^2)_{ii}\la\ul{G}\ra\notag\\
	& +\frac{1}{N^3}G_{ii} (G^2)_{ii}(F^2)_{ii}(K^2)_{ii}\bigg]=O_\prec(N^{-5})\,.
\end{align*}
To sum up, we have
\begin{align*}
	A_1=\frac{2}{N^2}T_1\mathbb{E} \ul{G^3}  \mathbb{E}\la \ul{F}\ra\la \ul{K}\ra+\frac{2}{N^2} T_1\mathbb{E} \ul{GF^2} \mathbb{E}\la\ul{G}\ra\la \ul{K}\ra+\frac{2}{N^2} T_1\mathbb{E} \ul{GK^2} \mathbb{E}\la \ul{G}\ra\la \ul{F} \ra +O_{\prec}(N^{-9/2})\,.
\end{align*}
Similarly, we have 
\begin{align*}
	&A_2=- \frac{2}{N^2} T_1\mathbb{E} \ul{G^3}  \mathbb{E}\la \ul{F}\ra\la \ul{K}\ra+O_\prec(N^{-5}) \quad \mbox{and} \quad A_3=O_\prec(N^{-5})\,.
\end{align*}
Finally, the resolvent identity shows
\begin{align*}
	&A_4=\frac{2}{N^2} \frac{1}{(z_1-z_3)^2}\mathbb{E} \la\ul{G} \ra\la \ul{F}\ra-\frac{2}{N^2}\frac{1}{(z_1-z_3)^2} \mathbb{E} \la \ul{K}\ra \la \ul{F}\ra-\frac{2}{N^2}\frac{1}{z_1-z_3} \mathbb{E} \la \ul{K^2}\ra \la \ul{F}\ra,\notag\\
	&A_5=\frac{2}{N^2} \frac{1}{(z_1-z_2)^2}\mathbb{E} \la\ul{G} \ra\la \ul{K}\ra-\frac{2}{N^2}\frac{1}{(z_1-z_2)^2} \mathbb{E} \la \ul{F}\ra \la \ul{K}\ra-\frac{2}{N^2}\frac{1}{z_1-z_2} \mathbb{E} \la \ul{F^2}\ra \la \ul{K}\ra.
\end{align*}
Combining the above estimates, we can get from  (\ref{100910}) that 
\begin{align}
	T_1^{-1} \mathbb{E}\la \ul{G}\ra \la \ul{F}\ra \la \ul{K}\ra
	=&\,\frac{2}{N^2}T_1 \mathbb{E} \ul{GF^2} \mathbb{E}\la\ul{G}\ra\la \ul{K}\ra+\frac{2}{N^2}T_1 \mathbb{E} \ul{GK^2} \mathbb{E}\la \ul{G}\ra\la \ul{F}\ra\notag\\
	&+\frac{2}{N^2} \frac{1}{(z_1-z_3)^2}\mathbb{E} \la\ul{G} \ra\la \ul{F}\ra-\frac{2}{N^2}\frac{1}{(z_1-z_3)^2} \mathbb{E} \la \ul{K}\ra \la \ul{F}\ra-\frac{2}{N^2}\frac{1}{z_1-z_3} \mathbb{E} \la \ul{K^2}\ra \la \ul{F}\ra\notag\\
	&+\frac{2}{N^2} \frac{1}{(z_1-z_2)^2}\mathbb{E} \la\ul{G} \ra\la \ul{K}\ra-\frac{2}{N^2}\frac{1}{(z_1-z_2)^2} \mathbb{E} \la \ul{F}\ra \la \ul{K}\ra-\frac{2}{N^2}\frac{1}{z_1-z_2} \mathbb{E} \la \ul{F^2}\ra \la \ul{K}\ra\notag\\
	&-W_2-W_3+O_\prec(N^{-\frac92}). \label{102301}
\end{align}

\subsection{The term $W_2$}
From the definition and  Assumption \ref{122910} (i), we have 
\begin{align*}
	W_2:= \frac{b_3}{2N^{5/2}}\sum_{i}   \mathbb{E} \frac{\partial^2}{\partial H_{ii}^2}\Big( G_{ii}\big\la \la \ul{F}\ra \la \ul{K}\ra\big\ra\Big). 
\end{align*}
By \eqref{101301}, we see that
\begin{align}
	\frac{\partial^2}{\partial H_{ii}^2}\Big( G_{ii}\big\la \la \ul{F}\ra \la \ul{K}\ra\big\ra\Big)= &\,2(G_{ii})^3 \big\la \la \ul{F}\ra \la \ul{K}\ra\big\ra +\frac{2}{N}(G_{ii})^2(F^2)_{ii}\la \ul{K}\ra+\frac{2}{N}(G_{ii})^2(K^2)_{ii}\la \ul{F}\ra \notag\\
	&+\frac{2}{N} G_{ii}(F_{ii})^3 \la \ul{K}\ra+\frac{2}{N} G_{ii} (K_{ii})^3\la \ul{F}\ra+\frac{2}{N^2}  G_{ii}(F^2)_{ii}(K^2)_{ii}\,. \label{101201}
\end{align}
We claim all but the last term on RHS of (\ref{101201}) are negligible. For instance, one can estimate the contribution from the second term in the RHS of (\ref{101201}) as 
\begin{align}
	&\frac{b_3}{N^{5/2}}  \mathbb{E} \Big\la \frac{1}{N}\sum_{i} (G_{ii})^2(F^2)_{ii}\Big\ra \la \ul{K}\ra\notag\\
	=&\,\frac{b_3}{N^{5/2}} \bigg( \mathbb{E}\Big\la \frac{1}{N}\sum_{i}(G_{ii}-m(z_1))^2 (F^2)_{ii}\Big\ra \la \ul{K}\ra+2m(z_1)  \mathbb{E} \Big\la \frac{1}{N}\sum_i (G_{ii}-m(z_1))((F^2)_{ii}-m'(z_2))\Big\ra \la \ul{K}\ra\notag\\
	&+2m(z_1)m'(z_2) \mathbb{E} \la\ul{G}\ra \la \ul{K}\ra+m_1^2\mathbb{E} \la \ul{F^2}\ra \la \ul{K}\ra\bigg)=O_\prec(N^{-\frac92}).  \label{121601}
\end{align}
Similarly, applying Lemma \ref{prop4.4} to all the other terms, we have 
\begin{align}
	W_2=\frac{b_3}{N^{9/2}}\sum_{i}     \mathbb{E} G_{ii} (F^2)_{ii}(K^2)_{ii}+O_\prec\big(N^{-\frac92}\big)=\frac{b_3}{N^{\frac72}}m(z_1)m'(z_2)m'(z_3)+O_\prec\big(N^{-\frac92}\big). 
	\label{101903}
\end{align}

\subsection{The term $W_3$} Recall the definition 
\begin{align*}
	W_3=\frac{b_4}{6N^3}\sum_{i}   \mathbb{E} \frac{\partial^3}{\partial H_{ii}^3}\Big( G_{ii}\big\la \la \ul{F}\ra \la \ul{K}\ra\big\ra\Big).
\end{align*}
By \eqref{101301}, we see that
\begin{align}
	&\frac{\partial^3}{\partial H_{ii}^3}\Big( G_{ii}\big\la \la \ul{F}\ra \la \ul{K}\ra\big\ra\Big)= -4(G_{ii})^4\big\la \la \ul{F}\ra \la \ul{K}\ra\big\ra-\frac{6}{N}(G_{ii})^2(F^2)_{ii}\la \ul{K}\ra-\frac{6}{N}(G_{ii})^2(K^2)_{ii}\la \ul{F}\ra \notag\\
	&\qquad-\frac{6}{N} (G_{ii})^2F_{ii}(F^2)_{ii}\la \ul{K}\ra-\frac{6}{N}(G_{ii})^2K_{ii}(K^2)_{ii}\la\ul{F}\ra-\frac{6}{N} G_{ii}(F_{ii})^2(F^2)_{ii}\la\ul{K}\ra-\frac{6}{N}G_{ii} (K_{ii})^2(K^2)_{ii}\la \ul{F}\ra \notag\\
	&\qquad-\frac{6}{N^2}(G_{ii})^2(F^2)_{ii}(K^2)_{ii}-\frac{6}{N^2} G_{ii} F_{ii}(F^2)_{ii}(K^2)_{ii}-\frac{6}{N^2} G_{ii} K_{ii}(K^2)_{ii}(F^2)_{ii}. \label{120401}
\end{align}
We claim that except for the last three  terms, all the other terms on the RHS of \eqref{120401} will have  negligible contribution to $W_3$. The estimate of these negligible terms is similar to (\ref{121601}), and thus we omit the details.  Finally, we can conclude
\begin{align}
	W_3=&-\frac{b_4}{N^5}\sum_i\Big((G_{ii})^2(F^2)_{ii}(K^2)_{ii}+G_{ii} F_{ii}(F^2)_{ii}(K^2)_{ii}+G_{ii} K_{ii}(K^2)_{ii}(F^2)_{ii} \Big)+O_\prec(N^{-5})\notag\\
	=&-\frac{b_4}{N^4}m(z_1)m'(z_2)m'(z_3)\big(m(z_1)+m(z_2)+m(z_3)\big)+O_\prec(N^{-5}).  \label{101904}
\end{align}

\subsection{Summing up}  Combining (\ref{102301}), (\ref{101903}) and (\ref{101904}), we arrive at 
\begin{align}
	\mathbb{E}\la \ul{G}\ra\la \ul{F}\ra \la \ul{K}\ra= &\,\frac{2}{N^2} T_1^2\mathbb{E} \ul{GF^2} \mathbb{E}\la\ul{G}\ra\la \ul{K}\ra+\frac{2}{N^2} T_1^2\mathbb{E} \ul{GK^2} \mathbb{E}\la \ul{G}\ra\la \ul{F}\ra\notag\\
	&+\frac{2}{N^2}T_1 \frac{1}{(z_1-z_3)^2}\mathbb{E} \la\ul{G} \ra\la \ul{F}\ra-\frac{2}{N^2}T_1\frac{1}{(z_1-z_3)^2} \mathbb{E} \la \ul{K}\ra \la \ul{F}\ra-\frac{2}{N^2}T_1\frac{1}{z_1-z_3} \mathbb{E} \la \ul{K^2}\ra \la \ul{F}\ra\notag\\
	&+\frac{2}{N^2} T_1\frac{1}{(z_1-z_2)^2}\mathbb{E} \la\ul{G} \ra\la \ul{K}\ra-\frac{2}{N^2} T_1\frac{1}{(z_1-z_2)^2} \mathbb{E} \la \ul{F}\ra \la \ul{K}\ra-\frac{2}{N^2} T_1\frac{1}{z_1-z_2} \mathbb{E} \la \ul{F^2}\ra \la \ul{K}\ra\notag\\
	&+\frac{b_3}{N^{7/2} } T_1m(z_1)m'(z_2)m'(z_3) +\frac{b_4}{N^4}T_1m(z_1)m'(z_2)m'(z_3)\big(m(z_1)+m(z_2)+m(z_3)\big)+O_\prec\big(N^{-\frac92}\big)\notag\\
	&=:D_2+D_3+D_4+O_\prec\big(N^{-\frac92}\big)\label{7.25}\,, 
\end{align}
where we denote the terms on first three lines on RHS of \eqref{7.25} by $D_2$. In order to simplify these terms, we need the following elementary identities
\begin{align*}
	m'(z_i)= \frac{z_i-\sqrt{z_i^2-4}}{2\sqrt{z_i^2-4}}, \quad  m''(z_i)=-\frac{2}{(z_i^2-4)^\frac{3}{2}}\,, 
\end{align*}
and
\begin{equation} \label{120501}
	T_i= -\frac{1}{\sqrt{z_i^2-4}}+O_\prec(N^{-1})=\frac{m'(z_i)}{m(z_i)}+O_\prec(N^{-1})\,. 
\end{equation}
Applying the above identities and the local law, we have 
\begin{align*}
	\mathbb{E}\ul{G^2F} &= \frac{\sqrt{z_1^2-4}\sqrt{z_2^2-4}-z_1z_2+4}{2\sqrt{z_1^2-4}(z_1-z_2)^2}+O_\prec(N^{-1}),
\end{align*}
and
\[
\mathbb{E} \ul{G^3F} =\frac{(z_1z_2-4)(z_1^2-4)-(z_1^2-4)^{\frac32}\sqrt{z_2^2-4}-2(z_1-z_2)^2}{2(z_1^2-4)^{\frac32}(z_1-z_2)^3}+O_\prec(N^{-1}).
\]
Plugging the above estimates into the the definition of $D_2$, we can get via a tedious but elementary calculation that 
\begin{align*}
	D_2= & \frac{8}{N^4}\big(h(z_1,z_2,z_3)+h(z_3,z_2,z_1)+h(z_2,z_1,z_3)\big)+O_\prec(N^{-1}).
\end{align*}
Further, by (\ref{120501}), it is easy to see
\begin{align*}
	D_3&= \frac{b_3}{N^{7/2} }m'(z_1)m'(z_2)m'(z_3)+O_\prec(N^{-1})\,,
\end{align*}
and
\[
D_4= \frac{b_4}{N^4}m'(z_1)m'(z_2)m'(z_3)\big(m(z_1)+m(z_2)+m(z_3)\big)+O_\prec(N^{-1})\,.
\]
Inserting the above three relations into \eqref{7.25} we conclude the proof of  Proposition \ref{proGFK}.

\vspace{2em}

\appendix 

\section{Proof of Proposition \ref{lem 2148}} \label{sec6}
By the resolvent identity, we have
\begin{equation} \label{7.1}
	\tr G=\sum_{k=0}^3(-1)^{k}\tr \widehat{G}(H_{\mathrm d}\widehat{G})^k-\tr \widehat{G}H_{\mathrm d}G(H_{\mathrm d}\widehat{G})^3\eqd \sum_{k=0}^3E_k\,.
\end{equation}
Let us check each term on RHS of \eqref{7.1}.   We shall repeatedly use the formula
\begin{equation} \label{7.2}
	AB-\bb E AB=A(B-\bb E B)+(A-\bb EA) \bb E B- \bb E (A-\bb EA)(B-\bb E B)\,.
\end{equation}

For brevity, we will focus on $z\in \b S_c^+$ (c.f. (\ref{sc+})). The case of $z\in \b S_c\setminus \b S_c^+$ is almost the same. 

\subsection{The estimate of $\langle E_1 \rangle $.}
By \eqref{7.2}, we have
\begin{multline} \label{7.3}
	\langle \tr \widehat{G}H_{\mathrm d}\widehat{G}\rangle=\Big\langle\sum_i (\widehat{G}^2)_{ii}H_{ii}\Big\rangle=\sum_i (\widehat{G}^2)_{ii}H_{ii}\\
	=\sum_i \langle(\widehat{G}^2)_{ii}\rangle H_{ii}+\sum_i (\bb E(\widehat{G}^2)_{ii}-\bb E \ul{\widehat{G}^2})H_{ii}+(\bb E \ul{\widehat{G}^2}-m' )\sum_i H_{ii}+m' \sum_i H_{ii}\,.
\end{multline}
Note that
$
\widehat{G}_{ii}=\ul{\widehat{G}}+\ul{\widehat{H}\widehat{G}}\widehat{G}_{ii}-(\widehat{H}\widehat{G})_{ii}\ul{\widehat{G}},
$
thus 
\begin{equation*}
	\begin{aligned}
		&\quad \bb E (\widehat{G}^2)_{ii}-\bb E \ul{\widehat{G}^2}=\partial_z (\bb E \widehat{G}_{ii}-\bb E \ul{\widehat{G}})=\partial_z\big(\bb E  \ul{\widehat{H}\widehat{G}}\widehat{G}_{ii}-\bb E(\widehat{H}\widehat{G})_{ii}\ul{\widehat{G}}\big)\\
		&=\partial_z\bigg(\frac{1}{N} {\sum_{k,j}}^*\bb E H_{jk}\widehat{G}_{kj}\widehat{G}_{ii}-\frac{1}{N}\sum_{j,k: j \ne i} \bb E H_{ij}\widehat{G}_{ji}\widehat{G}_{kk}\bigg)\,.
	\end{aligned}
\end{equation*}
Expanding the RHS of the above using Lemma \ref{lem:cumulant_expansion}, we can show that
\begin{equation} \label{7.4}
	\bb E (\widehat{G}^2)_{ii}-\bb E \ul{\widehat{G}^2} =O_{\prec}\Big(\frac{1}{N\eta^2}\Big)\,.
\end{equation}
In addition, Theorem \ref{refthm1} implies
\begin{equation} \label{7.5}
	\bb E \ul{\widehat{G}^2}-m'=O_{\prec}\Big(\frac{1}{N\eta^2}\Big)\,.
\end{equation}
Inserting \eqref{7.4} and \eqref{7.5} into \eqref{7.3}, we have
\begin{equation} \label{7.6}
	\langle E_1 \rangle =-\langle \tr \widehat{G}H_{\mathrm d}\widehat{G}\rangle=-\sum_i \langle(\widehat{G}^2)_{ii}\rangle H_{ii}-m' \tr H+O_{\prec}\Big(\frac{1}{N\eta^2}\Big)\,.
\end{equation}

\subsection{The estimate of $\langle E_2\rangle$}
We have
\begin{equation*} 
	E_2=	\tr \widehat{G}H_{\mathrm d}\widehat{G}H_{\mathrm d}\widehat{G}=\sum_{i,j} (\widehat{G}^2)_{ij}\widehat{G}_{ij}H_{ii}H_{jj}	
	=\sum_{i} (\widehat{G}^2)_{ii}\widehat{G}_{ii}H_{ii}^2+{\sum_{i,j}}^* (\widehat{G}^2)_{ij}\widehat{G}_{ij}H_{ii}H_{jj}\eqd E_{2,1}+E_{2,2}\,.
\end{equation*}
By Theorem \ref{refthm1} we see that 
\[
(\widehat{G}^2)_{ij} \prec \frac{1}{\sqrt{N\eta^3}}\,,\quad \widehat{G}_{ij}\prec \frac{1}{\sqrt{N\eta}}\,. 
\]
Together with the fact that $(\widehat{G}^2)_{ij} $, $\widehat{G}_{ij}$ are independent from $H_{\mathrm d}$. We can easily get that 
\begin{equation} \label{7.7}
	E_{2,2}\prec \frac{1}{N\eta^2}\,.
\end{equation}
By \eqref{7.2} we have
\begin{equation*}
	\begin{aligned}
		\langle E_{2,1} \rangle &=\sum_i \langle (\widehat{G}^2)_{ii}\widehat{G}_{ii} \rangle H^2_{ii}+\sum_{i} (H_{ii}^2-a_2N^{-1})\bb E(\widehat{G}^2)_{ii}\widehat{G}_{ii}\\
		&=a_2N^{-1}\sum_{i} \langle (\widehat{G}^2)_{ii}\widehat{G}_{ii} \rangle+ \sum_i \langle (\widehat{G}^2)_{ii}\widehat{G}_{ii} \rangle (H^2_{ii}-a_2N^{-1})+\sum_{i} (H_{ii}^2-a_2N^{-1})m'(z)m(z)\\
		&\quad + \sum_{i} (H_{ii}^2-a_2N^{-1})(\bb E(\widehat{G}^2)_{ii}\widehat{G}_{ii} -m'(z)m(z))\,.
	\end{aligned}
\end{equation*}
By Theorem \ref{refthm1} and Lemma \ref{prop4.4}, it is easy to check that
\[
\sum_i \langle (\widehat{G}^2)_{ii}\widehat{G}_{ii} \rangle (H^2_{ii}-a_2N^{-1}) \prec \frac{1}{N\eta^{3/2}}\quad \mbox{and} \quad \sum_{i} (H_{ii}^2-a_2N^{-1})(\bb E(\widehat{G}^2)_{ii}\widehat{G}_{ii} -m'(z)m(z)) \prec \frac{1}{N\eta^{3/2}}\,.
\]
In addition, Lemma \ref{prop4.4} shows
\[
a_2N^{-1}\sum_{i} \langle (\widehat{G}^2)_{ii}\widehat{G}_{ii} \rangle \prec \frac{1}{N\eta^2}\,.
\]
Thus we have
\[
\langle E_{2,1} \rangle=\sum_{i} (H_{ii}^2-a_2N^{-1})m'(z)m(z)+O_\prec \Big(\frac{1}{N\eta^2}\Big)\,,
\]
and together with \eqref{7.7} we get
\begin{equation} \label{7.8}
	\langle E_{2} \rangle=\sum_{i} (H_{ii}^2-a_2N^{-1})m'(z)m(z)+O_\prec \Big(\frac{1}{N\eta^2}\Big)\,.
\end{equation}

\subsection{The estimate of $\langle E_3 \rangle$} We have
\begin{multline*}
	-E_3=\sum_{i,j,k} (\widehat{G}^2)_{ij} \widehat{G}_{jk}\widehat{G}_{ki}H_{ii}H_{jj}H_{kk}={\sum_{i,j,k}}^* (\widehat{G}^2)_{ij} \widehat{G}_{jk}\widehat{G}_{ki}H_{ii}H_{jj}H_{kk}+2{\sum_{i,j}}^* (\widehat{G}^2)_{ij} \widehat{G}_{ji}\widehat{G}_{ii}H_{ii}^2H_{jj}\\
	+{\sum_{i,k}}^* (\widehat{G}^2)_{ii} \widehat{G}_{ik}\widehat{G}_{ki}H_{ii}^2H_{kk}+\sum_i (\widehat{G}^2)_{ii} \widehat{G}_{ii}^2H_{ii}^3\eqd E_{3,1}+\cdots+E_{3,4}\,.
\end{multline*}
By Theorem\ref{refthm1}, it is easy to check that $E_{3,1} \prec N^{-1}\eta^{-5/2}$. We also have
\[
\sum_{j: j \ne i}  (\widehat{G}^2)_{ij} \widehat{G}_{ji}\widehat{G}_{ii}H_{jj} \prec \frac{1}{N\eta^2}\,,
\]
and together with $H_{ii}^2\prec N ^{-1}$ we have $E_{3,2} \prec N^{-1}\eta^{-2}$. Similarly, $E_{3,3} \prec N^{-1}\eta^{-2}$. By \eqref{7.2}, we have
\[
\langle E_{3,4} \rangle =a_3N^{-3/2}\sum_i \langle (\widehat{G}^2)_{ii}\widehat{G}^2_{ii} \rangle +\sum_{i} (H_{ii}^3-a_3N^{-3/2})(\widehat{G}^2)_{ii}\widehat{G}^2_{ii} \prec \frac{1}{N\eta^2}\,.
\]
Hence we have
\begin{equation} \label{7.9}
	\langle E_3 \rangle \prec \frac{1}{N\eta^{5/2}}\,.
\end{equation}

\subsection{The estimate of $E_4$}
By Theorem \ref{refthm1}, we have
\[
E_4=\sum_{i,j} (\widehat{G}^2H_{\mathrm d}\widehat{G}H_{\mathrm d}\widehat{G})_{ij}(H_{\mathrm d}{G}H_{\mathrm d})_{ji} \prec N^{1/2}\eta^{-1/2} \max_{i,j:i \ne j} |(\widehat{G}^2H_{\mathrm d}\widehat{G}H_{\mathrm d}\widehat{G})_{ij}|+\max_i|(\widehat{G}^2H_{\mathrm d}\widehat{G}H_{\mathrm d}\widehat{G})_{ii}|\,.
\]
For $i \ne j$, we have
\begin{multline} \label{7.10}
	(\widehat{G}^2H_{\mathrm d}\widehat{G}H_{\mathrm d}\widehat{G})_{ij}=\sum_{k,l} (\widehat{G}^2)_{ik}H_{kk}\widehat{G}_{kl}H_{ll}\widehat{G}_{lj}={\sum_{k,l}}^*(\widehat{G}^2)_{ik}H_{kk}\widehat{G}_{kl}H_{ll}\widehat{G}_{lj}+\sum_{k} (\widehat{G}^2)_{ik}H_{kk}\widehat{G}_{kk}H_{kk}\widehat{G}_{kj}\\
	=\sum_{k:k \ne l}\sum_{l:\l \ne j}(\widehat{G}^2)_{ik}H_{kk}\widehat{G}_{kl}H_{ll}\widehat{G}_{lj}+\sum_{k:k \ne j}  (\widehat{G}^2)_{ik}H_{kk}\widehat{G}_{kj}H_{jj}\widehat{G}_{jj}+\sum_{k} (\widehat{G}^2)_{ik}(H_{kk}^2-a_2N^{-1}) \widehat{G}_{kk}\widehat{G}_{kj}\\
	+a_2N^{-1}\sum_{k} (\widehat{G}^2)_{ik} (\widehat{G}_{kk}-m(z))\widehat{G}_{kj}+a_2N^{-1}m(z)(\widehat{G}^3)_{ij}\,,
\end{multline}
and one can check that each term on RHS of \eqref{7.10} is bounded by $O_{\prec}(1/(N^{3/2}\eta^{5/2}))$. Thus
\[
\max_{i\ne j }\big|(\widehat{G}^2H_{\mathrm d}\widehat{G}H_{\mathrm d}\widehat{G})_{ij}\big| =O_{\prec}\Big(\frac{1}{N^{3/2}\eta^{5/2}}\Big)\,.
\]
Similarly, we can also show that
\[
\max_{i }\big|(\widehat{G}^2H_{\mathrm d}\widehat{G}H_{\mathrm d}\widehat{G})_{ii}\big| =O_{\prec}\Big(\frac{1}{N\eta^{2}}\Big)\,.
\]
Thus we have
\begin{equation} \label{7.11}
	E_4 =O_{\prec}\Big(\frac{1}{N\eta^3}\Big)\,.
\end{equation}

\subsection{Conclusion}

Combining \eqref{7.1}, \eqref{7.6}, \eqref{7.8}, \eqref{7.9} and \eqref{7.11}, we have
\[
\langle \tr G \rangle= \langle \tr \widehat{G} \rangle -\sum_i \langle(\widehat{G}^2)_{ii}\rangle H_{ii}-m' \tr H+\sum_{i} (H_{ii}^2-a_2N^{-1})m'(z)m(z)+O_{\prec}\Big(\frac{1}{N\eta^3}\Big)
\]
as desired. This finishes the proof.

\section{Proof of Lemma \ref{lem:exp tr f(H)}} \label{appA}
In this section, we estimate $\mathbb{E} \text{Tr} f(H)$. A key technical result is the following expansion of  $\mathbb{E} \ul{G}$.  

\begin{lem} \label{lem.EG}
Suppose the assumptions in Definition \ref{def:Wigner} hold, and recall the definition of $\b S_c^+$ in \eqref{sc+}. For $z\in \mathbf{S}_c^+$, we have 
\begin{align*}
\mathbb{E} \ul{G}=m-\frac{m'}{m}\Big( -\frac{1}{N} m'-\frac{a_2-2}{N} m^2-\frac{s_4}{N} m^4+\frac{a_3}{N^{3/2}}m^3\Big)+O_\prec\Big(\frac{1}{N^2\eta^{2}}\Big)\,.
\end{align*}
\end{lem}

\begin{proof}
By resolvent identity and Lemma \ref{lem:cumulant_expansion}, we have
\begin{align}
1+z\mathbb{E} \ul{G}= \mathbb{E} \ul{HG}=\frac{1}{N}\sum_{i,j} \mathbb{E} H_{ji} G_{ij}=\wt{W}_1+\wt{W}_2+\wt{W}_3++O_\prec\Big(\frac{1}{N^2\eta^{\frac12}}\Big),  \label{122320}
\end{align}
where
\begin{align*}
\wt{W}_k\deq\frac{1}{N}\sum_{i,j} \frac{1}{k!} \mathcal{C}_{k+1}(H_{ji}) \mathbb{E} \frac{\partial^k}{\partial H_{ji}^k} G_{ij}\,,
\end{align*}
 we used a routine estimate to bound the remainder term by $O_{\prec}(N^{-2}\eta^{-1/2})$. Applying (\ref{101301}) repeatedly, we have 
\begin{align*}
\wt{W}_1= -\mathbb{E}(\ul{G})^2-\frac{1}{N}\mathbb{E}\ul{G^2}-\frac{a_2-2}{N^2}\sum_{i}\mathbb{E}(G_{ii})^2,\quad \wt{W}_2= \frac{1}{2N}\sum_{ij}   \frac{\mathcal{C}_{3}(H_{ji})}{(1+\delta_{ij})^2} \mathbb{E}\big(6G_{ii}G_{jj}G_{ij}+2(G_{ij})^3\big),\notag
\end{align*}
and
\begin{equation*}
	\wt{W}_3= -\frac{1}{6N} \sum_{ij} \frac{\mathcal{C}_4(H_{ji})}{(1+\delta_{ij})^3} \big(36 G_{ii}G_{jj}(G_{ij})^2+6(G_{ii})^2(G_{jj})^2+6(G_{ij})^4\big)\,.
\end{equation*}
It is easy to see from Theorem \ref{refthm1} that

\[
\wt{W}_1= -(\mathbb{E}\ul{G})^2-\frac{1}{N}m'-\frac{a_2-2}{N}m^2+O_\prec\Big(\frac{1}{N^2\eta}\Big),\quad  \mbox{and} \quad \wt{W}_3= -\frac{s_4}{N}m^4+O_\prec\Big(\frac{1}{N^2\eta}\Big)\,.
\]
In addition, we have
\[
\wt{W_2}=\frac{a_3}{N^{5/2}}\sum_{i}\bb E G_{ii}^3+\frac{s_3}{N^{5/2}}{\sum_{i,j}}^* \bb E (3G_{ii}G_{jj}G_{ij}+(G_{ij})^3\big)=\frac{a_3}{N^{3/2}}m^3+\frac{3s_3}{N^{5/2}}{\sum_{i,j}}^* \bb E G_{ii}G_{jj}G_{ij}+O_{\prec}\Big(\frac{1}{N^2\eta^{3/2}}\Big)\,.
\]
Note that by the isotropic law Theorem \ref{refthm1},
\begin{align*}
{\sum_{i,j}}^*
\bb E G_{ii}G_{jj}G_{ij} &= {\sum_{i,j}}^* \bb E (G_{ii}-m)(G_{jj}-m)G_{ij}+m{\sum_{i,j}}^* \bb E (G_{jj}-m)G_{ij}+m{\sum_{i,j}}^* \bb E (G_{ii}-m)G_{ij}+m^2{\sum_{i,j}}^*\mathbb{E} G_{ij}
\notag\\
&=O_{\prec}(N^{1/2}\eta^{-1}) \,,
\end{align*} 
and thus
\[
\wt{W}_2=\frac{a_3}{N^{3/2}}m^3+O_{\prec}\Big(\frac{1}{N^2\eta^{3/2}}\Big)
\]
Plugging the computations of $\wt{W}_1$, $\wt{W}_2$ and $\wt{W}_3$ into (\ref{122320}) , we have 
\begin{align*}
1+z\mathbb{E} \ul{G}+(\mathbb{E}\ul{G})^2=-\frac{1}{N} m'-\frac{a_2-2}{N} m^2-\frac{s_4}{N} m^4+\frac{a_3}{N^{3/2}}m^3+O_\prec\Big(\frac{1}{N^2\eta^{3/2}}\Big)\,.
\end{align*}
Solving the equation, we can get 
\begin{align*}
\mathbb{E}\ul{G}-m= &\frac{1}{z+2m} \Big( -\frac{1}{N} m'-\frac{a_2-2}{N} m^2-\frac{s_4}{N} m^4+\frac{a_3}{N^{3/2}}m^3\Big)+O_\prec\Big(\frac{1}{N^2\eta^{2}}\Big) \notag\\
=& -\frac{m'}{m}\Big( -\frac{1}{N} m'-\frac{a_2-2}{N} m^2-\frac{s_4}{N} m^4+\frac{a_3}{N^{3/2}}m^3\Big)+O_\prec\Big(\frac{1}{N^2\eta^{2}}\Big)
\end{align*}
as desired.
\end{proof}
With the aid of Lemma \ref{lem.EG}, Lemma \ref{lem:exp tr f(H)} follows from a standard use of Helffer-Sj\"{o}strand formula Lemma \ref{HS}.

\vspace{2em}

\section{Proof of Lemma \ref{lem1558}} \label{app B}

In the sequel, for brevity, we set 
\begin{align*}
\phi_d(t):=\varphi_2(t)\varphi_3(t). 
\end{align*}

\textit{Case 1.} We first consider the case 
\begin{align}
(1-\gamma)c_1^f\neq 0.  \label{012101}
\end{align}
Further, we set 
\begin{align}
\wh{f}_i:= -\frac{1}{\pi}\int_{\mathbf{D}} \frac{\partial}{\partial \bar{z}} \tilde{f}(z)\la (\wh{G}^2)_{ii}\ra {\rm d}^2 z+\frac12(1-\gamma)c_1^f, \label{f_i}
\end{align}
and thus 
\begin{align*}
\sigma_{f,\gamma}(\wh{Z}_{f,\gamma,2}+\wh{Z}_{f,\gamma,3})= \sum_i \wh{f}_{i}H_{ii}+\frac12c_2^f\sum_i\Big(H_{ii}^2-\frac{a_2}{N}\Big). 
\end{align*}
By Lemma \ref{lem:cumulant_expansion}, we have
\begin{align}
\mathbb{E}_{d}\phi_d'(t)= \ii\sum_i \mathbb{E}_d H_{ii} \Big(\wh{f}_i +\frac{c_2^f}{2}H_{ii}\Big) \phi_d(t)- \frac{\ii}{2}c_2^fa_2 \mathbb{E}_d\phi_d(t) = \sum_{k=1}^{\ell} L_{d,k}+O_{\prec}\Big(\frac{t+1}{N}\Big)
\label{122901}
\end{align}
for some fixed $\ell \in \bb N_+$, where 
\begin{align*}
L_{d,k}=\ii \sum_i \frac{\mathcal{C}_{k+1}(H_{ii})}{k!} \mathbb{E}_d \frac{\partial^k}{\partial H_{ii}^k} \Big(\big(\wh{f}_i +\frac{c_2^f}{2}H_{ii}\big) \phi_d(t)\Big)-\delta_{k1}\frac{\ii}{2} c_2^fa_2\mathbb{E}_d\phi_d(t).
\end{align*}
It is straightforward to compute
\[
L_{d,1}= -t \frac{a_2}{N} \sum_i \mathbb{E}_d(\wh{f}_i+\frac{c_2^f}{2} H_{ii})(\wh{f}_i+c_2^fH_{ii}) \phi_d(t)\,,
\]
and
\begin{align*}
 & L_{d,2}= \ii\sum_i \frac{\mathcal{C}_3(H_{ii})}{2} \mathbb{E}_d \Big(\ii t\big(2c_2^f \wh{f}_i+\frac32 (c_2^f)^2H_{ii}\big) \notag\\
&\qquad\qquad+(\ii t)^2\big((\wh{f}_i)^3+\frac52 c_2^f (\wh{f}_i)^2 H_{ii}+2(c_2^f)^2 \wh{f}_iH_{ii}^2+\frac12(c_2^f)^3 H_{ii}^3\big)\Big)\phi_d(t)\,,
\end{align*}
as well as
\begin{align*}
	&L_{d,3}=\ii\sum_i \frac{\mathcal{C}_4(H_{ii})}{6} \mathbb{E}_d \bigg(\ii t\frac{3}{2}(c_2^f)^2+(\ii t)^2\Big(\frac{9}{2}c_2^f(\hat{f}_i)^2+\frac{15}{2}(c_2^f)^2\hat{f}_iH_{ii}+3(c_2^f)^3 H_{ii}^2\Big)\notag\\
	&\qquad\qquad+(\ii t)^3\Big((\hat{f}_i)^3+\frac{7}{2}c_2^f(\hat{f}_i)^3H_{ii}+\frac{9}{2}(c_2^f)^2(\hat{f}_i)^2H_{ii}^2+\frac{5}{2}\hat{f}_i(c_2^f)^3H_{ii}^3+\frac{1}{2}(c_2^f)^4H_{ii}^4\Big)\bigg)\phi_d(t)\,.
\end{align*}
According to the definition in (\ref{f_i}) and Theorem \ref{refthm1},  it is easy to check
\begin{align*}
&L_{d,1}=-a_2\frac{t}{4}\big((1-\gamma)c_1^f\big)^2\mathbb{E}_d\phi_d(t)+O_\prec\Big(\frac{t}{N}\Big). 
\end{align*}
Similarly, we have 
\begin{align*}
L_{d,2}=\ii\bigg[\Big(-\frac{t^2}{8}\big((1-\gamma)c_1^f\big)^3 +\ii t (1-\gamma)c_1^f c_2^f\Big)N\mathcal{C}_3(H_{11})+O_\prec\Big(\frac{t^2}{N}\Big)\bigg]\mathbb{E}_d\phi_d(t)+O_\prec\Big(\frac{t+1}{N}\Big)
\end{align*}
and 
\begin{align*}
L_{d,3}= O_\prec\Big(\frac{t^3}{N}\Big)\mathbb{E}_d \phi_d(t)+O_\prec\Big(\frac{t}{N}\Big)
\end{align*}
for $t\in [0, N^{(1-c)/2}]$.  One can similarly show that for any fixed $k\geq 4$, 
\begin{align*}
L_{d,k}= O_\prec\Big(\frac{t^k}{N^{\frac{k-1}{2}}}\Big)\mathbb{E}_d \phi_d(t)+O_\prec\Big(\frac{t^k}{N^{\frac{k+1}{2}}}\Big)=O_{\prec}\Big(\frac{t^3}{N}\Big)\mathbb{E}_d \phi_d(t)+O_\prec\Big(\frac{t}{N}\Big)
\end{align*}
for $t\in [0, N^{(1-c)/2}]$. Plugging the above estimates into (\ref{122901}), we get
\begin{align}
\mathbb{E}_d \phi_d'(t)= &\Big(-\frac14 a_2  \big((1-\gamma)c_1^f\big)^2 t +\wh{b}(t)\Big)\mathbb{E}_d\phi_d(t)+\wh{\mathcal{E}}(t), \label{123001}
\end{align}
where 
\begin{align*}
\wh{b}(t):=\ii \Big(-\frac{t^2}{8}\big((1-\gamma)c_1^f\big)^3 +\ii t (1-\gamma)c_1^f c_2^f\Big)N\mathcal{C}_3(H_{11})+O_\prec(\frac{t^3}{N})=O\Big(\mathcal{X} \frac{t^2+1}{\sqrt{N}}\Big)+O_\prec\Big(\frac{t^3}{N}\Big)\,,
\end{align*}
and
$
\wh{\mathcal{E}}(t)=O_{\prec}\big(\frac{t+1}{N}\big).
$
Further we denote by $\wt{B}(t):= \int_0^t \wh{b}(t) {\rm d} t$, and it is easy see
\begin{align*}
\wt{B}(t)-\wt{B}(s)=O(\mathcal{X}(t-s)(t^2+1) N^{-\frac12})+O_\prec((t-s)t^3N^{-1}), \quad 0\leq s\leq t\leq N^{(1-c)/2}.
\end{align*}
Solving the equation (\ref{123001}), we get 
\begin{align}
\mathbb{E}_d\phi_d(t)= &\exp\Big(-\frac18 a_2  \big((1-\gamma)c_1^f\big)^2 t^2+\wh{B}(t)\Big)\notag\\
 &+\int_{0}^t \exp\Big(-\frac18 a_2  \big((1-\gamma)c_1^f\big)^2 (t^2-s^2)+\wh{B}(t)-\wh{B}(s)\Big)\wh{\mathcal{E}}(s) {\rm d} s \label{012102}
\end{align}
for $0\leq t\leq N^{(1-c)/2}$. By the assumption (\ref{012101}), we see that 
\begin{align*}
\exp\Big(-\frac18 a_2  \big((1-\gamma)c_1^f\big)^2 t^2+\wh{B}(t)\Big)=\exp\Big(-\frac18 a_2  \big((1-\gamma)c_1^f\big)^2 t^2\Big)+O_\prec({\cal X}N^{-\frac12})+O_\prec(N^{-1})
\end{align*}
and
\[
\exp\Big(-\frac18 a_2  \big((1-\gamma)c_1^f\big)^2 (t^2-s^2)+\wh{B}(t)-\wh{B}(s)\Big)\leq C\exp\Big(-\frac{1}{16} a_2  \big((1-\gamma)c_1^f\big)^2 (t^2-s^2)\Big)
\]
for $0\leq s\leq t\leq N^{(1-c)/2}$. Plugging the above estimates to (\ref{012102}), we get 
\begin{align*}
\mathbb{E}_d\phi_d(t)=&\exp\Big(-\frac18 a_2  \big((1-\gamma)c_1^f\big)^2 t^2\Big)+O_\prec({\cal X}N^{-\frac12})+O_\prec(N^{-1}) \notag\\
&+ O_\prec(N^{-1}) \int_0^t \exp\Big(-\frac{1}{16} a_2  \big((1-\gamma)c_1^f\big)^2 (t^2-s^2)\Big)
(s+1) {\rm d} s.
\end{align*}
Similarly to (\ref{012106}) and (\ref{012105}), by discussing the case $t\in [0, \log N]$ and $t\in (\log N, N^{(1-c)/2}]$ separately, one can easily conclude
\begin{align*}
\mathbb{E}_d\phi_d(t)=\exp\Big(-\frac18 a_2  \big((1-\gamma)c_1^f\big)^2 t^2\Big)+O_\prec({\cal X}N^{-\frac12})+O_\prec(N^{-1}). 
\end{align*}
Further, under the assumption \eqref{012101}, we have 
\begin{align*}
\exp\Big(-\frac18 a_2  \big((1-\gamma)c_1^f\big)^2 t^2\Big)\varphi_4(t)=\exp\Big(-\frac18 a_2  \big((1-\gamma)c_1^f\big)^2 t^2\Big)+O_\prec(N^{-1}),
\end{align*}
by observing $X_{i}=O_\prec(N^{-1})$ (c.f. (\ref{def of Xi})).
This concludes the proof of (\ref{012110}) for the case $(1-\gamma)c_1^f\neq 0$.

\textit{Case 2.} Let us consider the case
\begin{align*}
(1-\gamma)c_1^f=0. 
\end{align*}
We have
\begin{align*}
\mathbb{E}_d\phi_d(t)=\mathbb{E}_d \exp\Big(\ii t \sum_i X_{i}\Big).
\end{align*}
For $t\in[0, N^{(1-c)/2}]$, by a simple Taylor expansion and the fact $X_{i}=O_\prec(N^{-1})$, we have 
\begin{align*}
\mathbb{E}_d\phi_d(t)=&\,\prod_{i=1}^N \Big(1-\frac{t^2}{2} \mathbb{E}_d X_{ii}^2+\frac{(\ii t)^3}{6} \mathbb{E}_dX_{ii}^3+O_\prec(t^4N^{-4})\Big)\notag\\
=& \,\exp\Big(-\frac{t^2}{2}\sum_i \mathbb{E}_d X_{ii}^2+\frac{(\ii t)^3}{6}\sum_{i}\mathbb{E}_d X_{ii}^3\Big)+O_\prec(N^{-1})\notag\\
=&\,\varphi_4(t)+O_\prec(N^{-1}). 
\end{align*}
This completes the proof of Lemma \ref{lem1558}. 

\vspace{2em}

\section{Fluctuation averaging of the Green functions} \label{appD}
In this section we prove some technical estimates on Green functions that we used in this paper.

\subsection{Proof of \eqref{2.12}} \label{sec:9.1}
Recall that we have the assumption $|\eta_1| \leq |\eta_2|\leq \cdots \leq |\eta_l|$. Let us prove
	\[
	\tr\big( G_1^{k_1}\cdots G_l^{k_l}\big)-Nm_{l}((z_1,k_1),...,(z_l,k_l)) \prec  \frac{1}{\big|\eta_1^{k_1} \cdots \eta_l^{k_l}\big|}\,,
	\]
	and the desired result follows from triangle inequality and Lemma \ref{prop_prec}.
	
	\textit{Case 1}. Suppose $2|\eta_{j}| \geq |\eta_{j+1}|$ for $j=1,...,l-1$. Then $|\eta_l|\leq 2^{l-1}|\eta_1|\leq 2^{l-1}|\eta_l|$. Then by Lemma \ref{HS}, we have
	\[
	\tr\big( G_1^{k_1}\cdots G_l^{k_l}\big)-Nm_{l}((z_1,k_1),...,(z_l,k_l))=\frac{1}{\pi} \int_{\bb C } \partial_{\bar{z}}(\tilde{f}(z)\chi(z/|\eta_1|)) (\tr G(z)-Nm(z)) \dd^2z
	\]
	where
	\[
	f(x)=\frac{1}{(x-z_1)^{k_1}\cdots (x-z_l)^{k_l}}\,,
	\]
	$\widetilde{f}$ is defined as in \eqref{def of f tilde}, and $\chi \in \cal C^{\infty}_c(\R)$ is a cutoff function satisfying $\chi(0) = 1$ and $\chi(y)=0$ for $|y|\geq 1$. Using Theorem \ref{refthm1}, the desired estimate then follows from the steps in \cite[Lemma 4.4]{HK}.
	
	\textit{Case 2.} Suppose we have integers $0=i_1<i_2<\cdots <i_{n-1}<i_n=l$ such that for each $q \in \{1,...,n-1\}$, we have $2|\eta_{j}| \geq |\eta_{j+1}|$ for $j=i_q+1,...,i_{q+1}-1$ and $2|\eta_{i_{q+1}}|<|\eta_{i_{q+1}+1}|$. We have the disjoint union
	\[
	\{1,...,l\}=\cup_{q=1}^{n-1} \{i_q+1,...,i_{q+1}\} \eqd \cup_{q=1}^{n-1}I_q\,.
	\]
	For $z_j$ and $z_{j'}$, we say that $z_j\sim z_{j'}$ if there exists $q$ such that $j,j' \in I_q$; otherwise $z_j\not\sim z_{j'}$. It is easy to check that this is an equivalent relation. Then we can see from Case 1 that
\begin{equation} \label{C2}
	\tr\big( G_{j_1}^{k_1}\cdots G_{j_p}^{k_p}\big)-Nm_{p}((z_{j_1},k_1),...,(z_{j_p},k_p)) \prec  \frac{1}{\big|\eta_{j_1}^{k_1} \cdots \eta_{j_p}^{k_p}\big|}
\end{equation}
	whenever $z_{j_1}\sim z_{j_2}\sim  \cdots \sim z_{j_p}$. Note that for $z_j \not\sim z_{j'}$, we have the resolvent identity
	$
	G_{j}G_{j'}=\frac{G_j-G_{j'}}{z_j-z_{j'}},
	$
	and 
\begin{equation} \label{C1}
	\frac{1}{z_j-z_{j'}} \prec \frac{1}{|\eta_j|+|\eta_{j'}|}\,.
\end{equation}
Repeatedly using the resolvent identity, \eqref{C2} and \eqref{C1}, we get the desired result.
	
\subsection{Proof of \eqref{2.13}} \label{sec:9.2}
From \cite[Lemma 4.4]{HK}, we see that
\begin{equation}\label{C3}
	(G_{j}^{k_j})_{ii}-m^{k_j-1}(z_j) \prec \frac{1}{\sqrt{N|\eta_{j}|}} \frac{1}{|\eta_j|^{k_j-1}}\quad  \mbox{and} \quad
	\tr (G_{j}^{k_j})-Nm(z_j) \prec \frac{1}{|\eta_j|^{k_j-1}}\,.
\end{equation}
The proof follows from \eqref{C3} and the decomposition
\[
(G_1^{k_1})_{ii}\cdots (G^{k_l}_l)_{ii}=[((G_1^{k_1})_{ii}-m^{k_1-1}(z_1))+m^{k_1-1}(z_1)]\cdots [((G_1^{k_l})_{ii}-m^{k_l-1}(z_l))+m^{k_l-1}(z_l)]\,.
\]

\subsection{Proof of \eqref{2.14}}
The steps are very similar to those of \eqref{2.12}. We omit the details.

\end{document}